\documentclass{article}
\usepackage{graphicx} % Required for inserting images
\usepackage{fancyhdr}
\usepackage{latexsym}
\usepackage{mathrsfs}
\usepackage{cancel}
\usepackage{color}
\usepackage{multirow}
\usepackage{eso-pic}
\usepackage{dsfont}
\usepackage{amssymb}
\usepackage{graphicx}
\usepackage{setspace}
\usepackage{geometry}
\usepackage{amsthm}
\usepackage{hyperref}
\usepackage[intlimits]{amsmath}
\usepackage{hyperref}
\usepackage{tikz-cd}
\usepackage{amssymb,amsfonts}
\usepackage[all,arc]{xy}
\usepackage{enumerate}
\usepackage{mathrsfs}
\usepackage{tikz}
\usepackage{bbm}
\usepackage{graphicx}
\usepackage{mathtools}
\usepackage{enumitem}
\usetikzlibrary{braids}

\newcommand{\wt}{\widetilde}
\newcommand{\wh}{\widehat}

\newcommand{\End}{\mathrm{End}}

\newcommand{\Hom}{\mathrm{Hom}}

\newcommand{\Mod}{\mathrm{-Mod}}

\newcommand{\Gr}{\mathrm{Gr}}

\newcommand{\bg}{V_{\beta\gamma}}

\newcommand{\FF}{V_{bc}}
\newcommand{\lp}{\left(}
\newcommand{\rp}{\right)}

\newcommand{\Ind}{\mathrm{Ind}}

\newcommand{\QCoh}{\mathrm{QCoh}}

\newcommand{\lpp}{(\!(}
\newcommand{\rpp}{)\!)}
\newcommand{\lbb}{[\![}
\newcommand{\rbb}{]\!]}

\newcommand{\pd}{{\partial}}

\newcommand{\lag}{\langle}
\newcommand{\rag}{\rangle}

\newcommand{\C}{\mathbb C}

\newcommand{\Z}{\mathbb Z}

\newcommand{\N}{\mathbb N}

\newcommand{\PP}{\mathbb P}

\newcommand{\fgl}{\mathfrak{gl}}

\newcommand{\fd}{\mathfrak{d}}
\newcommand{\fg}{\mathfrak{g}}
\newcommand{\fh}{\mathfrak{h}}
\newcommand{\fp}{\mathfrak{p}}

\newcommand{\CA}{{\mathcal A}}

\newcommand{\CD}{{\mathcal D}}

\newcommand{\CJ}{{\mathcal J}}
\newcommand{\CK}{{\mathcal K}}

\newcommand{\CM}{{\mathcal M}}
\newcommand{\CN}{{\mathcal N}}
\newcommand{\CO}{{\mathcal O}}

\newcommand{\CS}{{\mathcal S}}

\newcommand{\CV}{{\mathcal V}}

\newcommand{\CY}{{\mathcal Y}}
\newcommand{\CZ}{{\mathcal Z}}

\newcommand{\be}{\begin{equation}}
\newcommand{\ee}{\end{equation}}
\newcommand{\bsp}{\begin{split}}
\newcommand{\esp}{\end{split}}

\newcommand{\btik}{\begin{tikzcd}}
\newcommand{\etik}{\end{tikzcd}}

\newcommand{\Bun}{\textnormal{Bun}_G}

\newcommand{\conf}{\mathrm{conf}}

\newtheorem{Def}{Definition}[section]
\newtheorem{Thm}[Def]{Theorem}
\newtheorem{Prop}[Def]{Proposition}
\newtheorem{Cor}[Def]{Corollary}
\newtheorem{Lem}[Def]{Lemma}
\newtheorem{Rem}[Def]{Remark}

\newtheorem{Conj}[Def]{Conjecture}

\numberwithin{equation}{section}

\title{Double Yangian, Factorization, and qKZ-equation for Cotangent Lie Algebras}
\author{Raschid Abedin \& Wenjun Niu}
\date{\today}

\begin{document}

\maketitle

\begin{abstract}

    In this paper, we construct the dual \(Y^*_\hbar(\fd)\) and double $DY_\hbar (\fd)$ of the Yangian $Y_\hbar (\fd)$ associated with a cotangent Lie algebra $\fd=T^*\fg$. We define a coherent factorization algebra version of the dual Yangian \(Y^*_\hbar(\fd)^{\textnormal{co-op}}\) with opposite coproduct. Furthermore, we define a quantum vertex algebra structure on the quantum vacuum module $\CV_{\hbar,k}(\fd)$ of central extensions $\widehat{DY}_{\hbar,\ell} (\fd)$ of this double Yangian and show that its conformal blocks satisfy quantum KZ equations. We discuss examples of $\fd$ that arise from 3d $\CN=4$ gauge theories via the work of Costello-Gaiotto. These examples include Takiff Lie algebras $T^*\fg$, whose affine VOA is a large subalgebra of the chiral differential operator algebra of $G$, as well as the smallest type-A Lie superalgebra $\fgl (1|1)$. 
    
\end{abstract}

\tableofcontents

\section{Introduction}

\subsection{Main results}

Let $\fg$ be a finite-dimensional complex Lie algebra and $\fd\coloneqq T^*\fg = \fg \ltimes \fg^*$ be its cotangent Lie algebra. In \cite{ANyangian}, we constructed a quantum group $Y_\hbar(\fd)$, which we called the cotangent Yangian, since it is a quantization of Yang's Lie bialgebra structure on $\fd(\CO) = \fd \otimes \CO$, where \(\CO = \C[\![t]\!]\) is the algebra of formal Taylor power series in a variable \(t\). In \cite{abedin2024quantum}, we also constructed dynamical twists of this quantum algebra locally over the moduli space $\Bun (\Sigma)$ of $G$-bundles over a curve $\Sigma$, which provides quantizations of the classical dynamical  $r$-matrices of \cite{felder_kzb,abedin2024r}. In this paper, we study the factorization structure associated with $Y_\hbar(\fd)$ in the vein of \cite{EK3, EK4, EK5}.

In \cite{ANyangian}, we developed a way to construct a Hopf algebra quantizing a Lie bialgebra structure on the conormal \(N^*\fp_+\) from any splitting of Lie algebras $\mathfrak p=\fp_+\oplus \fp_-$ resembling the monoidal structure of the double quotient ${\exp(\fp_+)\setminus} \exp(\fp)/ \exp(\fp_+)$. We will review this construction in Section \ref{subsec:doublequotient}. Our Yangian $Y_\hbar (\fd)$ is defined as the Hopf algebra associated in this way to the datum $\fp=\fg(\CK)$, $\fp_+=\fg(\CO)\)  and \(\fp_-=t^{-1}\fg[t^{-1}]\). In Section \ref{subsec:dual+double}, we construct the Hopf dual and double of $Y_\hbar (\fd)$ in a similar fashion.\footnote{The construction of the dual and double applies to any Hopf algebra associated to a Lie algebra splitting.}

\begin{Thm}[Section \ref{subsec:dual+double}]\label{thm:main_dualdouble}

The following statements are true. 

    \begin{enumerate}
    
        \item The dual Hopf algebra\footnote{Here and in the following, by \((-)^*\) we mean the topological dual.} $Y_\hbar^*(\fd)$ of \(Y_\hbar(\fd)\) is the Hopf algebra associated to the datum $\fp=\fg (\CK)$, $\fp_+=t^{-1}\fg[t^{-1}]$, and $\fp_-=\fg (\CO)$. 

        \item The double $DY_\hbar (\fd)$ of $Y_\hbar(\fd)$ is the Hopf algebra associated to the datum $\fp=\fg (\CK)\times \fg (\CK)$, $\fp_+=\mathrm{diag}(\fg (\CK)) \coloneqq \{(x,x)\mid x\in \fg(\CK)\}$, and $\fp_-=\fg (\CO)\times t^{-1}\fg[t^{-1}]$. 

        \item The double Yangian $DY_\hbar (\fd)$ is twist equivalent (as a quasi-triangular Hopf algebra) to the double $DU(\fg(\CK))[\![\hbar]\!]$ of \(U(\fg(\CK))[\![\hbar]\!]\).

    \end{enumerate}
\end{Thm}

We also give an explicit formula for the \(R\)-matrix of \(DY_\hbar(\fd)\) in Section \ref{subsec:Rmat_double}.
Furthermore, by converting the pseudotriangular $R$-matrix $R(z)$ of \(Y_\hbar(\fd)\) into a Hopf algebra map $\rho(z)\colon DY_\hbar (\fd)\to Y_\hbar (\fd)\lpp z^{-1}\rpp$, we interpret the algebraic structure on the category $Y_\hbar (\fd)\Mod^{\textnormal{sm}}$ (the category of smooth modules) in terms of a categorical state-operator correspondence. 

\begin{Prop}[Section \ref{subsec:state-op}]
    Restriction along $\rho(z)$ provides a monoidal functor
    \be
\CY(-, z)\colon Y_\hbar(\fd)\Mod^{\textnormal{sm}}\longrightarrow DY_\hbar(\fd)\Mod\longrightarrow \End (Y_\hbar(\fd)\Mod)
    \ee
    such that we have the following functorial isomorphisms
    \be
\CY (M,z)\CY (N, w)P\cong \CY(N, w)\CY (M, z)P,\qquad M,N, P\in Y_\hbar(\fd)\Mod^{\textnormal{sm}}.
    \ee
    
\end{Prop}

In \cite{EK3}, the authors use the $R$-matrix $R(z)$ to demonstrate that the dual Hopf algebra $Y_\hbar^* (\fd)^{\textnormal{co-op}}$ admits a local factorization Hopf algebra structure over $\C = \PP^1\setminus\{\infty\}$. In our case, these local factorization Hopf algebras can be written down very explicitly using an appropriate sheaf of Lie algebras on \(\mathbb{P}^1\); see Section \ref{subsec:EKfac}. In Section \ref{subsec:SheafFac} and Section  \ref{subsec:cohFac}, we construct an honest sheaf factorization structure using this method and the factorization structure of the double quotient stack  $\fg(\CO)\!\setminus\! \fg(\CK)/\fg(\CO)$:

\begin{Thm}[Section \ref{sec:cohfact}]
    There exists a factorization algebra \(\mathscr{Y}^*_\hbar(\fd)^{\textnormal{co-op}}\) on the \(\hbar\)-adic extension of \(\mathbb{P}^1\) such that:
    \begin{enumerate}
        \item The restriction of \(\mathscr{Y}^*_\hbar(\fd)^{\textnormal{co-op}}\) to the configuration space of distinct points is a sheaf of Hopf algebras, whose coalgebra structure is compatible with the factorization structure. 
        
        \item Over distinct points of \(\C = \mathbb{P}^1\setminus\{\infty\}\), the fibers of \(\mathscr{Y}^*_\hbar(\fd)^{\textnormal{co-op}}\) are isomorphic to the locally factorized versions of the dual Yangian \(Y_\hbar^*(\fd)^{\textnormal{co-op}}\) provided by \cite{EK3}.
    \end{enumerate}
\end{Thm}

Finally, we apply the same strategy as in the proof of Theorem \ref{thm:main_dualdouble} to construct central extensions $\widehat{DY}_{\hbar,\ell} (\fd)$ of $DY_\hbar (\fd)$ as well. These quantize the classical central extensions of $\fd(\CK)$ with respect to bilinear forms of the form \(\kappa_\ell = \kappa_0 + \ell \kappa_\fg\) for \(\ell \in \C\), where \(\kappa_0\) is the canonical pairing form of \(\fd\) and \(\kappa_\fg\) is the Killing form of \(\fg\) extended by 0 to \(\fd\). We define the quantum vacuum module $\CV_{\hbar,k}(\fd)$ at level \(k\) of $\widehat{DY_\hbar} (\fd)$ as
\be
\CV_{\hbar,k}(\fd)\coloneqq\Ind_{Y_\hbar(\fd)\otimes \C[ c_\ell]}^{\widehat{DY}_{\hbar,\ell} (\fd)} \C_k\lbb\hbar\rbb,
\ee
where \(\C_k[\![\hbar]\!]\) is the representation on which \(c_\ell\) acts by multiplying with \(k\). We show that:

\begin{Thm}[Theorem \ref{Thm:QVAd} and Theorem \ref{Thm:qKZ}]\label{Thm:QVAintro}
    The following statements are true.
    \begin{enumerate}
        \item The space $\CV_{\hbar,k}(\fd)$ has the structure of a quantum vertex algebra.

        \item Under a restriction of the bilinear form and level, there is a quantum Segal-Sugawara operator $Q$ that acts on modules of $\widehat{DY}_{\hbar,\ell} (\fd)$. The action of $Q$ on quantum conformal blocks gives rise to the following quantum Knizhnik–Zamolodchikov (qKZ) equations
        \be
\nabla_i(\mathbf{z})=R^{i-1, i}(z_{i-1}-z_i+k\hbar) \cdots R^{1,i}(z_1-z_i+k\hbar)\cdot R^{n,i} (z_n-z_i)\cdots R^{i+1, i}(z_{i+1}-z_i),
        \ee
        where \(\mathbf{z} = (z_i)_{i = 1}^n \subset  \mathbb{P}^1\) is a finite set of ordered distinct points on the Riemann sphere.
    \end{enumerate}
\end{Thm}

The second statement in this theorem was proven by mimicking the proof of \cite{EK5}. However, for the first statement, our derivation is independent of \cite{EK5}, since we were not able to reproduce some steps in the calculations omitted there. An advantage of our new proof is that it is quite conceptual and clearly adaptable to general pseudotriangular Hopf algebras that come equipped with a nilpotent derivation, so we essentially show that all such Hopf algebras can be used to construct quantum vertex algebras.

Let us remark that  our constructions are very general and actually apply to DG Lie algebras as well, as long as one carefully follows the Koszul sign rule. Thus, for every DG Lie algebra $\fg$ and $\fd=T^*\fg$, our construction provides a quantum vertex algebra $\CV_{\hbar,k}(\fd)$ quantizing the classical vacuum vertex algebra $V_k(\fd)$ associated to \(\fd\). This is a large new class of quantum vertex algebras. We now introduce a class of $\fd$ that arise from the topological B twist of 3d $\CN=4$ gauge theories via the work of Costello-Gaiotto \cite{costello2019vertex}, which motivated the current study. Their work indicates that $V_k(\fd)$ is closely related to the geometry of Higgs branches, and can be connected to chiral differential operators via 3d mirror symmetry. We thus hope that our results can be used to construct quantizations of these vertex algebras. We now go into more detail about this direction.

\subsection{Motivation: quantum vertex algebras and 3d $\CN=4$ gauge theories}

In \cite{costello2019vertex}, Costello-Gaiotto described holomorphic boundary conditions in topological twists of 3d $\CN=4$ gauge theories. The boundary vertex algebras have now become an important ingredient in the study of 3d mirror symmetry. We consider a 3d $\CN=4$ theory defined by a pair $(G, T^*V)$, consisting of a Lie group \(G\) and the cotangent space of a representation \(V\) of \(G\). The (perturbative) VOA for the B-twist of this gauge theory is identified as the one associated to the affine Lie superalgebra of a finite-dimensional Lie superalgebra. This should receive non-perturbative corrections, leading to the non-perturbative VOA. These are best understood in abelian gauge theories \cite{BCDN23}, which can be used to prove abelian mirror symmetry. 

For the theory defined by $(G, T^*V)$ as above, the Lie superalgebra $\fd_{G, T^*V}$ underlying the VOA is of cotangent type. More precisely, let $\fh\coloneqq \fg\ltimes V[-1]$, where $V$ is of homological degree $1$, then $\fd_{G, T^*V}=T^*[-2]\fh=\fh\ltimes \fh^*[-2]$. This Lie superalgebra was shown to be the tangent Lie algebra of the symplectic reduction $T^*V/\!/\!/\!/ G$ (see \cite{NiuQGSR}), i.e.\ the Higgs branch of the 3d $\CN=4$ theory.  Under a specific choice of the level  \cite{garner2023vertex}, the vacuum module $V_k(\fd_{G, T^*V})$ of the central extension $\wh{\fd}_{G, T^*V}$ is the perturbative VOA of \cite{costello2019vertex}. Our construction, applied to $\fd_{G, T^*V}=T^*[-2]\fh$, results in a quantization $\CV_{\hbar,k}(\fd_{G, T^*V})$ of this perturbative VOA. 

This is consistent with the following intuition. Our study of the Yangian of $\fd_{G, T^*V}$ in \cite{ANyangian} was motivated by 4d $\CN=2$ gauge theories. It is well-known to experts (see for instance \cite{elliott2022taxonomy}) that the $S^1$-reduction of the 4d $\CN=2$ gauge theory defined by $(G, T^*V)$ is the 3d $\CN=4$ theory defined by the same data. The classical VOA $V_k(\fd_{G, T^*V})$ is the (perturbative) boundary VOA of the 3d gauge theory. Intuitively, the quantum VOA $\CV_{\hbar,k}(\fd_{G, T^*V})$ is the boundary VOA of the $S^1$-wrapped 4d gauge theory\footnote{Physically, the difference between the two is whether we shrink the size of $S^1$ to zero. Mathematically, it amounts to the difference between $S^1$ and its affinization $\mathrm{Spec}(\C[\epsilon])$.}, so differential equations satisfied by conformal blocks are replaced by difference equations, and the boundary anomaly cancellation is identical (see Remark \ref{Rem:anomaly}). Naturally, we can hope to lift the known results about $V_k(\fd_{G, T^*V})$ and the 3d B twist into this quantum/difference context. Let us mention some of them here, while a more detailed account will be given in Section \ref{subsec:examples}. 

Consider the case where $G$ is a reductive algebraic group and $V=0$. In this case, the Lie algebra $\fd=T^*\fg$ is also known as the \textit{Takiff Lie algebra}. It turns out that the associated VOA $V_k(\fd)$ can be embedded in $D_{G, k}$, the algebra of chiral differential operators (CDOs) on $G$ at level $k$\footnote{We learned about this, and some subsequent facts about CDOs, from a discussion with S. Nakatsuka.}.  Moreover, $D_{G, k}$ is an extension of $V_k(\fd)$ of the form
\be
D_{G, k}=\bigoplus_{V\in \mathrm{Irred}(G)} \Ind (V)\otimes V^*,
\ee
where $\Ind (V)$ is the induction of $V$ from $\fd(\CO)$ to $\wh{\fd}(\CK)$. We can then hope that there is a similar extension as a quantum vertex algebra for $\CV_{\hbar,k}(\fd)$, which can provide a quantization of the CDOs in $G$. We formulate this conjecture in Section \ref{subsec:examples}. By invoking mirror symmetry, the non-perturbative completion of $V_k(\fd)$ should be related to (more precisely, Morita equivalent to) chiral universal centralizer VOA $I_{G, k}$, we therefore expect a quantum VOA version of this as well. 

For a general gauge theory, it is not known how to compute the perturbative completion of $V_k(\fd_{G, T^*V})$. A proposal based on cohomology over the affine Grassmannian $\Gr_G$ was contained in \cite{dimofte2018dual}. In the case where $G=(\C^\times)^r$ is abelian, this completion is known to be a simple current extension \cite{BCDN23}. The extension has been shown to be related to BRST reductions of beta-gamma VOAs $\bg$ (\textit{a.k.a.}\ CDOs on Higgs branches). For concreteness, let us take $G=\C^\times$ and $V=\C$ with weight $1$. Then $\fd_{G, T^*V}=\fgl (1|1)$ is the simplest type-A Lie superalgebra, and it is well-known \cite{creutzig2013w} that $\bg\otimes \FF$ is a simple current extension of $V_k(\fgl (1|1))$. We therefore conjecture that a similar extension of $\CV_{\hbar,k}(\fgl(1|1))$ exists and provides a quantization of $\bg\otimes \FF$. In general, for abelian $G$, we expect the existence of an extension of $\CV_{\hbar,k}(\fd_{G, T^*V})$ similar to the one studied in \cite{BCDN23}, and it provides a quantization of the corresponding BRST cohomology of many copies of $\bg\otimes \FF$. These conjectures are also formulated more precisely in Section \ref{subsec:examples}.

Unfortunately, we don't know the role or meaning of these quantum chiral differential operator VOAs in physics.

\subsection*{Acknowledgements}
R.A.\ acknowledges the support by the Deutsche Forschungsgemeinschaft (DFG, German Research Foundation) – SFB 1624 – ``Higher structures, moduli spaces and integrability'' – 506632645.

\noindent W.N.\ thanks K.\ Costello and S.\ Nakatsuka for helpful discussions and his friend Don Manual for continued support. W.N.’s research is supported by the Perimeter Institute for Theoretical Physics, which, in turn, is supported in part by the Government of Canada through the Department of Innovation, Science and Economic Development Canada and the Province of Ontario through the Ministry of Colleges and Universities.

\section{Double Yangian for cotangent Lie algebras}\label{sec:double}

In this section, we compute the dual \(Y^*_\hbar(\fd)\) and the double \(DY_\hbar(\fd)\) of the cotangent Yangian $Y_\hbar (\fd)$ as well as \(R\)-matrix of \(DY_\hbar(\fd)\). Furthermore, we construct central extensions of $D Y_\hbar (\fd)$, which resemble the ones for the ordinary Yangians from \cite{EK4}. 

We proceed as follows. In Section \ref{subsec:doublequotient}, we recall the construction of Hopf algebras from double quotients of formal groups from \cite{ANyangian} and explain the intuition behind the computation of the dual and the double of the Hopf algebras constructed in this way. We also recall how the cotangent Yangian \(Y_\hbar(\fd)\) can be constructed from an appropriate double quotient. 
In Section \ref{subsec:dual+double}, we explicitly compute the dual Yangian $Y_\hbar^* (\fd)$ and the double Yangian \(DY_\hbar(\fd)\). Furthermore, we give an explicit formula for the \(R\)-matrix of the double in Section \ref{subsec:Rmat_double} and show in Section \ref{subsec:state-op} that the spectral \(R\)-matrix of the \(Y_\hbar(\fd)\) constructed in \cite{ANyangian} can now be understood in terms of a categorical state-operator correspondence for the category $Y_\hbar (\fd)\Mod$. In Section \ref{subsec:centralext}, we construct the central extensions of $DY_\hbar (\fd)$ using certain bilinear forms on $\fd$.

\subsection{Quantum groups from double quotients}\label{subsec:doublequotient}

\subsubsection{The general idea}
Following our previous work in \cite{ANyangian}, we use double quotients of formal groups associated with Lie algebra decompositions to construct various quantum groups. 
For this purpose, let us briefly repeat this construction and give a strategy of what we will do in Section \ref{sec:double}. In the Section \ref{subsec:dual+double} below, we will make the more sketchy outline of this section explicit in our case of interest.

To any Lie algebra \(\mathfrak{p}\) whose base space is a complex ind-vector space, one can associate a formal group \(\textnormal{exp}(\fp) \coloneqq \widehat{\fp}_0\). A convenient way to think of this formal group in our setting can be achieved in the \(\hbar\)-adic language:\ we extend the setting by \(\C[\![\hbar]\!]\) and scale the maximal ideal of \(0\) by \(\hbar\). Then \(\textnormal{exp}(\fp)\) is generated by \(e^{\hbar x}\) for \(x \in \fp\) and the multiplication is given by the Baker-Campbell-Hausdorff series.

Let now \(\fp_+ \subseteq \fp\) be a subalgebra. Then quasi-coherent sheaves on the double quotient stack \({\textnormal{exp}(\fp_+) \setminus }\textnormal{exp}(\fp) / \textnormal{exp}(\fp_+)\) admit a natural monoidal structure, given by a push-pull along a correspondence diagram induced from the multiplication of \(\textnormal{exp}(\fp)\). This category is usually called the Hecke category. 

If \(\fp_+\) admits a complementary subalgebra \(\fp_- \subseteq \fp\), i.e.\ \(\fp = \fp_+ \oplus \fp_-\), then we can construct a fiber functor for this monoidal structure and use this to represent the aforementioned category of sheaves using a Hopf algebra. The fiber functor is given by the Hecke action of ${\textnormal{exp}(\fp_+)\setminus}\textnormal{exp}(\fp) / \textnormal{exp}(\fp_+)$ on ${\textnormal{exp}(\fp_+) \setminus}\textnormal{exp}(\fp) / \textnormal{exp}(\fp_-)$, where the latter is equal to a single point. 

Algebraically, we may identify \({\textnormal{exp}(\fp_+) \setminus}\textnormal{exp}(\fp) \cong \textnormal{exp}(\fp_-) \cong \textnormal{Spf}(U(\fp_-)^*) = \textnormal{Spf}(S(\hbar\fp_-^*))\) and therefore 
\begin{equation}
    U_\hbar(\fp_+ \ltimes \fp_-^*) = U(\fp_+) \ltimes_{\C[\![\hbar]\!]} S(\fp_-^*)
\end{equation}
is an algebra whose modules are sheaves over the double quotient stack.
The monoidal structure can now be represented by an explicit coproduct \(\Delta_\hbar\) of \(U_\hbar(\fp_+ \ltimes \fp_-^*)\) which gives it the structure of a Hopf algebra.

As the notation suggests, this Hopf algebra is the quantization of a Lie bialgebra, i.e.\ a quantum group. Indeed, the Lie algebra decomposition \(\fp = \fp_+ \oplus \fp_-\) induces a Manin triple \(\fp \ltimes \fp^* = (\fp_+ \ltimes \fp_-^*) \oplus (\fp_- \ltimes \fp_+^*)\). The Hopf algebra \(U_\hbar(\fp_+ \ltimes \fp_-^*)\) quantizes the Lie bialgebra structure on \(\fp_+ \ltimes \fp_-^*\) which is determined by this Manin triple.

Let us remark that the choice of a different complementary subalgebra \(\fp_-\) leads to a twist equivalent Hopf algebra. This is intuitive, since the double quotient does not depend on this choice, so the monoidal structure should also not depend on this choice, but the fiber functor does.

In this paper, we want to extend the picture outlined so far by describing the dual and double of \(U_\hbar(\fp_+\ltimes \fp_-^*)\). This can also be done using the intuition from double quotients as well. In particular, a natural candidate for the dual is the Hopf algebra \(U_\hbar(\fp_- \ltimes \fp_+^*)\) determined by the ``opposite'' double quotient \({\textnormal{exp}(\fp_-) \setminus}\textnormal{exp}(\fp) / \textnormal{exp}(\fp_-)\). Indeed, the right Hecke action of this stack on ${\textnormal{exp}(\fp_+) \setminus }\textnormal{exp}(\fp) / \textnormal{exp}(\fp_-)$ commutes with the left Hecke action of ${\textnormal{exp}(\fp_+) \setminus}\textnormal{exp}(\fp) / \textnormal{exp}(\fp_+)$. Moreover, by adapting the calculations of \cite{BFNloop}, one can show that $\QCoh \lp {\textnormal{exp}(\fp_+) \setminus}\textnormal{exp}(\fp) / \textnormal{exp}(\fp_+)\rp$ and $\QCoh\lp {\textnormal{exp}(\fp_-) \setminus}\textnormal{exp}(\fp) / \textnormal{exp}(\fp_-)\rp$ are dual categories (namely they are commutants) acting on $\mathrm{Vect}=\QCoh \lp{\textnormal{exp}(\fp_+) \setminus}\textnormal{exp}(\fp) / \textnormal{exp}(\fp_-)\rp$. Furthermore, the calculations of \cite{BFNloop} show that the center of these two categories agree, and is isomorphic to the category of quasi-coherent sheaves on $\textnormal{exp}(\fp)/\textnormal{exp}(\fp)$, where the quotient is by the adjoint action. 

Algebraically, we may consider the Hopf algebra \(U_\hbar(\fp\ltimes \fp^*)\) induced by
\begin{equation}
    {\textnormal{exp}(\textnormal{diag}(\fp)) \setminus}\textnormal{exp}(\fp \times \fp) / \textnormal{exp}(\textnormal{diag}(\fp))
\end{equation}
defined by considering the subalgebras \(\textnormal{diag}(\fp) \coloneqq \{(a,a)\mid a \in \fp\}\subset \fp \times \fp\) and \(\fp_+ \times \fp_- \subseteq \fp \times \fp\), which are complementary to each other. We can naturally embed \(U_\hbar(\fp_- \ltimes \fp_+^*)\) and \(U_\hbar(\fp_- \ltimes \fp_+^*)^{\textnormal{co-op}}\) into \(U_\hbar(\fp\ltimes \fp^*)\) and are paired dually inside this algebra. 
In particular, this \(U_\hbar(\fp\ltimes \fp^*)\) becomes that double and \(U_\hbar(\fp_- \ltimes \fp_+^*)\) becomes the dual of \(U_\hbar(\fp_+\ltimes \fp_-^*)\)

In order to describe the structure of the Hopf algebra to \(\fp \times \fp = \textnormal{diag}(\fp) \oplus (\fp_+ \times \fp_-)\) it is convenient to remember that changing the subalgebra complementary to \(\Delta_\fp\) leads to a twist equivalent Hopf algebra. Now \(\textnormal{diag}(\fp)\) has the far simpler complementary Lie algebra \(\fp \times \{0\}\).
The Hopf algebra obtained from this subalgebra coincides with the quantum group from \cite{NiuQGSR} and therefore our double is twist equivalent to this Hopf algebra. In particular, these two Hopf algebras, while not isomorphic, are isomorphic as algebras and have equivalent monoidal categories attached to them.

\subsubsection{The cotangent Yangian}
Let \(\fg\) be a finite-dimensional Lie algebra with basis \(\{I_{\alpha}\}_{\alpha = 1}^d\subseteq \fg\) and \(\{I^{\alpha}\}_{\alpha = 1}^d \subset \fg^*\) be a dual basis of \(\fg\). 
The Yangian \(Y_\hbar(\fd)\) for the cotangent Lie algebra \(\fd = T^*\fg = \fg \ltimes \fg^*\) of \(\fg\) was introduced in \cite{ANyangian} as the quantum group corresponding to the Lie algebra decomposition \(\fg(\CK) = \fg(\CO) \oplus \fg_{<0}\) in the scheme outlined in the Section \ref{subsec:doublequotient}. Here, 
\begin{equation}
    \fg_{<0} \coloneqq t^{-1}\fg[t^{-1}]\,,\,\,\, \fg(\CO) \coloneqq \fg \otimes \CO \subseteq \fg \otimes \CK \eqqcolon\fg(\CK) \quad \textnormal{ for }\quad \C[\![t]\!] = \CO \subset \CK = \C(\!(t)\!).
\end{equation}
In a similar fashion, we define \(\fd_{<0},\fd(\CO) \subseteq \fd(\CK)\) and we write \(a_n = a \otimes t^n \in \fd(\CK)\) for any \(a \in \fd\) and \(n \in \Z\).

More explicitly, the construction of \(Y_\hbar(\fd)\) proceeds as follows:

\begin{enumerate}
    \item Consider the canonical invariant element \(C_\fd = \sum_{a = 1}^d( I^{\alpha} \otimes I_{\alpha} + I_{\alpha} \otimes I^{\alpha}) \in \fd \otimes \fd\), where \(\{I_{\alpha}\}_{\alpha = 1}^d \subseteq \fg, \{I^{\alpha}\}_{\alpha = 1}^d \subseteq \fg^*\) are bases dual to each other. Then the expression \begin{equation}
        \gamma(t_1,t_2) = \frac{C_\fd}{t_1-t_2} \in (\fd \otimes \fd)\left[t_1,t_2,\frac{1}{t_1-t_2}\right]  
    \end{equation}
    satisfies the spectral classical Yang-Baxter equation. Therefore,
    \begin{equation}\label{eq:delta_gamma}
        \delta^\gamma(x) = [x(t_1) \otimes 1 + 1 \otimes x(t_2),\gamma(t_1,t_2)]
    \end{equation}
    defines a topological Lie bialgebra structure \(\delta^\gamma \colon \fd(\CO) \to \fd(\CO) \otimes \fd(\CO)\). This Lie bialgebra structure is defined by the Manin triple \(\fd(\CK) = \fd(\CO) \oplus \fd_{<0}\) and \(Y_\hbar(\fd)\) will turn out to quantize this Lie bialgebra structure.

    \item The splitting \(\fg(\CK) = \fg(\CO) \oplus \fg_{<0}\) gives rise to a right action \(\lhd\) of \(U(\fg_{<0})\) on \(U(\fg(\CO))\) and a left action \(\rhd\) of \(U(\fg(\CO))\) on \(U(\fg_{<0})\). Furthermore, we have a canonical isomorphism \(S(\hbar\fg^*(\CO))[\![\hbar]\!] = S(\hbar(\fg_{<0})^*)[\![\hbar]\!] = U(\fg_{<0})^*[\![\hbar]\!]\). Dualizing \(\rhd\) and rescaling by \(\hbar^{-1}\), we get an action of \(U(\fg(\CO))\) on \(S(\fg(\CO))\). As an algebra
    \begin{equation}
        Y_\hbar(\fd) \coloneqq U(\fg(\CO))[\![\hbar]\!] \ltimes_{\C[\![\hbar]\!]} S(\fg^*(\CO)))[\![\hbar]\!].
    \end{equation}

    \item The comultiplication \(\Delta^\gamma_\hbar\) of \(Y_\hbar(\fd)\) can be defined using the right action \(\lhd\) of \(U(\fg_{<0})\) on \(U(\fg(\CO))\). Namely, 
    \begin{equation}\label{eq:comultiplication_on_g(O)}
        \Delta_\hbar^\gamma(x) = 1 \otimes x + \phi_\lhd(x)\,,\qquad x\in \fg(\CO)
    \end{equation}
    for \(\phi_{\lhd}(x) = \sum_{i \in I}x_i \otimes h_i\) for some index set \(I\) and some \(x_i \in U(\fg(\CO))[\![\hbar]\!]\) and \(h_i \in S(\fg^*(\CO))[\![\hbar]\!]\) satisfying \(\sum_{i \in I}\langle h_i,e^{\hbar y}\rangle x_i = x \lhd e^{\hbar y}\) for all \(y \in \fg_{<0}\). A convenient way to describe \eqref{eq:comultiplication_on_g(O)} is
    \begin{equation}
        \Delta_\hbar^\gamma(x) = \exp(\hbar \gamma_{<})(x \otimes 1+ 1 \otimes x)\exp(-\hbar \gamma_<),
    \end{equation}
    where \(\gamma_{<}\) is the \((\fg \otimes \fg^*)\)-component of the Taylor expansion of \(\gamma = \gamma(t_1,t_2)\) in \(t_1 < t_2\):
    \begin{equation}\label{eq:half_of_gamma}
        \gamma_< = \sum_{n = 0}^\infty\sum_{\alpha = 1}^d I_{\alpha,n} \otimes I^\alpha_{-n-1}.
    \end{equation}
\end{enumerate}

\begin{Rem}
    Let us note that we use the opposite coproduct in this paper compared to \cite{ANyangian}.
\end{Rem}

\begin{Rem}\label{rem:dense_subalgebra}
    Let us note that \(Y_\hbar(\fd)\) admits a dense Hopf subalgebra \(Y_\hbar^\circ(\fd)\), namely the one generated by polynomial elements \(\fd[t] \subset \fd(\CO)\). This can be viewed as the Hopf algebra corresponding to the Lie algebra splitting \(\fd[t,t^{-1}] = \fd[t] \oplus t^{-1}\fd[t^{-1}]\). In some circumstances it is necessary to work with \(Y_\hbar^\circ(\fd)\) instead of \(Y_\hbar(\fd)\). 
\end{Rem}

\subsection{The dual and the double of the cotangent Yangian}\label{subsec:dual+double}

\subsubsection{The dual}

We now want to construct what will become the dual \(Y_\hbar^*(\fd)\) of \(Y_\hbar(\fd)\). Essentially, it is the co-opposite of the Hopf algebra associated to \(\fg(\CK) = \fg_{<0} \oplus \fg(\CO)\), i.e.\ of the quantum group describing the monoidal structure of the double quotient \({\exp(\fg_{<0})\setminus} \exp(\fg(\CK))/\exp(\fg_{<0})\). More precisely:

\begin{enumerate}
    \item The negative of \eqref{eq:delta_gamma} for \(x \in \fd_{<0}\)
    defines the dual Lie bialgebra structure on \(\fd_{<0}\);

    \item We have a canonical isomorphism \(S(\hbar\fg^*_{<0})[\![\hbar]\!] = S(\hbar\fg(\CO)^*)[\![\hbar]\!] = U(\fg(\CO))^*[\![\hbar]\!]\). The \(\hbar^{-1}\)-rescaled dualization of \(\lhd\) defines an action of \(U(\fg_{<0})\) on \(S(\fg_{<0}^*)\) and as an algebra we define the dual Yangian by
    \begin{equation}
        Y^*_\hbar(\fd) \coloneqq  S(\fg^*_{<0})[\![\hbar]\!] \rtimes_{\C[\![\hbar]\!]}U(\fg_{<0})[\![\hbar]\!] .
    \end{equation}

    \item The comultiplication\footnote{The notation is chosen so that \(\Delta^\gamma_\hbar\) is the comultiplication of the whole double Yangian later.} \(\Delta_\hbar^{\gamma,\textnormal{op}}\) on \(Y_\hbar^*(\fd)\) is then defined by taking \(\Delta_\hbar^{\gamma,\textnormal{op}}\) dual to the Baker-Campbell-Hausdorff series on \(S(\fg_{<0}^*)[\![\hbar]\!]\) and define \(\Delta_\hbar^{\gamma,\textnormal{op}}\) on \(U(\fg_{<0})[\![\hbar]\!]\) by
    \begin{equation}\label{eq:comultiplication_on_g(O)_dual}
        \Delta_\hbar^{\gamma,\textnormal{op}} = x \otimes 1 + \phi_{\rhd}(x)\,,\qquad x\in \fg_{<0}
    \end{equation}
    for \(\phi_{\rhd}(x) = \sum_{i \in I}h_i \otimes x_i\) for some index set \(I\) and some \(x_i \in U(\fg_{<0})[\![\hbar]\!]\) and \(h_i \in S((\fg^*)_{<0})[\![\hbar]\!]\) satisfying \(\sum_{i \in I}\langle h_i,e^{\hbar y}\rangle x_i = e^{\hbar y} \rhd x\) for all \(y \in \fg(\CO)\). Another convenient way to describe \eqref{eq:comultiplication_on_g(O)_dual} is
    \begin{equation}
        \Delta_\hbar^{\gamma,\textnormal{op}}(x) = \exp(\hbar \gamma_{<}^{21})(x \otimes 1+ 1 \otimes x)\exp(-\hbar \gamma_<^{21}),
    \end{equation}
    where we wrote \(
        \gamma_<^{21} = \sum_{n = 0}^\infty\sum_{a = 1}^d I_{n}^a \otimes I_{a,-n-1}\) for the coefficient-wise tensor flip.
\end{enumerate}

\begin{Prop}
    The canonical pairing scaled by \(\hbar^{-1}\) defines an isomorphism \begin{equation}
        Y_\hbar^*(\fd) \cong Y_\hbar(\fd)^*.  
    \end{equation}
\end{Prop}
\begin{proof}
    By definition, the multiplication of \(U(\fg(\CO))[\![\hbar]\!]\) and \(U(\fg_{<0})[\![\hbar]\!]\) is dual to the coproduct on \(S(\fg^*_{<0})[\![\hbar]\!]\) and \(S(\fg^*(\CO))[\![\hbar]\!]\), respectively. In particular, the counits and units are in duality.

    Now we are left to check their commutation relation. Let $x \in U(\fg(\CO)), f\in S(\fg^*(\CO)), y\in \fg_{<0}$, and $g\in S (\fg_{<0}^*)\lbb\hbar\rbb$, then
    \be
        \lag x f, yg\rag= \lag \Delta_\hbar^\gamma (xf), y\otimes g\rag
    \ee
    We can assume that $g\in \fg_{<0}^*$ since that generates the entire symmetric algebra, and therefore we can assume that $x\in \fg (\CO)$. Using \eqref{eq:comultiplication_on_g(O)}, we obtain
    \be
        \Delta_\hbar^\gamma (xf)= (1\otimes x+\phi_{\lhd}(x)) \Delta_\hbar^\gamma(f)= \sum_{(f),i \in I}\left(f^{(1)}\otimes xf^{(2)}+x_if^{(1)}\otimes h_i f^{(2)}\right). 
\ee
Pairing this with $y \otimes g$ and using the fact that \(g\) pairs trivially with \(S(\fg^*(\CO))\), we get
\be\label{eq:proof_yangian_dual_difference_1}
\sum_{(f)} \lag f^{(1)},y\rag \lag g,x\rag \epsilon (f^{(2)}) = \lag f,y\rag \lag g,x\rag . 
\ee
On the other hand, if we pair this with $g\otimes y$, we get
\be\label{eq:proof_yangian_dual_difference_2}
\begin{split}
    &\sum_{(f),i \in I} \lag g,x_i\rag \epsilon (f^{(1)}) \langle h_i f^{(2)},y\rangle
    \\&=\sum_{(f),i \in I}\left( \lag g,x_i\rag \epsilon (f^{(1)}) \lag h_i,y\rag \epsilon (f^{(2)})+ \lag g,x_i\rag \epsilon (f^{(1)}) \epsilon (h_i) \lag f^{(2)},y\rag\right)
    \\&=\sum_{i\in I} \lag g,x_i\rag \lag h_i,y\rag \epsilon (f)+ \lag g,x\rag \lag f,y\rag.     
\end{split}
\ee
The difference between \eqref{eq:proof_yangian_dual_difference_1} and \eqref{eq:proof_yangian_dual_difference_2} is given by
\be\label{eq:proof_yangian_dual_difference_3}
    \sum_{i\in I} \lag g,x_i\rag \lag h_i,y\rag \epsilon (f).
\ee
We now claim that this is equal to $\lag xf, [y, g]\rag$ as calculated in $Y_\hbar^* (\fd)$. Indeed, we have
\be
\lag x f, [y, g]\rag= \langle g,x\lhd y\rangle\epsilon (f)= \sum_{i \in I}\lag g,x_i\rag \lag h_i ,y\rag \epsilon (f),
\ee
thanks to the definition of $\phi(x) = \sum_{i \in I}x_i\otimes h_i$. This completes the identification of the algebra structure of $Y_\hbar (\fd)^*$ with the coalgebra structure of $Y_\hbar^* (\fd)$. 

Next, let us identify the coalgebra structure of $Y_\hbar (\fd)^*$ with the algebra structure of $Y_\hbar^* (\fd)$. Therefore, for $\phi_\rhd(y) = \sum_{i \in I}r_i\otimes y_i$, we consider instead
    \be\label{eq:proof_yangian_dual_difference_3}
\begin{split}
    &\langle xf,yg\rangle = \sum_{(g),i\in I}\lag x \otimes f,yg^{(1)}\otimes g^{(2)}+r_ig^{(1)}\otimes y_i g^{(2)}\rag
    =\sum_{(g),i \in I} \langle \Delta_\hbar^\gamma(x),r_i \otimes g^{(1)}\rangle \lag f,y_i\rag \epsilon (g^{(2)})
    \\& =\sum_{(g),i \in I}\left( \langle r_i,x\rangle \epsilon(g^{(1)}) \lag f,y_i\rag \epsilon (g^{(2)}) + \epsilon(r_i)\langle g^{(1)},x\rag \lag f,y_i\rag \epsilon (g^{(2)})\right) 
    \\&=\sum_{i\in I} \lag f,y_i\rag \lag r_i,x\rag \epsilon (g)+ \lag g,x\rag \lag f,y\rag.    
\end{split}
\ee
for \(\phi_{\rhd}(y) = \sum_{i \in I} r_i \otimes y_i\).
On the other hand, we get
\be
        \begin{split}
            \lag fx, yg\rag = \lag f \otimes x,\Delta_\hbar^{\gamma,\textnormal{op}}(y g)\rag = \sum_{(g),i\in I}\lag f \otimes x,yg^{(1)}\otimes g^{(2)}+r_ig^{(1)}\otimes y_i g^{(2)}\rag
            = \lag f,y\rag \lag g, x\rag 
        \end{split}
    \ee
As before, the difference between these is equal to $\lag [x,f], yg\rag$ as calculated in $Y_\hbar (\fd)$, since
\be
\lag [x ,f], y g\rag= -\langle f,x\rhd y\rangle\epsilon (g)= -\sum_{i \in I}\lag f,y_i\rag \lag r_i ,x\rag \epsilon (g),
\ee
thanks to the definition of $\phi_\rhd(y) = \sum_{i \in I}r_i\otimes y_i$. This completes the identification of the algebra structure of $Y_\hbar (\fd)^*$ with $Y_\hbar^* (\fd)$. 
\end{proof}

\subsubsection{The double}

We now want to identify the Drinfeld double \(DY_\hbar(\fd) \coloneqq D(Y_\hbar(\fd))=Y_\hbar (\fd)\otimes Y^*_\hbar (\fd)^{\textnormal{co-op}}\) of \(Y_\hbar(\fd)\) as an algebra. In order to do so, let us recall the quantum group associated to cotangent Lie algebras from \cite{NiuQGSR} and apply it to the cotangent Lie algebra \(\fd(\CK) = T^*\fg(\CK)\). As an algebra, 
\begin{equation}
    U_\hbar(\fd(\CK)) = U(\fg(\CK))[\![\hbar]\!] \ltimes_{\C[\![\hbar]\!]}  S(\fg^*(\CK))[\![\hbar]\!],
\end{equation}
where the action of \(U(\fg(\CK))\) on \(S(\fg^*(\CK))\) is the canonical one induced by \(\fg(\CK)^* = \fg^*(\CK)\).

This algebra becomes a Hopf algebra, quantizing the Lie cobracket on \(\fd(\CK)\) obtained as a double of the trivial Lie bialgebra structure on \(\fg(\CK)\), if we equip it with the usual coproduct on \(U(\fg(\CK))[\![\hbar]\!]\) and the coproduct defined by the Baker-Campbell-Hausdorff series on \(S(\fg^*(\CK))[\![\hbar]\!]\). 

Our goal is to prove the following.

\begin{Thm}
    The Hopf algebras \(U_\hbar(\fd(\CK))\) and \(DY_\hbar(\fd)\) are twist equivalent. In particular, there is a canonical isomorphism \(
        U_\hbar(\fd(\CK)) \cong DY_\hbar(\fd)\)
    of algebras.
\end{Thm}
\begin{proof}
    It is easy to see that \(U_\hbar(\fd(\CK))\) is defined by the Lie algebra decomposition
    \begin{equation}
        \fg(\CK) \times \fg(\CK) = \textnormal{diag}(\fg(\CK)) \oplus (\fg(\CK) \times \{0\})    
    \end{equation}
    in terms of our general construction of quantum groups.
    Here, \(\textnormal{diag} \colon \fg(\CK) \to \fg(\CK) \times \fg(\CK)\) is the diagonal embedding.
    
    We want to recognize that \(D(Y_\hbar(\fd))\) is isomorphic to the Hopf algebra \(H\) defined by the Lie algebra decomposition 
    \begin{equation}\label{eq:decomposition_double}
        \fg(\CK) \times \fg(\CK) = \textnormal{diag}(\fg(\CK)) \oplus (\fg(\CO) \times \fg_{<0}).    
    \end{equation}
    Then the result follows using the twisting procedure from \cite[Section 3.1.3]{ANyangian}.

    Let us first describe \(H\) in more detail. The action of \(U(\fg(\CK))\) on \(U(\fg(\CO))\) and \(U(\fg_{<0})\) combines to an action of \(U(\fg(\CK))\) on \(U(\fg(\CO) \times \fg_{<0}) \cong U(\fg(\CO)) \otimes U(\fg_{<0})\) and consequently on its dual \((U(\fg(\CO)) \otimes U(\fg_{<0}))^*\). Then 
    \begin{equation}
        H = U(\fg(\CK))[\![\hbar]\!] \ltimes_{\C[\![\hbar]\!]} (U(\fg(\CO)) \otimes U(\fg_{<0}))^*[\![\hbar]\!]
    \end{equation}

    Now let us described the double \(DY_\hbar(\fd)\). By definition, it is the bicrossed product
    \begin{equation}
        \begin{split}
            DY_\hbar(\fd) &\coloneqq Y_\hbar(\fd) \bowtie_{\C[\![\hbar]\!]} Y^*_\hbar (\fd)^{\textnormal{co-op}}
            \\&= ((U(\fg(\CO))[\![\hbar]\!] \ltimes_{\C[\![\hbar]\!]} S(\fg^*(\CO))[\![\hbar]\!]) \bowtie_{\C[\![\hbar]\!]} (S(\fg^*_{<0})[\![\hbar]\!] \rtimes_{\C[\![\hbar]\!]} U(\fg_{<0})[\![\hbar]\!]) )
        \end{split}
    \end{equation}
    with the coproduct given by the product of the coproduct of \(Y_\hbar(\fd)\) and the opposite coproduct of \(Y_\hbar^*(\fd)\). Here, the algebra structure of the bicrossed product is determined by
    \begin{equation}
        (x_1 \otimes x_2) (y_1 \otimes y_2) = \sum_{(x_2),(y_1)} x(x_2^{(1)} \rhd y_1^{(1)}) \otimes (y_1^{(2)} \lhd x_2^{(2)})g.
    \end{equation}
    for \(x_1,y_1 \in Y_\hbar(\fd)\) and \(x_2,y_2 \in Y_\hbar^*(\fd)\); see e.g. \cite[Chapter IX]{Kassel1995quantum}.
    
    Now it is clear that the Hopf subalgebra \(S(\fg^*(\CO)) \bowtie S(\fg^*_{<0}) \cong (U(\fg(\CO)) \otimes U(\fg_{<0}))^*\). Therefore, we have an isomorphism of coalgebras \( DY_\hbar(\fd)\to H\), defined by 
    \begin{equation}
        x \otimes f\otimes y \mapsto x\sigma(y) \otimes f,
    \end{equation}
    where \(x \in U(\fg(\CO))[\![\hbar]\!], y\in U(\fg_{<0})[\![\hbar]\!]\), \(f \in S(\fg^*(\CO)) \bowtie S(\fg^*_{<0}) \cong (U(\fg(\CO)) \otimes U(\fg_{<0}))^*\), and \(\sigma \colon U(\fg_{<0}) \to U(\fg_{<0})\) is the natural anti-involution.
    Here, one has to observe that the comultiplication of \(Y^*_\hbar(\fd)\) uses 
    the right action of \(U(\fg_{<0})\) on \(S(\fg^*_{<0})\), so the co-opposite uses the opposite action, consistently with the left action of \(U(\textnormal{diag}(\fg_{<0}))\subseteq U(\textnormal{diag}(\fg(\CK)))\) on \(S(\fg_{<0}^*) \subseteq S(\fg^*(\CO))\bowtie S(\fg^*_{<0})\) used to construct the comultiplication of \(H\).
    
    Now \(U(\fg(\CO)) \bowtie U(\fg_{<0}) \cong U(\fg(\CK))\) holds as algebras, so the coalgebra isomorphism \(DY_\hbar(\fd) \to H\) is also an isomorphism of algebras on \(U(\fg(\CK)), (U(\fg(\CO) \otimes U(\fg_{<0}))^* \subseteq H\). It remains to match up the mixed multiplications. But since these are defined using the action of \(U(\fg(\CK))\) on \(U(\fg(\CO))\) and \(U(\fg_{<0})\), which is equivalently described by the action of \(U(\textnormal{diag}(\fg(\CK))\) on \(U(\fg(\CO) \times \fg_{<0})\), we obtain that the defined map is indeed an isomorphism of Hopf algebras. The only thing to keep in mind is that \(U(\fg_{<0})\subseteq U(\fg(\CK))\) acts from the left on  \(U(\fg(\CO))\) so, after we apply the anti-involution \(\sigma\), we obtain the left action of \(U(\textnormal{diag}(\fg_{<0})) \subseteq U(\textnormal{diag}(\fg(\CK)))\) on \(U(\fg(\CO))\subseteq U(\fg(\CO))\otimes U(\fg_{<0})\).
\end{proof}

\subsubsection{The \(R\)-matrix of \(DY_\hbar(\fd)\)}\label{subsec:Rmat_double}
As the double of a Hopf algebra, \(DY_\hbar(\fd)\) is quasi-triangular, that is, it admits an \(R\)-matrix. We want to derive an explicit formula for this $R$-matrix. It turns out that it is essentially the non-spectral version of the spectral \(R\)-matrix of \(Y_\hbar(\fd)\) which was shown in \cite{ANyangian} to be given by 
\begin{equation}\label{eq:spectral_Rmatrix_of_Yangian}
    R(z) = \exp(\hbar \gamma_{<,\textnormal{sing}}^{21}(z))\exp(-\hbar \gamma_{<,\textnormal{sing}}(-z))\in Y^\circ_\hbar (\fd)\otimes Y^\circ_\hbar (\fd)\lpp z^{-1}\rpp.
\end{equation}
Here, \(\gamma_{<,\textnormal{sing}}(z)\) is the Taylor expansion of \(\frac{\sum_{a = 1}^d I_{\alpha} \otimes I^{\alpha}}{t_1-z-t_2}\) in \(z = \infty\), i.e.\
\begin{equation}\label{eq:gamma_sing}
    \gamma_{<,\textnormal{sing}}(z) = \sum_{i,j = 0}^\infty(-1)^i\binom{i+j}{i}\frac{I_{a,i}\otimes I^a_{j}}{z^{i+j+1}}.    
\end{equation}
The algebra $Y^\circ_\hbar (\fd)$ is the dense Hopf subalgebra of $Y_\hbar (\fd)$ generated by $\fd[t]$. In the following, we denote by $DY^\circ (\fd)$ the dense Hopf subalgebra of $DY_\hbar (\fd)$ generated by $\fd[t,t^{-1}]$. 

\begin{Thm}\label{thm:Rmat_double}
The \(R\)-matrix of \(DY_\hbar(\fd)\) is given by
\be
R=\exp (\hbar\gamma_{<}^{21})\exp (-\hbar \gamma_{<}) \in Y_\hbar(\fd) \widehat \otimes Y^*_\hbar(\fd),
\ee
where \(\gamma_<\) was defined in \eqref{eq:half_of_gamma}. In particular, we have
\begin{enumerate}
    \item \(R\Delta_\hbar^\gamma(x)R^{-1} = \Delta_\hbar^{\gamma,\textnormal{op}}(x)\) for all \(x \in DY_\hbar(\fd)\);

    \item \((\Delta_\hbar^\gamma \otimes 1)R = R^{13}R^{23}\) and \((1 \otimes \Delta_\hbar^\gamma)R = R^{13}R^{12}\);

    \item \(R^{12}R^{13}R^{23} = R^{23}R^{13}R^{12}\).
\end{enumerate}
\end{Thm}

\begin{proof}
For \(x \in \fg(\CK)\) we have
\begin{equation}
    \Delta_\hbar^\gamma(x) = \exp(\hbar\gamma_{<})(x \otimes 1 + 1 \otimes x)\exp(-\hbar\gamma_<),
\end{equation}
so its obvious that \(\Delta_\hbar^{\gamma,\textnormal{op}}(x) = R\Delta_\hbar^\gamma(x)R^{-1}\) holds. Similarly, we obtain the result for \(x \in \fg_{<0}\).

On the other hand, observe that for all \(x,y \in \fg(\CK)\) we have that
\begin{equation}\label{eq:opposite_multiplication_for_exp}
    e^{\hbar y}e^{\hbar x} = e^{\hbar x}e^{-\hbar \textnormal{ad}(x)}e^{\hbar y} = e^{\hbar\textnormal{ad}(y)}e^{\hbar x}e^{\hbar y}.    
\end{equation}
holds. Furthermore, the map
\begin{equation}\label{eq:tensors_to_maps}
    U(\fg(\CK))[\![\hbar]\!] \otimes_{\C[\![\hbar]\!]} S(\fg^*(\CK))[\![\hbar]\!] \to \textnormal{End}(U(\fg(\CK))[\![\hbar]\!]\,,\qquad a \otimes f \to \langle f,-\rangle a 
\end{equation}
maps \(\exp(\hbar\gamma_{<})\) to the natural map \(U(\fg(\CK)) \to U(\fg(\CO))\) of the projection \(a \mapsto a_{\ge 0}\). Similarly, \(\exp(\hbar\gamma_{<}^{21})\) is identified with the natural map \(a \mapsto a_{< 0}\). Using this, as well as the compatibility of maps of the form \eqref{eq:tensors_to_maps} with products (see \cite[Lemma 3.4.]{ANyangian}), we have  
\begin{equation}
    \begin{split}
        \langle R\Delta_\hbar^\gamma(f)R^{-1},e^{\hbar x}\otimes e^{\hbar y}\rangle &= \langle (e^{-\hbar \textnormal{ad}(y_{\ge0})} \otimes e^{\hbar\textnormal{ad}(x_{< 0})})\Delta_\hbar^\gamma(f),e^{\hbar x} \otimes e^{\hbar y}\rangle 
        \\& = \langle \Delta_\hbar^\gamma(f),e^{\hbar \textnormal{ad}(y_{\ge0})}e^{\hbar x} \otimes e^{-\hbar\textnormal{ad}(x_{<0})}e^{\hbar y}\rangle.
    \end{split}
\end{equation}
For \(x,y \in \fg(\CO)\) and \(x,y \in \fg_{<0}\) the last identity is equal to
\begin{equation}
    \langle f,e^{\hbar y}e^{\hbar x}\rangle = \langle\Delta_\hbar^{\gamma,\textnormal{op}}(f),e^{\hbar x} \otimes e^{\hbar y}\rangle
\end{equation}
under consideration of \eqref{eq:opposite_multiplication_for_exp}. Since \(R\Delta_\hbar^\gamma(f)R^{-1} \in S(\fg_{<0}^*)[\![\hbar]\!]\otimes_{\C[\![\hbar]\!]}S(\fg_{<0}^*)[\![\hbar]\!]\) and \(R\Delta_\hbar^\gamma(f)R^{-1} \in S(\fg^*(\CO))[\![\hbar]\!]\otimes_{\C[\![\hbar]\!]}S(\fg^*(\CO))[\![\hbar]\!]\) for \(f \in S(\fg_{<0}^*)[\![\hbar]\!]\)
and \(f \in S(\fg^*(\CO))[\![\hbar]\!]\) respectively, we can deduce that \(R\Delta_\hbar^\gamma(f)R^{-1} = \Delta^{\gamma,\textnormal{op}}_\hbar(f)\) holds in both cases.

The cocycle conditions follow from the respective identities for \(\exp(\hbar \gamma_{<0})\) from \cite[Lemma 3.6.2.]{ANyangian}. The Yang-Baxter equation \emph{3.}\ follows from the cocycle conditions. 
\end{proof}

\subsubsection{Categorical state-operator correspondence}\label{subsec:state-op}
Recall that the category $DY_\hbar (\fd)\Mod$ is the Drinfeld center of the monoidal category $Y_\hbar (\fd)\Mod$. Objects in the Drinfeld center are given by a pair $(M, c)$, where $M$ is a module of $Y_\hbar (\fd)$ and $c$ is a functorial isomorphism
\be
c\colon N\otimes M \stackrel{\cong}\to M\otimes N,\qquad \forall N,
\ee
called a half-braiding. 

Let us also recall that for any smooth object $M\in Y_\hbar (\fd)\Mod$, the meromorphic spectral $R$-matrix provides a functorial isomorphism
\be
c(z)\colon  N\otimes \tau_z (M)\stackrel{\cong}\to \tau_z(M)\otimes N, \qquad \forall N.
\ee
Therefore, the pair $(\tau_z(M), c(z))$ for $z\ne 0$ provides a one-parameter family of objects in the Drinfeld-center, and consequently a one-parameter family of modules of $DY_\hbar (\fd)$. To specify this module, we need to specify the action of $Y^*_\hbar (\fd)^{\textnormal{co-op}}$. The following is clear.

\begin{Lem}

 Let $M$ be a smooth module of $Y_\hbar (\fd)$.   The formula
    \be
x*\tau_z(m)=\langle R(z)(1\otimes m), x\otimes 1\rangle
    \ee
    defines an action of $Y^*_\hbar (\fd)^{\textnormal{co-op}}$ on $M[z,z^{-1}]$. 
    
\end{Lem}

\begin{proof}

    Clearly, this is well defined thanks to the smoothness of $M$. The fact that this defines an algebra action follows from the same argument as in the copseudotriangular structure \cite{EK3}. 
    
\end{proof}

By \cite[Section 7.14]{etingof2016tensor}, the cocycle condition of $R(z)$ implies that the action of $Y^*_\hbar (\fd)^{\textnormal{co-op}}$ and $Y_\hbar^\circ(\fd)$ combine into an action of $DY_\hbar^\circ(\fd)$ on $M[z,z^{-1}]$. Here, \(Y_\hbar^\circ(\fd) \subset Y_\hbar(\fd)\) is the dense Hopf subalgebra generated by \(\fd[t]\) from Remark \ref{rem:dense_subalgebra} and its double \(DY_\hbar^\circ(\fd)\) is the dense Hopf subalgebra of \(DY_\hbar(\fd)\) generated by \(\fd[t,t^{-1}]\). This defines a functor from smooth modules of $Y_\hbar (\fd)$ to modules of $DY_\hbar^\circ (\fd)$. We call this functor $\CY (-, z)$.

\begin{Prop}
    There is a map of Hopf algebras
    \be
DY_\hbar^\circ (\fd)\longrightarrow Y_\hbar (\fd)\lpp z^{-1}\rpp,
    \ee
    such that $\CY (-, z)$ is the restriction functor along the algebra map.
    
\end{Prop}

\begin{proof}

The map $Y^*_\hbar (\fd)^{\textnormal{co-op}}\longrightarrow Y_\hbar (\fd)\lpp z^{-1}\rpp$ is given by
\be
x\mapsto \langle R(z), x\otimes 1\rangle\in Y_\hbar (\fd)\lpp z^{-1}\rpp.
\ee
The map $Y_\hbar^\circ (\fd)\longrightarrow Y_\hbar (\fd)\lpp z^{-1}\rpp$ is given by $\tau_z$. The fact that they combine into a Hopf algebra map from $DY_\hbar^\circ (\fd)$ also follows from the cocycle conditions of $R(z)$. 

\end{proof}

\begin{Rem}
    The above algebra homomorphism can be described explicitly. Namely, on \(\fd[t,t^{-1}] \subseteq DY_\hbar(\fd)\), which generates \(DY_\hbar(\fd)\) topologically, the map is determined by the natural power series expansions
    \begin{equation}\label{eq:double_to_yangian_hom}
        x_n \mapsto \begin{cases}
            (t+z)^nx = \sum_{m = 0}^n\binom{n}{m} t^mz^{m-n}x\in \fd[t,z]\subseteq \fd[t](\!(z^{-1})\!)& \textnormal{ if }n \ge 0;\\
            \sum_{m = -n-1}^\infty \binom{m-n-1}{m}(-1)^m\frac{t^{m}}{z^{m-n}}x \in \fd[t][\![z^{-1}]\!] \subseteq \fd[t](\!(z^{-1})\!)&\textnormal{ if }n<0.
        \end{cases}
    \end{equation}
    Indeed, for \(n\ge 0\) this is simply the definition of \(\tau_z\). For \(n < 0\) on the other hand, this follows from $R (z)=\exp(\hbar\gamma_{<,\textnormal{sing}}^{21}(z))\exp(-\hbar\gamma_{<,\textnormal{sing}}(-z)) \in 1 + \hbar \gamma_{\textnormal{sing}}(z) + \hbar^2 \fd(\CK)^2Y_\hbar(\fd)$ for 
    \begin{equation}\label{eq:gamma_sing}
        \gamma_{\textnormal{sing}}(z) = \sum_{i,j = 0}^\infty(-1)^i\binom{i+j}{i}\frac{I_{a,i}\otimes I^a_{j} + I^a_{i}\otimes I_{a,j}}{z^{i+j+1}}.    
    \end{equation}
\end{Rem}

Let us comment on the structure we have on $Y_\hbar (\fd)\Mod^{\textnormal{sm}}$, the category of smooth modules, i.e.\ those modules where for any element there exists \(n \in \mathbb{N}\) such that \(t^N\fd\) acts trivially for all \(N \ge n\). Combining the functor $\CY (-, z)$ with the obvious functor
\be
DY_\hbar^\circ (\fd)\Mod\longrightarrow \End (Y_\hbar^\circ (\fd)\Mod),
\ee
we obtain a functor (with the same name)
\be
\CY(-, z): Y_\hbar (\fd)\Mod^{\textnormal{sm}}\longrightarrow \End (Y_\hbar^\circ (\fd)\Mod).
\ee
This is a $z$-dependent functor satisfying
\be
\CY (M, z)\CY (N, w)P\cong \CY (\CY (M, z-w)N, w) P,\qquad \CY (M, z)N\cong \tau_z \CY (N, -z)M.
\ee
These are the categorical analogs of OPEs. We see that $Y_\hbar (\fd)\Mod^{\textnormal{sm}}$ is a categorical analog of a vertex algebra, albeit with a regular OPE. Unlike ordinary vertex algebras, in the categorical situation, locality consists of extra data (that of braiding and associativity) rather than a condition. These data are packaged in the statement that $\CY (-, z)$ is a $z$-dependent monoidal functor from $Y_\hbar (\fd)\Mod^{\textnormal{sm}}$ to $\CZ (Y_\hbar (\fd)\Mod^{\textnormal{sm}})$. 

\begin{Rem}
    In the above we used the dense Hopf subalgebra $Y_\hbar^\circ(\fd) \subset Y_\hbar(\fd)$ generated by \(\fd[t]\). We can avoid doing that and replace $M[z,z^{-1}]$ by the formal Laurent serise $M\lpp z\rpp$. This will admit an action of $DY_\hbar (\fd)$ since the image of $Y^*_\hbar (\fd)^{\textnormal{co-op}}\to Y_\hbar(\fd)\lpp z^{-1}\rpp$ belongs to $Y_\hbar(\fd)[z,z^{-1}]$. This makes $\CY(M, z)$ closer to a state-operator correspondence in the world of vertex algebras.  
\end{Rem}

\subsection{Central extension}\label{subsec:centralext}

\subsubsection{Classical central extensions}\label{subsubsec:classicalcentral}

We now want to investigate central extensions of \(\fd(\CK)\) similar to affine Kac-Moody algebras. There are two natural invariant bilinear forms on \(\fd\): the canonical pairing $\kappa_0$ of $\fg$ and $\fg^*$ extended by zero to a bilinear form on \(\fd\) and the Killing form of \(\fd\), which is degenerate and can is twice of the extension by zero of the Killing form \(\kappa_\fg\) of \(\fg\). Therefore, for any \(\ell \in \C\) we can consider the non-degenerate invariant bilinear form $\kappa_\ell\coloneqq \kappa_0+\ell\kappa_\fg$. Using this form, we define the central extension $\wh \fd_\ell=\fd (\CK)\oplus \C c_\ell$ in the usual manner by
\begin{equation}
    [xt^n,yt^m] = [x,y]t^{n+m} + n\delta_{n+m,0}\kappa_\ell(x,y)c_\ell
\end{equation}
for \(x,y \in \fd\) and \(n,m\in\mathbb{Z}\).

Similarly, the Lie algebra $\fd (\CK)$ also admits two natural derivations. We have already used the usual time derivation $\pd_t$. The other derivation, which we denote by $\pd_\epsilon$, is given by
\be
\pd_\epsilon (xt^n)= \pd_t(\psi(x)t^n) = n\psi(x)t^{n-1},\qquad \pd_\epsilon ft^n=0 
\ee
for \(x \in \fg, f\in \fg^*\) and \(n \in \mathbb{Z}\).
Here, $\psi \colon \fg \to \fg^*$ is the natural map induced by the Killing form \(\kappa_\fg\), i.e.\ it is defined by \(\kappa_0(\psi(x),y) = \kappa_\fg(x,y)\) for all \(x,y \in \fg\). Let us note the following.

\begin{Lem}\label{lem:canonical_derivative_partial_epsilon}
    The canonical derivative of \((\fd(\CK),\delta^\gamma)\) is \(2 \partial_\epsilon\). 
\end{Lem}
\begin{proof}
The canonical derivation of the Lie bialgebra \((\fd(\CK),\delta^\gamma)\) is given by
\begin{equation}
    \begin{split}
        ([,]\circ\delta)(xt^n) &= [,] \circ \left[xt^n_1 \otimes 1 + 1 \otimes xt_2^n,\frac{C_\fd}{t_1-t_2}\right] = [,] \circ \left[1 \otimes \frac{xt_2^n-xt_1^n}{t_1-t_2},C_\fd\right]
        \\&=
        \sum_{a = 1}^d\left(\left[I_{\alpha},\left[I^{\alpha},\frac{xt_1^n-xt_2^n}{t_1-t_2}\right]\right] + \left[I^{\alpha},\left[I_{\alpha},\frac{xt_1^n-xt_2^n}{t_1-t_2}\right]\right]\right)\Bigg|_{t_1 = t_2 = t} \\&= \sum_{a = 1}^d(\left[I_{\alpha},\left[I^{\alpha},nxt^{n-1}\right]\right] + \left[I^{\alpha},\left[I_{\alpha},nxt^{n-1}\right]\right])
    \end{split}
\end{equation}
for all \(x \in \fd(\CK)\) and \(n \in \mathbb{Z}\). Now, it remains to observe that
\begin{equation}
    \sum_{a = 1}^d\kappa_0([I_{\alpha},[I^{\alpha},I_b]] + [I^{\alpha},[I_{\alpha},I_b]],I_c) = \sum_{a,a' = 1}^d \left( f_{ac}^{a'}f_{a'b}^a + f_{a'c}^af_{ab}^{a'}\right) = 2\kappa_\fg(I_b,I_c)
\end{equation}
Here, we wrote \([I_{\alpha},I_b] = \sum_{c = 1}^d f_{ab}^cI_c\) and hence \([I^{\alpha},I_b] = \sum_{c = 1}^d f_{bc}^aI^c\) and used the coordinate formula \(\kappa_\fg(I_b,I_c) = \sum_{a,a' = 1}^d f_{ba}^{a'}f_{ca'}^a\) for the Killing form.
%To verify that this is indeed a Lie algebra derivation, we simply compute
%\be
 %[\pd_\epsilon x_{a, n}, x_{b, m}]+[x_{a, n}, \pd_\epsilon x_{b, m}]=n[t_{a, n-1}, x_{b, m}]+m[x_{a, n}, t_{b, m-1}]= (m+n)[t_a, x_b]_{m+n-1},
%\ee
%where we used the fact that $[t_a, x_b]=[x_a, t_b]$, which follows from the identity $\mathrm{Tr}_\fg(x_a[x_b, x_c])=\mathrm{Tr}_\fg(x_b[x_a, x_c])$. Moreover, since $\Tr_\fg (x_a[x_b, x_c])=\Tr_\fg ([x_a, x_b]x_c)$, the last term is equal to $\pd_\epsilon [x_{a, n}, x_{b, m}]$, which is the desired equality.   
\end{proof}

It is not difficult to see that both $\pd_t$ and $\pd_\epsilon$ extend to derivations of the extension $\wh\fd_\ell$ by acting trivially on $c_\ell$. Let us write $\pd_s\coloneqq\pd_t+s\pd_\epsilon$ for \(s \in \C\). We obtain the following two-parameter family of affine Kac-Moody algebras associated to \(\fd\): 
\be
\wt\fd_{\ell, s}\coloneqq \C\pd_s\ltimes \wh \fd_\ell. 
\ee

\begin{Lem}
    The pairing $\kappa_{\ell-s}$ extends to a symmetric non-degenerate invariant bilinear form \(\beta_{\ell-s}\) on $\wt\fd_{\ell, s}$
    by putting \(\beta_{\ell-s}(xt^n,yt^m) = \delta_{n+m,-1} \kappa_{\ell -s}(x,y)\), where \(x,y \in \fd\) and \(m,n \in \Z\), as well as 
    \begin{equation}
        \beta_{\ell-s}(c_\ell,\partial_s) = 1,\beta_{\ell-s}(\partial_s,\partial_s) = 0 = \beta_{\ell-s}(c_\ell,c_\ell), 
    \end{equation}
    and \(\fd(\CK)^\bot = \C \partial_s \oplus \C c_\ell\)
\end{Lem}

\begin{proof}
    Clearly, the given rules give a well-defined non-degenerate symmetric bilinear form on \(\widetilde{\fd}\).
    We only need to check invariance. The only expressions where invariance is not already provided are of the form
    \begin{equation}
        \begin{split}
        &\beta_{\ell-s}([\partial_s,xt^n],yt^{-n}) = \beta_{\ell-s}(\partial_s (xt^n),yt^{-n}) =  n\kappa_{\ell-s}(x + s\psi(x),y) \\&= ns\kappa_0(\psi(x),y) + n(\ell-s)\kappa_\fg(x,y) = n\ell\kappa_\fg(x,y)
        \end{split}
    \end{equation}
    for \(x,y\in\fg\) and \(n\in \Z\). Namely, the consistency for invariance now holds because of
    \begin{equation}
        \begin{split}
        \beta_{\ell-s}(\partial_s,[xt^n,yt^{-n}]) = \beta_{\ell-s}(\partial_s,c_\ell)n\kappa_\ell(x,y) = n\kappa_\ell(x,y)=n\ell\kappa_\fg(x,y).
        \end{split}
    \end{equation}
    This conlcudes the proof.
\end{proof}

The Lie algebra $\wt\fd_{\ell, s}$ admits two complementary Lagrangian Lie subalgebras
\be\label{eq:extlagrangian}
\widetilde\fd_{\ell, s, <0}\coloneqq t^{-1}\fd[t^{-1}]\oplus \C c_\ell, \qquad \widetilde\fd_{\ell,s, \geq 0}\coloneqq \C\pd_s\ltimes \fd\lbb t\rbb,
\ee
and therefore we obtain a Lie bialgebra structure on $\widetilde\fd_{\ell, s}$ for every pair $(\ell, s) \in \C^2$.

\subsubsection{Qantization of the central extension}\label{subsubsec:centquant}

We now consider the problem of quantizing the Lie bialgebra structures from \eqref{eq:extlagrangian}. Unfortunately, our quantization strategy cannot be applied for arbitrary $(\ell,s)$, since the bilinear form $\kappa_{\ell-s}$ is not equal to the original bilinear form on $\fd (\CK)$. However, in the following, we outline a quantization for $\ell=s$, which is sufficient since in the other cases the quantizations are not immediately related to central extensions of the cotangent double Yangian.

Let us write $\wt\fd_{\ell, <0}\coloneqq\wt\fd_{\ell,\ell, <0}$, $\wt\fd_{\ell,\ge 0}\coloneqq\wt\fd_{\ell,\ell, \ge 0}$, and $\wt\fd_{\ell}\coloneqq\wt\fd_{\ell,\ell}$. 
There are two ways one can adjust our construction of $Y_\hbar^*(\fd)$ to the centrally extended picture. One way is similar to the work of \cite{EK4}, where we find a quantization of $\pd_\epsilon$ and define a primitive central extension of $Y_\hbar^*(\fd)$ by hand. We will follow a different method, which again uses our double quotient quantization idea. We will show in the next section that the two provide the same answer. 

Consider the usual affine Kac-Moody algebra
\begin{equation}\label{eq:KMLie}
    \wt\fg \coloneqq\C \pd_t\ltimes \lp \fg (\CK)\oplus \C c_\epsilon\rp    
\end{equation}
whose bracket is given
\be
 [xt^n, yt^m]=[x, y]t^{n+m}+ n\delta_{m+n,0}\kappa_\fg (x, y)c_\epsilon, 
\ee
for all \(x,y\in\fg\) and \(n,m\in\Z\).
Let $T^*\wt\fg = \widetilde\fg \ltimes \widetilde\fg^*$ be the cotangent Lie algebra associated to $\wt\fg$. It is easy to see that $T^*\wt\fg$ is isomorphic to the semi-direct product
    \be
(\C\pd_t\oplus \C\pd_\epsilon)\ltimes (\fd (\CK)\oplus \C c_t\oplus \C c_\epsilon), 
    \ee
    where \(c_t\) and \(\partial_\epsilon\) are dual to \(\partial_t\) and \(c_\epsilon\) respectively and the only new commutation relations are determined 
    \be
 [xt^n, ft^m]=[x, f]t^{m+n}+n\delta_{m+n,0}\kappa_0 (x, f) c_t 
    \ee
for \(x\in\fg\), \(y \in \fg^*\) and \(n,m\in\Z\).

In particular, the Lie algebra $\wt\fd_\ell$ is a quotient of a subalgebra of $T^*\wt\fg$. Namely, we have
\begin{equation}
    \widetilde\fd_\ell \cong \left(\C(\pd_t+\ell\pd_\epsilon)\ltimes (\fd (\CK)\oplus \C c_t\oplus \C c_\epsilon)\right)/\C(\ell c_t - c_\epsilon)
\end{equation}
Observe that \(\C(\ell c_t - c_\epsilon)\) is the kernel of the bilinear form $\kappa_0$ restricted to the subalgebra under consideration. 

From the splitting
\be
\wt\fg= \lp \C \pd_t\ltimes\fg (\CO)\rp \bigoplus \lp \fg_{<0}\oplus \C c_\epsilon\rp, 
\ee
we get a Lagrangian splitting
\be
T^*\wt\fg= \underbrace{\lp \lp \C \pd_t\oplus \C\pd_\epsilon \rp\ltimes\fd (\CO)\rp}_{= N^*\widetilde\fg_{\ge 0}}\bigoplus \underbrace{\lp \fd_{<0}\oplus \C c_\epsilon\oplus \C c_t\rp}_{= N^*\wt\fg_{<0}}. 
\ee
Here, we identified the first factor with the conormal \(N^*\wt\fg_{\ge 0}\) of $\wt\fg _{\ge0}\coloneqq\C \pd_t\ltimes\fg (\CO) $ and the second factor with the conormal \(N^*\wt\fg_{<0}\) of $\wt\fg_{<0}\coloneqq\fg_{<0}\oplus \C c_\epsilon$. The general quantization scheme for Lie algebra decompositions yields quantum groups
\be
U_\hbar (N^*(\wt\fg_{<0})) \quad\textnormal{ and }\quad U_\hbar (N^*(\wt\fg_{\ge0})),
\ee
which are dual to each other, as well as the double \(U_\hbar (T^*\wt\fg)\), which is a quasitriangular Hopf algebra that is twist equivalent to $DU (\wt\fg)\lbb\hbar\rbb$.  

\begin{Lem}
    The pairing with $c_t$ coincides with the linear function on $U(\wt\fg)$ given by
    \be
        U(\wt\fg)=\C[\pd_t]\ltimes U(\wh\fg)\to \C[\pd_t]\to \C
    \ee
    sending $\pd_t$ to $1$ and $\pd_t^n$ to $0$. Therefore, $c_t$ and $ c_\epsilon$ are both central and primitive in $U_\hbar (T^*\wt\fg)$. 
\end{Lem}

In particular, we can consider the quotient $U_\hbar (T^*\wt\fg)/(\ell c_t-c_\epsilon)$, and take the subalgebra generated by $\fd (\CO), c_t$, and $\pd_t+\ell\pd_\epsilon$. Then this Hopf subalgebra is a quantization of $\wt\fd_\ell$. We denote this Hopf algebra by \(\widetilde{DY}_{\hbar,\ell}(\fd)\). Similarly, one has Hopf subalgebras $U_\hbar (\wt\fd_{\ell, <0})$ and $U_\hbar (\wt\fd_{\ell, \geq 0})$ and we write
\begin{equation}
    \widehat{Y}^*_{\hbar,\ell}(\fd) \coloneqq U_\hbar (\wt\fd_{\ell, <0}) \quad\textnormal{ and }\quad \widetilde{Y}_{\hbar,\ell}(\fd) \coloneqq U_\hbar (\wt\fd_{\ell, \ge0}).
\end{equation}
It is not difficult to see that \(
    \widehat{Y}^*_{\hbar,\ell}(\fd) \cong Y_{\hbar}^*(\fd) \otimes \C[c_\ell] \) is a central extension, \(
    \widetilde{Y}_{\hbar,\ell}(\fd) \cong \C[\partial_\ell] \ltimes Y_\hbar(\fd) \), and \(\wt{DY}_{\hbar,\ell}(\fd)\) is the double of \(\widetilde{Y}^*_{\hbar,\ell}(\fd)\). 
    We also write 
    \begin{equation}\label{eq:central_extension_of_double}
        \widehat{DY}_{\hbar,\ell}(\fd) = Y_\hbar(\fd) \otimes_{\C[\![\hbar]\!]}\widehat{Y}_{\hbar,\ell}^*(\fd)^{\textnormal{co-op}} \subset \widetilde{DY}_{\hbar,\ell}(\fd),
    \end{equation}
    which is a subalgebra and central extension of \(DY_\hbar(\fd)\).

\subsubsection{Another quantization and extension of the $R$-matrix}

In \cite{EK4}, another construction of the central extension is based on an explicit formula. Namely, one looks for a primitive quantization of $\pd_\ell$, which we will denote by the same symbol, and define a coproduct on the algebra $Y_\hbar^* (\fd)\otimes \C[c_\ell]$ by declaring that $c_\ell$ is primitive and
\be\label{eq:primitivecop}
\wt\Delta_\hbar^{\gamma, op}(x)=\exp \lp \hbar (c_\ell\otimes \pd_\ell-\pd_\ell\otimes c_\ell)/2\rp \Delta_\hbar^{\gamma, op} (x), \qquad \forall x\in Y_\hbar^* (\fd)
\ee
holds.
To apply this, we first need to find a quantization of $\pd_\ell$. The following is evident.

\begin{Lem}\label{lem:square_of_antipode}
Let $\pd_\epsilon$ be the element in $U_\hbar (N^*\wt \fg_{\geq 0}) = (Y_\hbar^*(\fd) \otimes \C[c_\epsilon,c_t])^*$ defined by
\be
\langle \pd_\epsilon, x c_\epsilon^nc_t^m \rangle=\delta_{m0}\delta_{n1}\epsilon (x), \qquad x\in Y_\hbar^* (\fd)\otimes \C[c_\epsilon, c_t].
\ee
Then $\pd_\epsilon$ is a quantization of the derivation from Section \ref{subsubsec:classicalcentral}, which was also denoted by $\pd_\epsilon$. Moreover, it satisfies \(S^2 = \exp(\hbar \textnormal{ad}(\partial_\epsilon))\).
\end{Lem}
\begin{proof}
    The fact that \(\partial_\epsilon\) is well-defined and quantizes the element of the same name from Section \ref{subsubsec:classicalcentral} is clear. Now \(S^2 = \exp(\hbar \textnormal{ad}(\partial_\epsilon))\) follows from the fact that \(2\textnormal{ad}(\partial_\epsilon)\) is the canonical derivative of \((\widetilde{\fd}_{\ell},\widetilde{\delta}^\gamma)\). Indeed, \(S^2\) is an inner automorphism defined by a group-like element (see \cite{Drinfeld_almost_cocommutative}) that preserves the \(\epsilon\) and \(\hbar\)-grading, so its an exponential of an adjoint action with respect to a degree zero element of symmetric degree one, and \(S^2 = 1 + \hbar \textnormal{ad}(\partial_\epsilon) + \hbar^2(\dots)\) holds since \(2\textnormal{ad}(\partial_\epsilon)\) is the canonical derivative; see \cite{drinfeld1986quantum}.
\end{proof}
Therefore, the derivation $\pd_t+\ell\pd_\epsilon$ provides the quantization of $\pd_\ell$. 

\begin{Cor}
    There is an equivalence between the quantization of Section \ref{subsubsec:centquant} and that of equation \eqref{eq:primitivecop}. 
\end{Cor}

Another consequence of this equivalence is the following. 

\begin{Cor}
    Let $\wt R$ be the $R$-matrix of $\wt{DY}_{\hbar,\ell} (\fd)$, and let $R$ be the $R$-matrix of $DY_{\hbar} (\fd)$ which can be viewed as an element in $\widetilde{D Y}_{\hbar,\ell} (\fd)^{\otimes 2}$. Then
    \be
\wt R=\exp (\hbar \pd_\ell\otimes c_\ell/2)R\exp (\hbar \pd_\ell\otimes c_\ell/2).
    \ee
\end{Cor}

\section{Spectral $R$-matrix and factorization}\label{sec:cohfact}

According to \cite{EK3}, a spectral R-matrix $R(z)$ of a Hopf algebra $H$ with spectral parameter in $\C^\times$ induces a local Hopf algebra factorization structure on $H^*$ over $\PP^1\setminus \{\infty\}$. The cotangent Yangian $Y_\hbar (\fd)$ admits such an R-matrix, and in Section \ref{sec:double} we explicitly constructed the dual and double of $Y_\hbar (\fd)$. In this section, we study the factorization structure on $Y_\hbar^*(\fd)^{\textnormal{co-op}}$ induced by $R(z)$. We show that, in contrast to the dual of Yangians for simple Lie algebras, the dual of the cotangent Yangian \(Y_\hbar^*(\fd)^{\textnormal{co-op}}\) can be made into a coherent factorization algebra over $\PP^1$ in the sense of Beilinson-Drinfeld \cite{beilinson2025chiral}, and this factorization structure is compatible with the coalgebra structure of \(Y_\hbar^*(\fd)^{\textnormal{co-op}}\). We also show that on the configuration space of distinct points, this globalizes the Hopf algebra factorization structure from \cite{EK3}. In this section we focus on factorization over the completion $X=\PP^1$ of the additive algebraic group \(\C\) by a point at infinity, but one could also consider other one-dimensional algebraic groups instead of \(\C\). 

This section is structured as follows. In Section \ref{subsec:EKfac}, we recall the factorization structure of \cite{EK3} associated to a spectral $R$-matrix, and its implication on our centrally-extended dual cotangent Yangian. In Section \ref{subsec:SheafFac}, we construct the same factorization using a sheaf of Lie algebras and its local sections. In Section \ref{subsec:cohFac}, we construct the global coherent version, making this into a genuine factorization algebra with a compatible coproduct.

\subsection{Local factorizations associated to the dual cotangent Yangian}\label{subsec:EKfac}

 As already mentioned in Section \ref{subsec:Rmat_double}, it was shown in \cite{ANyangian} that \(Y_\hbar(\fd)\), or more precisely the dense Hopf subalgebra \(Y_\hbar^\circ(\fd) \subset Y_\hbar(\fd)\) generated by \(\fd[t]\), is pseudotriangular in the sense of \cite{drinfeld1986quantum}, i.e.\ it admits a spectral $R$-matrix \(R(z) \in (Y_\hbar^\circ(\fd) \otimes Y_\hbar^\circ(\fd))(\!(z^{-1})\!)\) such that
\begin{enumerate}
        \item[(T1)] \(
            R(z)  (\tau_z \otimes 1)\Delta_{\hbar}^\gamma(a)= ((\tau_z\otimes 1)\Delta_{\hbar}^{\gamma,\textnormal{op}}(a))R(z)\) for all \(a \in Y_\hbar^\circ(\fd)\);
        
        \item[(T2)] \((\Delta_{\hbar,z_1}^\gamma\otimes 1) R(z_2)=R^{13}(z_1+z_2)R^{23}(z_2)\) and \((1\otimes \Delta^\gamma_{\hbar, z_2})R(z_1+z_2)=R^{13}(z_1+z_2)R^{12}(z_2)\) for \(\Delta^\gamma_{\hbar,z} \coloneq (\tau_z \otimes 1)\Delta_\hbar^\gamma\);

        \item [(T3)] $(\tau_{z_1}\otimes \tau_{z_2})R(z_3)=R(z_3+z_1-z_2)$; 

        \item [(T4)] \((\varepsilon \otimes 1)R(z) = 1 = (1 \otimes \varepsilon)R(z)\)
\end{enumerate}
Here, we wrote \(\tau_z = e^{z\partial_t} \colon Y_\hbar^\circ(\fd) \to Y_\hbar^\circ(\fd)[\![z]\!]\) for the formal shift \(t \mapsto t + z\).
% \begin{equation}\label{eq:spectral_Rmatrix_of_Yangian}
%     R(z) = R_s^{21}(z)R_s(-z)^{-1}\,,\qquad R_s(z) = \exp(\hbar \gamma_{<,\textnormal{sing}}(z)),
% \end{equation}
% where \(\gamma_{<,\textnormal{sing}}(z)\) is the Taylor expansion of \(\frac{\sum_{a = 1}^d I_{\alpha} \otimes I^{\alpha}}{t_1-z-t_2}\) in \(z = \infty\), i.e.\
% \begin{equation}\label{eq:gamma_sing}
%     \gamma_{<,\textnormal{sing}}(z) = \sum_{i,j = 0}^\infty(-1)^i\binom{i+j}{i}\frac{I_{a,i}\otimes I^a_{j}}{z^{i+j+1}}.    
% \end{equation}

The axioms (T1)-(T4) dualize to the following properties of the \(\C[\![\hbar]\!]\)-bilinear form
\begin{equation}\label{eq:copseudotriangular}
    B \colon Y^*_\hbar(\fd) \otimes_{\C[\![\hbar]\!]}Y^*_\hbar(\fd) \to \C(\!(z)\!)[\![\hbar]\!]\,,\qquad B(a \otimes b) \coloneqq \langle a \otimes b, R(z)\rangle
\end{equation}
and for \(a,b,c \in Y_\hbar^*(\fd)\):
\begin{enumerate}
        \item[(C1)] \(
            \nabla_\hbar^\gamma\big((\tau_z \otimes 1)B^{13}(\Delta_\hbar^\gamma(a) \otimes \Delta_\hbar^\gamma(b))\big) = \nabla_\hbar^{\gamma,\textnormal{op}}\big((\tau_z \otimes 1)B^{24}(\Delta_\hbar^\gamma(a) \otimes \Delta_\hbar^\gamma(b))\big)\);
        
        \item[(C2)] \(B(ab,c) = B(a \otimes b,\Delta_\hbar^{\gamma}(c))\) and \(B(a,bc) = B(\Delta_\hbar^{\gamma,\textnormal{op}}(a),b \otimes c)\);

        \item [(C3)] \(B(\tau_{z_1}(a),\tau_{z_2}(b))(z) = B(a,b)(z_1+z-z_2)$; 

        \item [(C4)] \(B(1,a) = \varepsilon(a) = B(a,1)\).
\end{enumerate}
In \cite{EK3}, a Hopf algebra equipped with a pairing \(B\) satisfying (C1)-(C4) is referred to as copseudotriangular. In particular, \(Y_\hbar^*(\fd)^{\textnormal{co-op}}\) is copseudotriangular in this sense. Copseudotriangularity is essentially a Hopf pairing between $Y_\hbar^*(\fd)^{\textnormal{co-op}}$ and $Y_\hbar^*(\fd)$ depending on a spectral parameter $z$. Note that the codomain of $B$ is actually $\C[z,z^{-1}]\lbb\hbar\rbb$ so one this can be evaluated over $z\in \C^\times$. 

In \cite{EK3}, the authors prove that a copseudotriangular structure allows one to define a canonical locally factorized Hopf algebra structure on $Y_\hbar^*(\fd)^{\textnormal{co-op}}$ over finite set of points on a curve (in our case \(\C=X\setminus \{\infty\}\)). More precisely, for a finite set of points $\mathbf z=\{z_i\}_{i\in I}$ indexd by a set \(I\) on $\C$, one can consider the tensor product coalgebra
\be
Y^*_\hbar (\fd)_{\mathbf z}^{\textnormal{co-op}}\coloneqq  Y_\hbar^*(\fd)^{\textnormal{co-op}, \otimes |I|}=\bigotimes_{i\in I} Y_\hbar^*(\fd)_{z_i}^{\textnormal{co-op}},
\ee
and induce an algebra structure on this by declaring that each $Y_\hbar^*(\fd)^{\textnormal{co-op}}$ is a Hopf subalgebra, such that the commutation relation between $Y_\hbar^*(\fd)_{z_i}$ and $Y_\hbar^*(\fd)_{z_j}$ is defined by (C1) with $z=z_i-z_j$. Note that this is well-defined since $B$ is in fact valued in $\C[(z_i-z_j)^\pm]$ rather than formal series. 

More precisely, we can introduce
\begin{equation}
         X_L,X_R \colon Y_\hbar^*(\fd) \otimes_{\C[\![\hbar]\!]} Y_\hbar^*(\fd) \to (Y_\hbar^*(\fd) \otimes_{\C[\![\hbar]\!]} Y_\hbar^*(\fd))(\!(z)\!)
     \end{equation}
     defined by 
     \begin{equation}
        \begin{split}
            &X_L(a \otimes b) \coloneqq B^{13}(\Delta_\hbar^{\gamma}(a) \otimes \Delta_\hbar^{\gamma}(b)) = \langle \Delta_\hbar^{\gamma}(a) \otimes \Delta_\hbar^{\gamma}(b),R(z)^{13}\rangle
            \\ 
            &X_R(a\otimes b) \coloneqq B^{24}(\Delta_\hbar^{\gamma}(a) \otimes \Delta_\hbar^{\gamma}(b)) = \langle\Delta_\hbar^{\gamma}(a) \otimes \Delta_\hbar^{\gamma}(b),R(z)^{24}\rangle.
        \end{split}
     \end{equation}
    Then $Y^*_\hbar (\fd)_{\mathbf z}^{\textnormal{co-op}}$ is defined as the free Hopf algebra generated by $Y_\hbar^*(\fd)_{z_i}^{\textnormal{co-op}}$ with relations given by \(X_{L,ij} = X_{R,ij}^{21}\) for every pair \(i,j \in I\). 

The consistency of these commutation relations with the coproduct is a consequence of the fact that \(B\) is a translation-equivariant Hopf pairing. Observe that the copseudotriangular structure \(B\) is simply the composition of the Hopf algebra map \(Y_\hbar^*(\fd)^{\textnormal{co-op}} \to Y_\hbar(\fd)(\!(z^{-1})\!)\) constructed in Section \ref{subsec:state-op} with the canonical pairing of \(Y_\hbar^*(\fd)\) with \(Y_\hbar(\fd)\). If we use the completed Yangian \(Y_\hbar(\fd)\) rather than the dense subalgebra \(Y_\hbar^\circ(\fd)\), then we can evaluate $z$ and replace the above map by a map $Y_\hbar^*(\fd)_{z}^{\textnormal{co-op}}\to Y_\hbar (\fd)$ by evaluating at $z \in \C^\times$. The commutation relation defined by (C1) agrees with the commutation relation of the image of $Y_\hbar^* (\fd)_z^{\textnormal{co-op}}$ with $Y_\hbar^*(\fd)^{\textnormal{co-op}}$ inside the double $DY_\hbar(\fd)$. In particular, we have an embedding of Hopf algebras $Y_\hbar^* (\fd)_{z, 0}^{\textnormal{co-op}}\to DY_\hbar (\fd)$. The Hopf algebra structure of $Y_\hbar^* (\fd)_{z, 0}^{\textnormal{co-op}}$ and its translation-equivariance  determines $Y_\hbar^*(\fd)_{\mathbf z}^{\textnormal{co-op}}$ for general $\mathbf z$. 

Note that this structure extends to the centrally-extended dual Yangian, since we can extend the pseudotriangular structure to $ \widehat{Y}^*_{\hbar,\ell}(\fd)^{\textnormal{co-op}}$ via the Hopf algebra map $ \widehat{Y}^*_{\hbar,\ell}(\fd)^{\textnormal{co-op}}\to  Y^*_{\hbar}(\fd)^{\textnormal{co-op}}$. We denote by $\wh Y_{\hbar, \ell}^* (\fd)_{\mathbf z}^{\textnormal{co-op}}$ the corresponding local factorization Hopf algebra.

\subsection{Factorization from a sheaf of Lie algebras}\label{subsec:SheafFac}

We now construct the same local factorization Hopf algebra using a sheaf of Lie algebras. More precisely, let $\CA\coloneqq \fg\otimes \CO(-1)$, which has the structure of a sheaf of Lie algebras over $X = \mathbb{P}^1$. The Lie algebra structure comes from the Lie algebra structure on $\fg$ and the (non-unital) algebra structure on $\CO(-1)$. An important property of this sheaf of Lie algebras is, that is has trivial cohomology:
\begin{equation}\label{eq:trivial_cohomology}
    H^0(X, \CA)=0 = H^1(X,\CA).
\end{equation} 
Let $\mathbf{z}=\{z_i\}_{i \in I}$ be a finite set of distinct points in $X$, and let $\mathbb D_{\mathbf z}$ be the formal neighborhood of $\mathbf z$ in $X$ and $\mathbb D_{\mathbf z}^\circ$ the formal punctured neighborhood. We write 
\be
\CA^{\CK}_{\mathbf z}=\Gamma(\mathbb D_{\mathbf z}^\circ, \CA ), ~\CA^{\CO}_{\mathbf z}=\Gamma(\mathbb D_{\mathbf z}, \CA ), ~\CA^{\mathrm{out}}_{\mathbf z}=\Gamma(X\setminus\mathbf{z}, \CA).
\ee
Then \eqref{eq:trivial_cohomology} implies that there is a decomposition of Lie algebras
\be\label{eq:Afiberdecom}
\CA^{\CK}_{\mathbf z}=\CA^{\CO}_{\mathbf z}\oplus \CA^{\mathrm{out}}_{\mathbf z}.
\ee
In case that \(\mathbf{z} = \{0\}\) and if we chose the global coordinate \(t\) on \(X\setminus\{\infty\}\), we see that this is precisely the decomposition $\fg (\CK)=\fg (\CO)\oplus t^{-1}\fg[t^{-1}]$ that determines the cotangent Yangian. On the other hand, if $\mathbf{z}=\{\infty\}$, we obtain the decomposition $\fg(\!(t^{-1})\!)=t^{-1}\fg[\![t^{-1}]\!]\oplus \fg[t]$. 

To proceed with our quantization scheme, we need an appropriate dual of $\CA$. To this end, let $\CA^\vee\coloneqq \fg^*\otimes \CO(-1)$ be the Serre dual of $\CA$. We can similarly define $\CA^{\vee, \heartsuit}_{\mathbf z}$ for $\heartsuit\in \{\CK, \CO, \mathrm{out}\}$. There is a natural topologically perfect pairing:
\be\label{eq:AApair}
\CA^{\CK}_{\mathbf z}\otimes \CA^{\vee, \CK}_{\mathbf z}\to \C, \qquad a\otimes f\mapsto \sum_i \mathrm{Res}_{z_i}\langle a, f\rangle.
\ee
The residue formula implies that under this pairing, we can identify
\be
(\CA^{\CO}_{\mathbf z})^*=\CA^{\vee,\mathrm{out}}_{\mathbf z} \quad \textnormal{ and }\quad  (\CA^{\mathrm{out}}_{\mathbf z})^*=\CA^{\vee, \CO}_{\mathbf z}.
\ee
We can now define the Hopf algebra
\be
\mathscr{Y}^*_\hbar(\fd)_{\mathbf z}^{\textnormal{co-op}}\coloneqq U(\CA^{\vee, \mathrm{out}}_{\mathbf z})\lbb\hbar\rbb\ltimes_{\C\lbb\hbar\rbb} S(\CA^{\vee, \mathrm{out}}_{\mathbf z})\lbb\hbar\rbb,
\ee
whose Hopf structure follows from the decomposition \eqref{eq:Afiberdecom}. We now show that this induces a local factorization structure on $Y_\hbar^*(\fd)^{\textnormal{co-op}}$.

\begin{Prop}

    There is an isomorphism of coalgebras
    \be
\mathscr{Y}^*_\hbar(\fd)_{\mathbf z}^{\textnormal{co-op}}\cong \bigotimes_i \mathscr{Y}^*_\hbar(\fd)_{z_i}^{\textnormal{co-op}},
    \ee
    that makes $\mathscr{Y}^*_\hbar(\fd)_{\mathbf z}^{\textnormal{co-op}}$ into a local factorization Hopf algebra over $X$. 
    
\end{Prop}

\begin{proof}
 There are decompositions
  \be
 \CA^{\CO}_{\mathbf z}=\bigoplus_{i\in I} \CA^\CO_{z_i}, \qquad \CA^{\vee, \mathrm{out}}_{\mathbf z}=\bigoplus_{i\in I} \CA^{\vee, \mathrm{out}}_{z_i},
  \ee
  where the first decomposition is one of Lie algebras and the second one is one of vector spaces. The first follows from \eqref{eq:trivial_cohomology} and the second follows similarly from $H^0(X, \CA^\vee)=0 = H^1(X,\CA^\vee)$. These are compatible with the pairing of equation \eqref{eq:AApair}, and so determines an isomorphism of Hopf algebras
  \be
S(\CA^{\vee, \mathrm{out}}_{\mathbf z})\lbb\hbar\rbb\cong \bigotimes_{i \in I} S(\CA^{\vee, \mathrm{out}}_{z_i})\lbb\hbar\rbb. 
  \ee
Similarly, we have decompositions
\be
\CA^{\vee, \CO}_{\mathbf z}=\bigoplus_{i\in I} \CA^{\vee, \CO}_{z_i},\qquad \CA^{\mathrm{out}}_{\mathbf z}=\bigoplus_{i\in I} \CA^{\mathrm{out}}_{z_i},
\ee
where the second one is \textbf{not} one of Lie algebras, although each individual piece is a subalgebra. This gives a decomposition
\be
U(\CA^{\mathrm{out}}_{\mathbf z})\lbb\hbar\rbb=\bigotimes_{i\in I} U(\CA^{\mathrm{out}}_{z_i})\lbb\hbar\rbb,
\ee
 where each piece is a subalgebra. Furthermore, it is clear that $U(\CA^{\mathrm{out}}_{z_i})\lbb\hbar\rbb$ stabilizes $S(\CA^{\vee, \mathrm{out}}_{z_i})\lbb\hbar\rbb$, since the map $U(\CA^{\CO}_{\mathbf z})\to U(\CA^\CO_{z_i})$ is a map of $\CA^{\mathrm{out}}_{z_i}$-modules. This implies that the embedding $S(\CA^{\vee, \mathrm{out}}_{z_i})\lbb\hbar\rbb\to S(\CA^{\vee, \mathrm{out}}_{\mathbf z})\lbb\hbar\rbb$ respects the $\CA^{\mathrm{out}}_{z_i}$-algebra structures. This gives us the desired isomorphism as vector spaces. Moreover, each $\mathscr{Y}^*_\hbar(\fd)_{z_i}^{\textnormal{co-op}}$ is a subalgebra. 

 We are left to show that this is an identification of coalgebras. It is clear for $S(\CA^{\vee, \mathrm{out}}_{\mathbf z})\lbb\hbar\rbb$, so we are left with the coproduct of $U(\CA^{\mathrm{out}}_{\mathbf z})\lbb\hbar\rbb$. Choosing a primitive element $x\in \CA^{\mathrm{out}}_{z_i}$, we know that the coproduct is defined so that 
 \be
\Delta_\hbar^{\gamma,\textnormal{op}}(x)=1\otimes x+\phi_{\rhd}(x)^{\textnormal{op}}, \qquad \phi_{\rhd}(x)\in \CA^{\mathrm{out}}_{\mathbf z}\lbb\hbar\rbb\otimes_{\C[\![\hbar]\!]} S(\CA^{\vee, \mathrm{out}}_{\mathbf z})\lbb\hbar\rbb.
 \ee
Here $\phi_{\rhd}(x)^{\textnormal{op}}$ satisfies $\langle \phi_{\rhd}(x)^{\textnormal{op}}, e^{\hbar y}\rangle=e^{\hbar y}\rhd x$ for all $y\in \CA^\CO_{\mathbf z}$. However, when $y\in \CA^\CO_{\mathbf z_j}$ for $j\neq i$, $e^{\hbar y}\rhd x=x$ since the commutators in the factor $z_j$ belong to $U(\CA_{z_j}^{\CO})$. Therefore, $\Delta(x)$ belongs to $\mathscr{Y}^*_\hbar(\fd)_{z_i}^{\textnormal{co-op}}$ and is identified with the coproduct in this subalgebra. 
\end{proof}

Using the translations \(\tau_z\) in $\C=X\setminus\{\infty\}$, one can show that $\mathscr{Y}^*_\hbar(\fd)_{z}^{\textnormal{co-op}}$ are isomorphic to $Y^*_{\hbar}(\fd)^{\textnormal{co-op}}$ for all $z\ne \infty$, whereas $\mathscr{Y}^*_\hbar(\fd)_{\infty}^{\textnormal{co-op}}=Y_\hbar^\circ (\fd)^{\textnormal{op}}$. Therefore, one can deduce that $\mathscr{Y}^*_\hbar(\fd)_{0, \infty}^{\textnormal{co-op}}$ is nothing but the dense subset of the double of $Y^*_{\hbar}(\fd)^{\textnormal{co-op}}$ generated by \(\fd[z,z^{-1}]\); see equation \eqref{eq:decomposition_double}.

We now identify this factorization with the one induced by the pseudotriangular structure. 

\begin{Thm}\label{Thm:IdenFac}
    The local factorization Hopf algebra $\mathscr{Y}^*_\hbar(\fd)_{\mathbf z}^{\textnormal{co-op}}$ agrees with $Y_\hbar^* (\fd)_{\mathbf z}^{\textnormal{co-op}}$ defined in Section \ref{subsec:EKfac}. 
\end{Thm}

\begin{proof}
    We need to show that commutators of $\mathscr{Y}^*_\hbar(\fd)_{z_i}^{\textnormal{co-op}}$ and $\mathscr{Y}^*_\hbar(\fd)_{z_j}^{\textnormal{co-op}}$ are induced by $R(z_i-z_j)$. For simplicity, we may assume that $\mathbf z=\{0, z\}$ for some $z\in X\setminus \{\infty\}$.
    
    Let us consider the commutator of $\CA_{0}^{\mathrm{out}}$ and $\CA^{\mathrm{out}}_z$ inside $\CA^{\mathrm{out}}_{0, z}$. We have commutative diagram of Lie algebras
    \be
\btik
\CA^{\mathrm{out}}_z\rar \dar& \CA^{\mathrm{out}}_{0, z}\dar& \lar \CA^{\mathrm{out}}_{0}\dar\\
\CA^\CO_0\rar & \CA^\CK_0 &\lar \CA^{\mathrm{out}}_0
\etik
    \ee
where the downward arrows are embeddings. Therefore, the commutation relations of $\CA_{0}^{\mathrm{out}}$ and $\CA^{\mathrm{out}}_z$ can be studied inside $\CA^\CK_0$. The third downarrow above is the identity, and the first downarrow is given by Taylor expansion at $0$. Comparing this with equation \eqref{eq:double_to_yangian_hom}, this Taylor expansion is exactly the pairing with $R(z)$. Therefore, the commutation relation between $\CA^{\mathrm{out}}_z$ and $\CA^{\mathrm{out}}_{0}$ agrees with the one coming from expanding $\CA^{\mathrm{out}}_z$ to $\CA^\CO_0$ via $R(z)$ then computing their commutation relation in $DY_\hbar^*(\fd)$, which agrees with the local factorization structure induced by $R(z)$. 

Now consider the commutation relation between $\CA_z^{\mathrm{out}}$ and $S(\CA_0^{\vee, \mathrm{out}})$. Let $x\in \CA_z^{\mathrm{out}}$ and $f\in S(\CA_0^{\vee, \mathrm{out}})$. Then the commutation relation of these two generators in $Y_\hbar^*(\fd)^{\textnormal{co-op}}_{0, z}$ reads:
\be
[x,f]=\sum_{(f)}\left(f^{(1)}B(f^{(2)}, x)+\sum_{j\in J}B(x_j,S(f^{(1)}))f^{(2)}h_j\right),
\ee
where we wrote $\Delta (x)=1\otimes x+\sum_{j \in J} x_j\otimes h_j$ for some index set \(J\) and elements $x_i \in \CA_z^{\textnormal{out}}$, $h_j\in S(\CA_z^{\vee, \mathrm{out}})$. In particular, for $a\in \CA_0^\CO$ and $b\in \CA_z^\CO$, we have
\be
\langle [x,f],e^{\hbar b} \otimes e^{\hbar a}\rangle=\sum_{(f)}\left(\langle f^{(1)},e^{\hbar a}\rangle B(f^{(2)}, x)+\sum_{j \in J}B(x_j, S(f^{(1)}))\langle f^{(2)},e^{\hbar a}\rangle \langle h_j,e^{\hbar b}\rangle\right).
\ee
Recall that $x\vert_{\mathbb{D}_0} \coloneqq \langle 1 \otimes x,R(z)\rangle$ is the Taylor expansion of \(x\) at $0$, and using this notation the above is equal to
\be\label{eq:EK_commutator}
\langle f,e^{\hbar a}x\vert_{\mathbb{D}_0}\rangle -\langle f,(e^{\hbar b})\rhd x\vert_{\mathbb{D}_0} e^{\hbar a}\rangle.
\ee
On the other hand, in $\mathscr{Y}_\hbar^*(\fd)^{\textnormal{co-op}}_{0,z}$, we have the following commutation relation
\be
\langle [x,f],e^{\hbar b}e^{\hbar a}\rangle=\langle f,e^{\hbar b}e^{\hbar a} \lhd_{0,z} x\rangle=\langle f, e^{\hbar b}[e^{\hbar a}, x])\rangle+\langle f,e^{\hbar b}\lhd_{0,z} x e^{\hbar a}\rangle, 
\ee
where $\lhd_{0,z}$ is computed using the decomposition $\CA_{0,z}^\CK=\CA_{0,z}^\CO\oplus \CA_{0,z}^{\mathrm{out}}$, whereas $\lhd$ above is computed using $\CA_z^\CK=\CA_z^\CO\oplus \CA_z^{\mathrm{out}}$. We look at the second term. Note that the action of $x$ on $e^{\hbar b}$ is given by
\be
e^{\hbar b}\lhd_{0,z} x=x\vert_{\mathbb{D}_0}\otimes e^{\hbar b}+1\otimes e^{\hbar b}x\vert_{\mathbb{D}_z}\in {U(\CA_{0,z}^{\mathrm{out}})\setminus} (U(\CA_{0}^\CK)\otimes U(\CA_z^\CK)) \cong U(\CA_{0}^\CO)\otimes U(\CA_z^\CO), 
\ee
where \(x\vert_{\mathbb{D}_z}\) is the Laurent expansion of \(x\) in \(z\).
Now we have $e^{\hbar b}x\vert_{\mathbb{D}_z}=(e^{\hbar b}\rhd x\vert_{\mathbb{D}_z}) e^{\hbar b}+e^{\hbar b}\lhd x\vert_{\mathbb{D}_z}$. When we pair with $f \in S(\CA_0^{\vee,\textnormal{out}})$, the term $e^{\hbar Y}\lhd x\vert_{\mathbb{D}_z}$ does not contribute since it belongs to $U(\CA_z^\CO)$. Therefore, we are left with
\be
x\vert_{\mathbb{D}_0}\otimes e^{\hbar Y}+(e^{\hbar Y}\rhd x\vert_{\mathbb{D}_z}) e^{\hbar Y}.
\ee
Consequently,
\be\label{eq:geometric_commutator}
\langle [x,f],e^{\hbar b}e^{\hbar a}\rangle= \langle f,[e^{\hbar a}, x\vert_{\mathbb{D}_0}]\rangle +\langle f,x\vert_{\mathbb{D}_0}e^{\hbar a}\rangle +\langle f,(e^{\hbar b}\rhd x\vert_{\mathbb{D}_z})e^{\hbar a}\rangle.
\ee
The first two terms give us $\langle f,e^{\hbar X}x\vert_{\mathbb{D}_0}\rangle$, and we claim that the second is equal to $-\langle f,e^{\hbar Y}\rhd x\vert_{\mathbb{D}_z} e^{\hbar X}\rangle $. Indeed, this follows from
\be
e^{\hbar b}\rhd x=e^{\hbar b}\rhd x\vert_{\mathbb{D}_0}+e^{\hbar b}\rhd x\vert_{\mathbb{D}_z}\in \CA_z^{\mathrm{out}}\subseteq \CA_{0,z}^{\mathrm{out}}
\ee
and this element vanishes when paired with $f$ by definition. This shows that the two commutators of \(Y_\hbar^*(\fd)_{0,z}\) as constructed by \cite{EK3} from \eqref{eq:EK_commutator} and the geometric one from \eqref{eq:geometric_commutator} agree. Therefore, we have completed the proof that the local Hopf algebra factorization structure on $\mathscr{Y}^*_\hbar(\fd)_{\mathbf z}^{\textnormal{co-op}}$ coincides with the one induced by $R (z)$. 
\end{proof}

We can define the central extension $\widehat{\mathscr{Y}}_{\hbar,\ell}^*(\fd)_{\mathbf z}^{\textnormal{co-op}}$ as well by applying equation  \eqref{eq:primitivecop}. We will not go into detail here.

\subsection{Coherent Factorization}\label{subsec:cohFac}

Having established in Theorem \ref{Thm:IdenFac} that $\mathscr{Y}_\hbar^*(\fd)_{\mathbf z}^{\textnormal{co-op}}$ equals $Y_\hbar ^*(\fd)_{\mathbf z}^{\textnormal{co-op}}$, we are ready to sheafify the construction following \cite{BDchiral}, making this into a factorization algebra. This is not possible in the case of Yangians for simple Lie algebras, as those do not have a vertex algebra structure compatible with the coproduct. The cotangent Yangian does admit such a structure (\cite[Proposition 4.1]{ANyangian}), and it makes sense that an actual factorization structure exists. We still focus on $X=\PP^1$ case in what follows. 

\subsubsection{Factorization algebra}

Let $I$ be a finite set, and let $X^I$ be the $|I|$-th power of $X$. The space $X^I\times X$ admits two projections $p_I\colon X^I\times X\to X^I$ and $q_I\colon X^I\times X\to X$. Denote by $\Gamma_I\subseteq X^I\times X$ and $U_I\subseteq X^I\times X$ the subschemes defined by
\be
\Gamma_I\coloneqq\Big \{(x_i, x)\Big\vert x=x_i \text{ for some } i \in I\Big \}, \qquad U_I\coloneqq(X^I\times X)\setminus \Gamma_I. 
\ee
Denote by $\wh\Gamma_I$ the formal neighborhood of $\Gamma_I$ inside $X^I\times X$ and by $\wh\Gamma_I^\circ$ the formal punctured neighborhood $\wh\Gamma_I\setminus \Gamma_I$. Furthermore, let $p_I^\circ, q_I^\circ$ be the restrictions of the projections \(p_I,q_I\) to $U_I$, $\wh p_I,\wh q_I$ be their restrictions to $\wh\Gamma_I$, and $\wh p_I^\circ,\wh q_I^\circ$ be their further restriction to $\wh\Gamma_I^\circ$. If $\CJ_I$ denotes the ideal that defines $\Gamma_I$, the formal scheme $\wh\Gamma_I$ is defined by
\be
\wh\Gamma_I=\mathrm{Spf}\lp \varprojlim_n \CO_{X^I\times X}/\CJ_I^n\rp. 
\ee
From this, we define the following three sheaves of Lie algebras on $X^I$:
\be
\CA_I^\CO\coloneqq \hat p_{I,*}\hat q_I^* (\CA), \qquad \CA_I^\CK\coloneqq \hat p_{I,*}^{\circ}\hat q_I^{\circ,*} (\CA), \qquad \CA_I^{\textnormal{out}}\coloneqq p_{I,*}^\circ q_I^{\circ,*} (\CA). 
\ee
Here, $\hat q_I^* \CA \coloneqq \varprojlim_n q_I^*\CA/ \CJ_I^n$ and $\hat q_I^{\circ,*}$ is defined similarly. These are sheaves of $\CO$-modules on $X^I$, and in fact are sheaves of Lie algebras in the category of $\mathcal{D}$-modules on $X^I$. The sheaves $\CA_I^\CO$ and $\CA_I^\CK$ are not quasi-coherent, whereas $\CA_I^{\textnormal{out}}$ is. These sheaves carry factorization structures, in the sense that for a surjection $\pi\colon J\twoheadrightarrow I$ there are obvious isomorphisms
\be\label{eq:factLie}
\Delta^{J/I,*}\CA_J^\heartsuit\cong \CA_I^\heartsuit, \qquad j^{J/I,*}\lp\prod_{i\in I} \CA_{J_i}^\heartsuit\rp\cong j^{J/I,*}\lp \CA_J^\heartsuit\rp, \qquad \heartsuit=\{\CO, \CK, \textnormal{out}\},
\ee
where \(J_i \coloneqq \pi^{-1}(\{i\})\) for every \(i \in I\), \(\Delta^{J/J}\colon X^I \to X^J\) is defined by \((x_i)_{i\in I} \mapsto (x_{\pi(j)})_{j \in J}\) and \(j^{J/I} \colon U^{J/I} \coloneqq X^J\setminus \Delta^{J/I}(X^I) \to X^J\) is the open embedding.
Note that this \textbf{does not} make $\CA_I^\heartsuit$ into factorization algebras. For $\heartsuit \in \{\CO, \CK\}$, these isomorphisms are naturally isomorphisms of Lie algebras, whereas for $\heartsuit=\textnormal{out}$, they are not. Nonetheless, $\CA_{J_i}^{\textnormal{out}}$ are Lie subalgebras of $j^{J/I,*} \CA_J^\heartsuit$. The fiber of $\CA_I^{\heartsuit}$ at $\mathbf z$ is precisely $\CA_z^\heartsuit$ constructed in the previous section. Similarly, we can define $\CA_I^{\vee, \heartsuit}$. 

We need the following lemma.

\begin{Lem}\label{Lem:CAflat}
    The sheaves $\CA_I^{\heartsuit}$ and $\CA_I^{\vee, \heartsuit}$ are flat on $X^I$ for and finite set \(I\) and \(\heartsuit \in \{\CO,\CK,\textnormal{out}\}\).
\end{Lem}

\begin{proof}
We prove this for $\heartsuit=\textnormal{out}$ since the proofs for other cases are similar. Let $(x_i)_{i \in I}\in X^I$ be a point. Assume for a moment that $x_i\ne \infty$ for all $i \in I$. Let $U_\infty$ be $X\setminus \{0\}$, and $U=X\setminus \{\infty\}$, which is a neighborhood of $x_i$ for all \(i \in I\). In particular, $U^I$ is a neighborhood of $(x_i)_{i \in I} \in X^I$. Let us compute $\Gamma(U^I,\CA_I^{\textnormal{out}})$. Choose a global coordinate $t$ of $U$, and write $z=t^{-1}$ of $U_\infty$. We can cover $(U^I\times X)\setminus \Gamma_I$ by the open subsets
    \be
    U_1\coloneqq (U^I\times U_\infty) \setminus \Gamma_I\quad \textnormal{ and }\quad U_2\coloneqq (U^I\times U)\setminus \Gamma_I.
    \ee
    We note that 
    \begin{equation}
        U_1 \cap U_2 = (U^I \times (U\cap U_\infty))\setminus\Gamma_I
    \end{equation}
    We can compute $\Gamma(U^I,\CA_I^{\textnormal{out}})$ using the \v Cech complex associated to the above covering:
    \be\label{eq:cech}
\btik
\fg [(t_i)_{i \in I}, z] \left[\prod_{i\in I} \frac{1}{1-t_i z}\right] \arrow[dr, swap, "z\cdot -"]&\bigoplus & \fg[(t_i)_{i \in I}, t]\left[\prod_{i\in I} \frac{1}{t-t_i}\right]\arrow[dl, "t\mapsto z^{-1}"]\\
 &\fg[(t_i)_{i \in I}, z, z^{-1}]\left[\prod_{i\in I} \frac{1}{1-t_i z}\right] &
\etik
    \ee
    It is clear now that the cohomology of $\Gamma(U^I,\CA_I^{\textnormal{out}})$ is concentrated in degree $0$, as the map going down is surjective. The kernel consists of all pairs 
    \begin{equation}
        A\in \fg[(t_i)_{i \in I}, t]\left[\prod_{i\in I} \frac{1}{t-t_i}\right] \quad \textnormal{ and }\quad B\in \fg[(t_i)_{i \in I}, z]\left[\prod_{i\in I} \frac{1}{1-zt_i}\right]
    \end{equation}
    such that $zB+A=0$. In particular, we have $\lim_{t\to \infty} A(t)=0$ and the subspace in which this is satisfied is given by
    \be\label{eq:cechcoh}
        K \coloneqq \sum_{(k_i)_{i \in I} \in \mathbb{N}_0^I}\sum_{n< \sum_{i \in I}k_i}\fg[(t_i)_{i \in I}] t^n  \prod_{i \in I}\lp \frac{1}{t-t_i}\rp^{k_i}. 
    \ee
It is not difficult to show that for any element $A \in K$, one has
\begin{equation}
    B((t_i)_{i \in I},z) \coloneqq -z^{-1}A((t_i)_{i \in I},z^{-1}) \in \fg[(t_i)_{i \in I}, z]\left[\prod_{i\in I} \frac{1}{1-zt_i}\right].
\end{equation}
Therefore, the cohomology of \eqref{eq:cech} is isomorphic to \(K\). As $\C[(t_i)_{i \in I}]$-module, \(K\) is free with basis given by 
\be\label{eq:basis_of_K}
\left\{I_{\alpha,n, k}\coloneqq I_\alpha t^k \prod_{i\in I} \lp \frac{1}{t-t_i}\rp^n \Bigg| \,\,\, 0\le \alpha \le d, 0\leq k<|I|, ~n>0\right\}.
\ee
Indeed, that this set is a basis follows from degree considerations and the simple fact that, for \(r < |I|\), \(t^{s|I| + r} =  \sum_{j = 0}^s(\Pi_{i \in I}(t-t_i))^ja_j(t,(t_i)_{i \in I})\) holds for some polynomials \((a_j)_{j = 0}^s \subset \C[t,(t_i)_{i \in I}]\) of \(t\)-degree less then \(|I|\).

More generally, we can cover \(X^I\) with open subsets of the form \(U^{I_1} \times U_\infty^{I_2}\) for \(I = I_1 \sqcup I_2\) and then \(\Gamma(U^{I_1} \times  U_\infty^{I_2},\mathcal{A}_I^{\textnormal{out}})\) is isomorphic to the space of pairs 
\begin{equation}
        {\begin{split}
        &A\in \fg[(t_i)_{i \in I_1},(z_i)_{i \in I_2}, t]\left[\prod_{i\in I_1} \frac{1}{t-t_i}\prod_{i\in I_2} \frac{1}{z_it-1}\right]\\ &B\in \fg[(t_i)_{i \in I_1},(z_i)_{i \in I_2}, z]\left[\prod_{i\in I_1} \frac{1}{1-zt_i}\prod_{i\in I_2} \frac{1}{z_i-z}\right]
        \end{split}}
        \quad \textnormal{ such that }zB + A = 0.
\end{equation}
Again, we can describe the space of these elements explicitly as
\be
K \coloneqq \sum_{(k_i)_{i \in I} \in \mathbb{N}_0^{I}}\sum_{n< \sum_{i \in I}k_i} \fg[(t_i)_{i\in I_1},(z_i)_{i \in I_2}] t^n \prod_{i \in I_1}\lp \frac{1}{t-t_i}\rp^{k_{i}}\prod_{i \in I_2}\left(\frac{1}{1-z_it}\right)^{k_{i}}. 
\ee
This is unfortunately not a free module with a basis of the form \eqref{eq:basis_of_K} anymore, since the \(t^{s|I| +r}\) cannot be expressed in terms of powers of \(\prod_{i \in I_1}(t-t_i)\prod_{i \in I_2}(z_it-1)\) anymore. However, it is still flat since it is an injective limit of free modules.
Indeed, consider the \(\C[(t_i)_{i \in I_1},(z_i)_{i \in I_2}]\)-span \(K_m\) of
\be\label{eq:basis_of_K}
\left\{I_{\alpha,m, k}\coloneqq I_\alpha t^k \lp\prod_{i\in I_1}  \frac{1}{t-t_i}\prod_{i\in I_2}  \frac{1}{z_it-1}\rp^m \Bigg| \,\,\, 0\le \alpha \le d, 0\leq k<|I| m\right\}.
\ee
Then these spaces are free with the given elements as basis.
Moreover, it is clear that $\varinjlim_m K_m =\Gamma(U^{I_1} \times  U_\infty^{I_2},\mathcal{A}_I^{\textnormal{out}})$. 
\end{proof}

This allows us to define a factorization algebra
\be\label{eq:Yfactalg}
\mathscr{Y}_\hbar^*(\fd)_I^{\textnormal{co-op}}:= \mathrm{Sym}_{X_I}(\CA_I^{\mathrm{out}}\lbb\hbar\rbb)\otimes_{X_I} \mathrm{Sym}_{X_I}(\CA_I^{\vee, \mathrm{out}}\lbb\hbar\rbb).
\ee
The supscription ``co-op" is pretentious, as we haven't given it any other structure besides factorization. We now proceed to define the Hopf algebra structure. However, due to non-freeness of the sheaves involved, this Hopf algebra structure will only be defined over the configuration space
\be
\conf_I(X)\coloneqq X^I\setminus \left\{(x_i)_{i \in I}\,\,\,\bigg|\,\,\, x_i=x_j \text{ for some } i\ne j\right\}.
\ee

\subsubsection{Hopf algebra structure}

Consider the canonical maps of sheaves of Lie algebras
\be\label{eq:embedOK}
\btik
\CA_I^\CO\rar & \CA_I^\CK & \CA_I^{\textnormal{out}}\lar 
\etik
\ee
which clearly respect the factorization structure of equation \eqref{eq:factLie}. We prove the following statement.

\begin{Prop}\label{Prop:CAsplit}
    The embeddings of \eqref{eq:embedOK} identify $\CA_I^\CK$ as a direct sum $\CA_I^\CO\oplus \CA_I^{\textnormal{out}}$.
\end{Prop}

\begin{proof}
We need to show that the induced map $\CA_I^\CO\oplus \CA_I^{\textnormal{out}}\to \CA_I^\CK$ is an isomorphism. For this, we just need to show this for every stalk. Let $(x_i)_{i \in I}\in X^I$, and suppose for a moment that $x_i\ne \infty$ for all $i \in I$. Let \(V \subseteq U^I\) be an open neighborhood of \((x_i)_{i \in I}\) and fix global coordinate $t$ of $U = X \setminus\{\infty\}$ as before and let \(t_i\) be the coordinate of the \(i\)-th factor of \(V\). The sheaf $\Gamma(V,\CA_I^\CO)$ is given by
\be
\varprojlim_n \fg[V][t]/ \left(\prod_{i\in I} (t-t_i)^n\fg[V][t]\right)=\sum_{n = 0}^\infty \sum_{0\leq k< |I|} \fg[V]t^k \prod_{i\in I} (t-t_i)^n, 
\ee
where sums are infinite in general.
Similarly, the sheaf $\Gamma(U^I,\CA_I^\CK)$ is given by
\be
\sum_{n\gg-\infty} \sum_{0\leq k< |I|} \fg [V]t^k \prod_{i \in I} (t-t_i)^n. 
\ee
Here, the subscript $n>-\infty$ means that the coefficients in $\fg[V]$ are non-zero only for finitely many negative $n$. Comparing these local sections with those of $\Gamma(V,\CA_I^{\textnormal{out}})$ from equation \eqref{eq:cechcoh} (see also its basis in \eqref{eq:basis_of_K}), one sees that $\Gamma(V,\CA_I^\CK)=\Gamma(V,\CA_I^\CO)\oplus \Gamma(V,\CA_I^{\textnormal{out}})$ as vector spaces. 
Similarly, the statement can be proven for \(x_j = \infty\) for some \(j \in I\).

\end{proof}

Proposition \ref{Prop:CAsplit} and Lemma \ref{Lem:CAflat} give rise to a tensor product decomposition of universal enveloping algebras
\be
U(\CA_I^\CK)=U (\CA_I^{\textnormal{out}})\otimes_{\CO_{X^I}} U(\CA_I^\CO)
\ee
that respects the factorization structure. By the flatness of all the sheaves involved, the universal enveloping algebras above are classical objects and are identified with the symmetric tensor algebras of each of the sheaves as modules.

The last ingredient we need to sheafify the construction of Section \ref{subsec:SheafFac} is a pairing, which is nicely supplemented by geometry. We will define this pairing for the restriction of the sheaves to $\mathrm{conf}_I(X)$. 

\begin{Lem}\label{Lem:CAfreeconf}
    Let $\CA_{\conf_I(X)}^{\heartsuit}$ and $\CA_{\conf_I(X)}^{\vee,\heartsuit}$ be the restrictions of $\CA_I^\heartsuit$ and $\CA_{\conf_I(X)}^{\vee,\heartsuit}$ to $\conf_I(X)$ respectively. Then $\CA_{\conf_I(X)}^{\textnormal{out}}$ and $\CA_{\conf_I(X)}^{\vee,\textnormal{out}}$ are locally free sheaves and $\CA_{\conf_I(X)}^{\CO}$ and $\CA_{\conf_I(X)}^{\vee,\CO}$ as well as $\CA_{\conf_I(X)}^{\CK}$ and $\CA_{\conf_I(X)}^{\vee,\CK}$ are locally direct products of the sheaf of regular functions. 
\end{Lem}

\begin{proof}
    It suffices to prove the statement for $\CA_I^{\heartsuit}$, since the proof for $\CA_I^{\vee,\heartsuit}$ is analogous.
    By the factorization property of $\CA_I^{\heartsuit}$, applied to the projection $I\twoheadrightarrow I$, we find that
    \be
        \CA_{\conf_I(X)}^{\heartsuit}=\prod_{i \in I} \CA_{\{i\}}^{\heartsuit}\Big\vert_{\conf_I(X)}. 
    \ee
    We therefore only need to prove the statement for $\CA_{\{1\}}^\heartsuit$. Let us first consider $\heartsuit=\textnormal{out}$. If we choose any open neighborhood away from 
    $\infty\in X$, this was already proven in Lemma \ref{Lem:CAflat}. Let $U_\infty$ be the neighborhood of $\infty$, then we have seen that $\Gamma(U_\infty,\CA_{\{1\}}^{\textnormal{out}})$ can be identified with the $\C[z_1]$-module
    \be\label{eq:locally_free_on_conf_proof_I}
        \sum_{k\geq 1}\sum_{n<k} \fg[z_1] t^n \left(\frac{1}{1-z_1t}\right)^{k}. 
    \ee
Therefore, one can find a set of basis for this vector space as a free module over $\C[z_1]$, by
\be\label{eq:basis_Aout1}
w_{\alpha, n}=\frac{t^{n-1} I_\alpha}{(1-z_1t)^n}\,, \qquad n\geq 1, 1 \le \alpha \le d. 
\ee
Indeed, these elements are clearly linearly independent and using identities like 
\begin{equation}
    \frac{z_1t^{n} I_\alpha}{(1-z_1t)^{n+1}} + \frac{t^{n-1} I_\alpha}{(1-z_1t)^n} = \frac{z_1t^nI_\alpha + (1-z_1t)t^{n-1}I_\alpha}{(1-z_1t)^{n+1}} = \frac{t^{n-1}I_\alpha}{(1-z_1t)^{n+1}} 
\end{equation}
inductively, we can see that these elements generate \eqref{eq:locally_free_on_conf_proof_I}. 
This proves the claim for $\heartsuit=\textnormal{out}$. 

Let us now focus on $\heartsuit=\CO$ as the proof for $\heartsuit=\CK$ is identical. The space of sections of $\CA_{\{1\}}^{\CO}$ over any open subset \(V \subseteq U = X \setminus\{\infty\}\) is given by
\be
\Gamma(V,\CA_{\{1\}}^{\CO})=\prod_{n\geq 0} \fg[V](t-t_1)^n,
\ee
which is a direct product, as claimed. Similarly, for an open neighborhood \(V\) of \(\infty\), we have that
\be\label{eq:basis_AO1}
\Gamma(V,\CA_{\{1\}}^{\CO})=\prod_{n\geq 0}  z\fg[V](z-z_1)^{n},
\ee
is also a direct product. 
\end{proof}

Let $\Omega$ be the dualizing sheaf of $X$ and $\Omega_{\lpp I\rpp}$ be the sheaf
    \be
\Omega_{\lpp I\rpp}\coloneqq \hat p_{I,*}^\circ \hat q_I^{\circ,*}  \Omega. 
    \ee
Fibers of this sheaf at $(x_i)_{i\in I}$ is the space of differential forms on the punctured formal neighborhood of $\{x_i\}_{i \in I}$. The sum of residues provides a map $\mathrm{res}\colon \Omega_{\lpp I\rpp}\to \CO_{X^I}$. Combined with the canonical pairing between $\fg$ and $\fg^*$ provides a map
\be
\btik
\CA_I^{\CK}\otimes_{\CO_{X^I}} \CA_I^{\vee, \CK}\rar{\kappa} & \Omega_{\lpp I\rpp}\rar{\mathrm{res}} & \CO_{X^I}
\etik
\ee

\begin{Lem}

   The pairing above defines a topologically perfect pairing between  $\CA_{\conf_I(X)}^{\CK}$ and $\CA_{\conf_I(X)}^{\vee, \CK}$, such that one can identify $\CA_{\conf_I(X)}^{\CO}$ with the dual sheaf of $\CA_{\conf_I(X)}^{\vee, \textnormal{out}}$. 
    
\end{Lem}

\begin{proof}
    Using factorization structure, we may restrict to the case $\CA_{\{1\}}^\heartsuit$. Let $V$ be a localization of $U = X\setminus\{\infty\}$. We can identify
    \be
\CA_{\{1\}}^\CK(V)=\sum_{n\gg -\infty} \fg[V](t-t_1)^n, \qquad \CA_{\{1\}}^{\vee, \CK}(V)=\sum_{n\gg -\infty} \fg^*[V](t-t_1)^n.
    \ee
The residue pairing is the canonical pairing
\be
\langle (t-t_1)^mI_\alpha, (t-t_1)^n I^\beta\rangle=\delta_{\alpha\beta}\mathrm{res}_{t=t_1} (z-z_1)^{m+n}=\delta_{\alpha\beta}\delta_{m+n, -1}. 
\ee
Clearly, it is a topologically perfect pairing. Under this pairing, $\CA_{\{1\}}^{\CO}(V)$ is identified with $\CA_{\{1\}}^{\vee, \textnormal{out}}(V)^*$, the linear dual. 

On the other hand, suppose $V$ is a localization of $U_\infty = X\setminus\{0\}$, then 
\be
\CA_{\{1\}}^\CK(V)=\sum_{n\gg -\infty} \fg[V](z-z_1)^n, \qquad \CA_{\{1\}}^{\vee, \CK}(V)=\sum_{n\gg -\infty} \fg^*[V](z-z_1)^n.
\ee
The pairing is now given by
\be
    \begin{split}
        \langle (z-z_1)^m I_\alpha, (z-z_1)^nI^\beta\rangle&= \delta_{\alpha\beta}\mathrm{res}_{z=z_1} (z-z_1)^{m+n}z^{-2}.
        %\\& = \delta_{\alpha\beta}\textnormal{res}_{z = z_1}  z_1^{-2}(z-z_1)^{m+n}\frac{1}{(1+(z-z_1)/z_1)^2}    \\& = \delta_{\alpha\beta}\textnormal{res}_{z = z_1} \sum_{k\ge 0}(-1)^{k+1}k(z-z_1)^{m+n+k-1}t_1^{-1-k} \\&=
        %\begin{cases}
            %0&\textnormal{if }m+n\ge 0,\\
            %-(-1)^{m+n}(m+n)z_1^{m+n-1}\delta_{\alpha\beta}& \textnormal{if }m+n <0.
        %\end{cases}
    \end{split}
\ee
Here, the shift by $z^{-2}$ comes from the fact that $\CA\otimes \CA^\vee=\fg\otimes \fg^*\otimes \CO(-2)$. This is also a perfect pairing. 
Since
\begin{equation}
    \bigg\langle \frac{zI^\alpha}{(z-z_1)^{m+1}},z(z-z_1)^nI_\beta\bigg\rangle = \delta_{\alpha\beta} \delta_{mn} 
\end{equation}
holds and 
\begin{equation}
    \begin{split}
        &\left\{\frac{zI^\alpha}{(z-z_1)^{m+1}}\,\bigg|\, m \ge 0, 1 \le \alpha \le d\right\} \subseteq \Gamma(V,\CA^{\vee,\textnormal{out}}_{\{1\}}),
        \\&\{z(z-z_1)^{m}I_\alpha\mid m \ge 0, 1 \le \alpha \le d\} \subseteq \Gamma(V,\CA^{\CO}_{\{1\}})
    \end{split}
\end{equation}
are (topological) bases (see \eqref{eq:basis_Aout1} and \eqref{eq:basis_AO1}), we see that \(\Gamma(V,\CA_{\{1\}}^{\vee,\textnormal{out}})^*\) is identified with \(\Gamma(V,\CA_{\{1\}}^{\CO})\).
\end{proof}

As a consequence of this lemma, we can identify $U(\CA_{\conf_I(X)}^\CO)$ with the linear dual of the sheaf $S(\CA_{\conf_I(X)}^{\vee, \textnormal{out}})$. We can now apply the general idea underlying the construction of our Yangian $Y_\hbar (\fd)$ and its dual $Y_\hbar^*(\fd)$, and give the restriction \(\mathscr{Y}_\hbar^*(\fd)_{{\conf_I(X)}}^{\textnormal{co-op}}\) of \eqref{eq:Yfactalg} to \({\conf_I(X)}\) a sheaf of Hopf algebra structure via
\be
\mathscr{Y}_\hbar^*(\fd)_{\conf_I(X)}^{\textnormal{co-op}}\coloneqq S(\CA_{\conf_I(X)}^{\vee, \textnormal{out}})\lbb\hbar\rbb \rtimes_{\CO_{X_\hbar}^I}U(\CA_{\conf_I(X)}^{\textnormal{out}})[\![\hbar]\!]. 
\ee
It is now clear that the factorization algebra we defined here is precisely the sheafification of the local factorization from Section \ref{subsec:SheafFac}. 

\begin{Rem}
    All the constructions above are clearly equivariant with respect to the action of the permutations $S_I$ of the set \(I\) on $\conf_I(X)$, and in particular, descend to sheaves over the genuine configuration space $\mathrm{Conf}_I(X)=[\conf_I(X)/S_I]$. 
\end{Rem}

\begin{Rem}
    If one only wants to define the coalgebra structure, then it is possible to extend the above to the entire $X^I$. The idea is to consider the action groupoid $\wh G(\CO)_{X^I}\times \wh \Gr_{G, X^I}$ over $ \wh \Gr_{G, X^I}$, where $\wh G(\CO)_{X^I}$ and $\wh \Gr_{G, X^I}$ are formal completions at the identity section. Let $\omega_I$ be the dualizing sheaf of this space and $\pi_I$ be the projection to $X^I$, then $\pi_{I, *}(\omega_I)$ is a chiral algebra with a compatible co-algebra structure. The fibers of $\pi_{I, *}(\omega_I)$ at $\mathbf z$ can be canonically identified with $Y_\hbar^* (\fd)_{\mathbf z}^{\textnormal{co-op}}$ as a coalgebra. 
\end{Rem}

\section{Quantum vertex algebras and quantum KZ equations}\label{sec:QVA}

In this section, we study the vacuum module
\be
\CV_{\hbar,k}(\fd)\coloneqq\Ind_{Y_\hbar(\fd)\otimes \C[ c_\ell]}^{\widehat{DY}_{\hbar,\ell} (\fd)} \C_k\lbb\hbar\rbb,
\ee
where \(\C_k[\![\hbar]\!]\) is the representation on which \(c_\ell\) acts by multiplying with \(k\). We will endow this vector space with the structure of a quantum vertex algebra. Moreover, we show that under the condition that $k\ell=-1$, this quantum vertex algebra admits a conformal element through the quantum Segal-Sugawara construction, and its associated difference connection is the quantum KZ connection. In a word, we prove Theorem \ref{Thm:QVAintro}. 

This section is structured as follows. In Section \ref{subsec:QVA}, we study the quantum vertex algebra structure associated to $\CV_{\hbar,k}(\fd)$. In Setion \ref{subsec:QKZ}, we construct the conformal element and construct quantum KZ connection from conformal blocks. In Section \ref{subsec:examples}, we consider examples of $\fd$ from 3-dimensional $\CN=4$ gauge theories. 

\subsection{Quantum vertex algebras from cotangent Yangians}\label{subsec:QVA}

\subsubsection{Recollections on quantum vertex algebras}

Recall from \cite{EK5} that a braided vertex algebra \((V,Y,T,\Omega,S)\) consists of the usual datum \((V,Y,T,\Omega)\) of a vertex algebra over \(\C[\![\hbar]\!]\) but with an additional braiding element \(S\) that controls the failure of locality. More precisely, \(V\) is a topologically free \(\C[\![\hbar]\!]\)-module equipped with a meromorphic multiplication 
\begin{equation}
V \otimes_{\C[\![\hbar]\!]} V \to V(\!(z)\!) \,,\qquad x \otimes y \mapsto  Y(x,z)y\,,\qquad Y(x,z) \in \textnormal{End}_{\C[\![\hbar]\!]}(V)[\![z,z^{-1}]\!],  
\end{equation}
an \(\C[\![\hbar]\!]\)-linear operator \(T \colon V \to V\), a vacuum vector \(\Omega \in V\) and a \(\C[\![\hbar]\!]\)-linear braiding map \(S = 1 + \hbar(\dots) \colon V \otimes_{\C[\![\hbar]\!]} V \to (V \otimes_{\C[\![\hbar]\!]} V)(\!(z)\!)\).

This datum is then subject to the following two axioms similar to a usual vertex algebra:
\begin{itemize}[leftmargin=33pt, labelsep=5pt]
    \item[(QV1)] \textbf{Vacuum vector axiom:} The vector \(\Omega \in V\), called vacuum vector, satisfies \(Y(\Omega,z)= 1\) and \(Y(v,z)\Omega \in v + zV[\![z]\!]\) for all \(v \in V\);
    
    \item[(QV2)] \textbf{Translation operator axiom:} The operator \(T\), called translation operator, satisfies \(T\Omega = 0\) and \(\frac{dY(z)}{dz} = [T,Y(v,z)]\).
\end{itemize}
However, contrary to a usual vertex algebra, the meromorphic multiplication on \(V\) satisfies the following \(S\)-braided adjustment of locality, which is \cite[(QA1)]{EK5} in our notation:
\begin{itemize}[leftmargin=33pt, labelsep=5pt]
    \item[(QV3)] \textbf{S-locality:} The following version of locality is satisfied
    \begin{equation}\label{eq:Slocality}
        \begin{split}
        \langle Y^{13}(z_1)Y^{23}(z_2),S(z_1-z_2)(u \otimes v)\rangle = \langle Y^{23}(z_2)Y^{13}(z_1),u \otimes v\rangle.
        \end{split}
    \end{equation}
    Here, we wrote
    \begin{equation}
        Y(z) \in (V^* \otimes_{\C[\![\hbar]\!]} \textnormal{End}_{\C[\![\hbar]\!]}(V))[\![z,z^{-1}]\!]\,,\qquad \langle v\otimes 1, Y(z)\rangle \coloneqq Y(v,z) \textnormal{ for all }v \in V
    \end{equation}
    for the element associated to \(Y(-,z) \colon V \to \textnormal{End}_{\C[\![\hbar]\!]}(V)[\![z,z^{-1}]\!]\) and the equality \eqref{eq:Slocality} is understood as by extending the matrix entries to functions in \(\C [z_1,z_2,(z_1-z_2)^{-1}][\![\hbar]\!]\) after evaluating on an arbitrary vector in \(V\) and applying an arbitrary functional \(V^*\). 
\end{itemize}
Furthermore, the braiding map \(S\) has to satisfy:
\begin{itemize}
    \item[(S1)] \textbf{Unitarity:} \(S^{21}(z) = S^{-1}(z)\);
    \item[(S2)] \textbf{Compatibility with \(T\):} \([T \otimes 1,S(z)] = -\frac{d S(z)}{dz}\)
    \item[(S3)] \textbf{Yang-Baxter equation:} \(
        S^{12}(z_1)S^{13}(z_1+z_2)S^{23}(z_2) = S^{23}(z_2)S^{13}(z_1+z_2)S^{12}(z_1).\)
\end{itemize}
It is clear that \(V/\hbar V\) is now a usual vertex algebra and we call this the classical limit of \(V\).

Contrary to usual vertex algebras, braided vertex algebras are not automatically associative in a reasonable sense. In general, they only satisfy a quasi-associativity constraint controlled by \(S\); see \cite[Proposition 1.1]{EK5}. A quantum vertex algebra is now a braided vertex algebra satisfying the additional constraint
\begin{itemize}[leftmargin=33pt, labelsep=5pt]
    \item[(QV4)] \textbf{Hexagon axiom:} \(S^{23}(z_2)Y^{12}(z_1) =  Y^{12}(z_1)S^{23}(z_2)S^{13}(z_1+z_2)\). 
\end{itemize}
This axiom ensures the associativity 
\begin{equation}\label{eq:associativity}
    Y(u, z_1+z_2)Y (v, z_2)=Y(Y (u, z_1)v, z_2)\,,\qquad \textnormal{ for all }u,v \in V.
\end{equation}
Here, \eqref{eq:Slocality} is understood as an equality of the respective matrix coefficients inside \(\C(\!(z_2)\!)(\!(z_1)\!)[\![\hbar]\!]\) after evaluating on arbitrary elements of \(V \otimes_{\C[\![\hbar]\!]} V\) and then applying arbitrary elements of \(V^*\). Furthermore, \eqref{eq:associativity} is understood as an identity of matrix coefficients in \(\C[\![z_1,z_2]\!][z_1^{-1},(z_1+z_2)^{-1}][\![\hbar]\!]\)
after evaluating on arbitrary elements of \(V\) and then applying arbitrary elements of \(V^*\).

\subsubsection{Construction of the quantum vacuum vertex algebra}
Let $\CV_{\hbar,k}(\fd)$ be the $\widehat{DY}_{\hbar,\ell}(\fd)$-module
\be
\CV_{\hbar,k}(\fd)=\Ind_{Y_{\hbar}(\fd) \otimes \C[c_\ell]}^{\widehat{DY}_{\hbar,\ell}(\fd)} \C_k\lbb\hbar\rbb=\widehat{DY}_{\hbar,\ell}(\fd)\otimes_{Y_{\hbar}(\fd) \otimes \C[c_\ell]} \C_k\lbb\hbar\rbb.
\ee
Here, $\C_k\lbb\hbar\rbb$ is the trivial module of $Y_\hbar (\fd)$ and $c_\ell$ acts on it by multiplication with \(k \in \C\). This is the quantum analogue of the vacuum  module of $\wh\fd_{\ell}$ at level $k$. One can show that as a vector space, we can identify $\CV_{\hbar,k}(\fd)=Y_\hbar^*(\fd)$.
In this section, we want to prove that there is a canonical quantum vertex algebra structure on \(\mathcal{V}_k(\fd)\).

Let $R$ be the quantum $R$-matrix of $DY_\hbar (\fd)=Y_\hbar (\fd)\otimes_{\C[\![\hbar]\!]} Y^*_\hbar (\fd)^{\textnormal{co-op}}$ described explicitly in Section \ref{subsec:Rmat_double}. We can think of this now as a tensor in $Y_\hbar (\fd) \otimes_{\C[\![\hbar]\!]} \wh Y_{\hbar,\ell}^* (\fd)^{\textnormal{co-op}}$. On the other hand, the meromorphic $R$-matrix $R(z)$ can be thought of as an element in $(Y_\hbar (\fd)\otimes_{\C[\![\hbar]\!]} Y_\hbar (\fd))[\![z^{-1}]\!]$. Let
\begin{equation}
    T(z) \in (Y_\hbar (\fd)\otimes_{\C[\![\hbar]\!]} \End (\CV_{\hbar,k}(\fd)))\lbb z \rbb \quad \textnormal{ and }\quad T^\dagger(z) \in (Y_\hbar (\fd)\otimes_{\C[\![\hbar]\!]} \End (\CV_{\hbar,k}(\fd)))\lbb z^{-1} \rbb     
\end{equation}
be the tensors series obtained from \((\tau_z \otimes 1)R\) and \(R(z)\) respectively by applying the representation map $\widehat{DY}_{\hbar,\ell}(\fd)\to \End (\CV_{\hbar,k}(\fd))$ in the second tensor factor. Here, we recall that \(\tau_z\) is the translation operator \(t \mapsto t +z\) associated to \(\partial_t\).
In the following, we write $A(z+\lambda e_\ell)\coloneqq e^{\lambda z \textnormal{ad}(\pd_\ell)}A(z) $ for series \(A(z) \in Y_\hbar(\fd)(\!(z)\!)\) and \(\lambda \in \C[\![\hbar]\!]\) in case this expression is well-defined.  

\begin{Lem}
    The following relations hold:
    \be
\begin{aligned}\label{eq:RTT}
R^{12}(z_1-z_2) T^{13}(z_1) T^{23}(z_2)& =T^{23}(z_2)T^{13}(z_1) R^{12}(z_1-z_2),  \\
R^{12}(z_1-z_2-k\hbar e_\ell /2)  T^{13}(z_1) T^{\dagger, 23}(z_2) &=T^{\dagger, 23}(z_2)T^{13}(z_1)R^{12}(z_1-z_2+k\hbar e_\ell/2). 
\end{aligned}
\ee
\end{Lem}

\begin{proof}
    The first relation is simply follows from the fact that 
    \begin{equation}
        \begin{split}
            R^{12}(z_1-z_2)(e^{z_1 \textnormal{ad}(\partial_t)}&\otimes e^{z_2 \textnormal{ad}(\partial_t)} \otimes 1)(\Delta^\gamma_\hbar \otimes 1)(R)
            \\&=(e^{z_1 \textnormal{ad}(\partial_t)}\otimes e^{z_2 \textnormal{ad}(\partial_t)} \otimes 1)(\Delta^{\gamma,\textnormal{op}}_{\hbar} \otimes 1) (R)R^{12}(z_1-z_2)
        \end{split} 
    \end{equation}
    holds according to Theorem \ref{thm:Rmat_double}.2 and the pseudotriangular structure of \(Y_\hbar(\fd)\).
    For the second relation, recall that $\wt R = \exp (\hbar \pd_\ell\otimes c_\ell/2)R\exp (\hbar \pd_\ell\otimes c_\ell/2)$ is the $R$-matrix of the double of $\wt{DY}_{\hbar,\ell} (\fd)$. We now claim that the following holds
    \be\label{eq:RTT_proof}
R^{12}(z_1-z_2) \wt R[z_1]^{13} R(z_2)^{23}=R(z_2)^{23}\wt R[z_1]^{13} R^{12}(z_1-z_2),
    \ee
    where \(\wt R[z_1] = (e^{z_1 \partial_t} \otimes 1)\wt R\).
Multiplying by $R^{21}(z_2-z_1)=R^{12}(z_1-z_2)^{-1}$ from the left and right and changing the first and second tensor factor, we see that the above is equivalent to
\be
R^{12}(z_2-z_1)R^{13}(z_2)\wt R^{23}[z_1]=\wt R^{23}[z_1]R^{13}(z_2)R^{12}(z_2-z_1).
\ee
 The latter equation now follows from the fact that $\wt R \widetilde{\Delta}^\gamma_\hbar \wt R^{-1}=\widetilde{\Delta}^{\gamma,\textnormal{op}}_\hbar$ holds. Now the second identity in \eqref{eq:RTT} follows from \eqref{eq:RTT_proof} by multiplying both sides with $\exp (-\hbar \pd_\ell\otimes c_\ell/2)^{13}$ from the left and right and using $c_\ell=k$. 
\end{proof}

\begin{Lem}
    The tensor series
    \be
\CY_\hbar (z)=T(z) \cdot T^\dagger (z+k\hbar e_\ell/2)^{-1}\in (Y_\hbar (\fd)\otimes_{\C[\![\hbar]\!]} \End (\CV_{\hbar,k}(\fd)))\lbb z,z^{-1}\rbb
    \ee
is a well-defined formal series. 
\end{Lem}

\begin{proof}

    Let $M, N$ be smooth modules of $Y_\hbar (\fd)$. Then for any $m\otimes n\in M\otimes N$, we have $T^\dagger (z+k\hbar e_\ell/2)^{-1}(m\otimes n)\in (M\otimes N)\lpp z\rpp$. Since $T(z)$ is a formal power series in $z$, its action on $(M\otimes N)\lpp z\rpp$ is well-defined. 
    
\end{proof}

As stated in \cite{reshetikhin_tianshansky,EK5}, from equation \eqref{eq:RTT} and the unitarity of $R(z)$, one can deduce
\be\label{eq:YTT}
\begin{aligned}
    \CY_\hbar^{13} (z_1) & R^{12}(z_1-z_2+k\hbar e_\ell)^{-1} \CY_\hbar^{23} (z_2) R^{12}(z_1-z_2)\\ &=R^{12}(z_1-z_2)^{-1} \CY_\hbar^{23} (z_2) R^{12}(z_1-z_2-k\hbar e_\ell)  \CY_\hbar^{13} (z_1).
\end{aligned}
\ee
As the notation suggests, this $\CY_\hbar (z)$ defines the state-operator correspondence for the quantum vertex algebra structure on $\CV_{\hbar,k}(\fd) (\fd)$ by writing \(\CY_\hbar(v,z) \coloneqq \langle v\otimes 1,\CY_\hbar(z)\rangle\) for all \(v \in \CV_{\hbar,k}(\fd)\). In its definition, the series $T^\dagger$ are the annihilation operators, and they appear before the creation operators $T$. The vacuum vector $\Omega = 1$ is the identity element of $\CV_{\hbar,k}(\fd)=Y_\hbar^*(\fd)$ and the translation operator is $\pd_t$. We are left to define $\CS (z)$. To do so, consider 
\begin{equation}
    L_1(z), L_2(z) \colon Y_\hbar (\fd)\otimes_{\C[\![\hbar]\!]} Y_\hbar (\fd)\to (Y_\hbar (\fd)\otimes_{\C[\![\hbar]\!]} Y_\hbar (\fd))(\!(z^{-1})\!)    
\end{equation}
defined by
\be
\begin{aligned}
    &  a\otimes b\stackrel{L_1(z)}\longmapsto (a\otimes 1)R(z+k\hbar e_\ell)^{-1}(1\otimes b)R(z),\\
    & a\otimes b\stackrel{L_2(z)}\longmapsto R(z)^{-1}(1\otimes b)R(z-k\hbar e_\ell)(a\otimes 1).
\end{aligned}
\ee
Equation \eqref{eq:YTT} can be neatly written as
\be\label{eq:LYY}
L_1^{12}(z_1-z_2)  \CY_\hbar^{13} (z_1)\CY_\hbar^{23} (z_2)=L_2^{12}(z_1-z_2)\CY_\hbar^{23} (z_2) \CY_\hbar^{13} (z_1).  
\ee
Moreover, the two endomorphisms satisfy
\be
L_1(-z)^{21}= \mathrm{Ad}_{R(z)} L_2(z). 
\ee
Let $\CS^\perp(z)$ be the endomorphism of $Y_\hbar (\fd)\otimes_{\C[\![\hbar]\!]} Y_\hbar (\fd)$ so that \(L_2(z)\CS^\perp (z)=L_1(z)\) or
\be\label{eq:CSperp}
\CS^\perp (z)= L_2^{-1}(z)L_1(z) = L_1^{21,-1} (-z)\mathrm{Ad}_{R(z)} L_1(z). 
\ee
This is well-defined since clearly $L_1$ and $L_2$ are invertible. We define $\CS(z)$ to be the adjoint of $\CS^\perp (z)$. The main statement to be proven in this section is the following. Its proof will occupy rest of this subsection. 

\begin{Thm}\label{Thm:QVAd}

    The tuple $(\CV_{\hbar,k}(\fd), \CY_\hbar, 1, \pd_t, \CS)$ defines a quantum vertex algebra whose classical limit is the classical vertex algebra $V_k(\fd)$. 
    
\end{Thm}

Let us write $\CY_\hbar$ a bit more explicit by using the identification $\CV_{\hbar,k}(\fd)\cong Y^*_\hbar (\fd)$. We have
\be
\begin{split}
\CY_\hbar(v, z)&= \langle v \otimes 1, T(z)T^{\dagger}(z+k\hbar e_\ell/2)\rangle = \langle \Delta_\hbar^\gamma(v) \otimes 1,T^{13}(z)T^{\dagger,23}(z+k\hbar e_\ell/2)^{-1}\rangle \\&=
\sum_{(v)}\tau_z(v^{(2)})R(S^{-1}(v^{(1)}), z+k\hbar e_\ell/2),    
\end{split}
\ee
for all \(v \in \CV_{\hbar,k}(\fd)\), where we used Sweedler's notation for the double, so \(\Delta_\hbar^{\gamma,\textnormal{op}}(v) = \sum_{(v)}v^{(2)}\otimes v^{(1)}\) is the coproduct of \(v\) in \(Y_\hbar^*(\fd)\), and we wrote
\begin{equation}\
    R(v,z+k\hbar \ell/2) \coloneqq \langle v \otimes 1,R(z + k\hbar e_\ell/2)\rangle.
\end{equation}
Let us consider the Hopf algebra $\widehat{DY}_{\hbar, \ell}(\fd)$, which contains $Y^*_\hbar (\fd)^{\textnormal{co-op}}$ as a subalgebra, and consider the following map
\be
\wt\CY_\hbar\colon \widehat{DY}_{\hbar, \ell}(\fd)\longrightarrow \Big(\widehat{DY}_{\hbar, \ell}(\fd)/(c_\ell-k)\Big)\lbb z,z^{-1}\rbb, \qquad v\mapsto \sum_{(v)}\tau_z(v^{(2)})R(S^{-1}(v^{(1)}), z).
\ee
Here, by $R(v, z) \coloneqq \langle v \otimes 1,R(z)\rangle $ we mean the map corresponding to \(R(z)\) and we use Sweedler's notation for the coproduct of \(\wh{DY}_{\hbar,\ell}(\fd)\) and not of \(Y_\hbar^*(\fd)\). 

\begin{Lem}
    The map \(\wt\CY_\hbar\) satisfies \(\wt\CY_\hbar (c_\ell, z)=k\) and
    \be\label{eq:CY_yangian_rightaction}
        \wt\CY_\hbar (va, z)=\epsilon (a)\wt\CY_\hbar (v, z), 
    \ee
    for all \(a \in Y_\hbar(\fd)\) and \(v \in Y^*_\hbar(\fd)^{\textnormal{co-op}} = \mathcal{V}_{\hbar,k}(\fd)\). Therefore, \(\wt\CY_\hbar\) descends to a map from $\CV_{\hbar,k}(\fd)$ to $\widehat{DY}_{\hbar, \ell}(\fd)/(c_\ell-k)$ and this map coincides with $\CY_\hbar$.
\end{Lem}

\begin{proof}

The equality \(\wt\CY_\hbar (c_\ell, z)=k\) follows from $R(c_\ell, z)=0$ and \eqref{eq:CY_yangian_rightaction} follows from
\begin{equation}
    \begin{split}
        \wt\CY_\hbar (va, z)&=\sum_{(a),(v)}\tau_z(v^{(2)})\tau_z(a^{(2)})R(S^{-1}(a^{(1)}),z)R(S^{-1}(v^{(1)}),z)
        \\&= \sum_{(a)}\tau_z(a^{(2)}S^{-1}(a^{(1)}))\sum_{(v)}\tau_z(v^{(2)})R(s^{-1}(v^{(1)}),z) = \epsilon(a)\wt\CY_\hbar(v,z), 
    \end{split}
\end{equation}
where we used that $R(b, z)=\tau_z(b)$ holds for all \(b \in Y_\hbar(\fd)\). We are left to prove that the map induced by \(\wt\CY_\hbar\) on $\CV_{\hbar,k}(\fd)$ coincides with $\CY_\hbar$. 
    
Let $v\in Y_\hbar^*(\fd)^{\textnormal{co-op}}\subseteq \widehat{DY}_{\hbar, \ell}(\fd)$ and recall that the coproduct in this algebra is given by
    \be
        \wt{\Delta}_\hbar^{\gamma}(v) = \sum_{(v)}\exp (\hbar/2 (\pd_\ell\otimes c_\ell-c_\ell\otimes \pd_\ell)) (v^{(1)}\otimes v^{(2)}).
    \ee
    Applying this to $\wt \CY_\hbar$ and using $R(c_\ell)=0$, the only contribution from the derivations is from $\exp (\hbar/2( c_\ell\otimes \pd_\ell))$. Therefore, using $c_\ell=k$, we find
    \be
        \begin{split}
           \wt \CY_\hbar (v,z)&=\sum_{(v)}\tau_z (v^{(2)})R(e^{\hbar k\pd_\ell/2}S^{-1}(v^{(1)}), z)
           \\&=\sum_{(u)}\tau_z (v^{(2)})R(S^{-1}(v^{(1)}), z+\hbar ke_\ell/2)=\CY_\hbar (v,z), 
        \end{split}
    \ee
    holds, as desired. 
\end{proof}

As a consequence, we show that $\CY_\hbar$ is a map of $Y_\hbar (\fd)$-modules in an appropriate sense.

\begin{Cor}\label{Cor:Ymodulemap}
    For all $a\in \CY_\hbar (\fd)$ and $u,v\in \CV_{\hbar,k}(\fd)$, we have
    \be
        \sum_{(a)}\CY_\hbar (a^{(2)}u, z)\tau_z (a^{(1)})v=\tau_z(a)\CY_\hbar (u,z)v.
    \ee
    In other words, $\CY_\hbar$ is a $Y_\hbar (\fd)$-module map from $(\tau_z\CV_{\hbar,k}(\fd)\otimes_{\C[\![\hbar]\!]} \CV_{\hbar,k}(\fd))^{\textnormal{op}}$ to $\tau_z\CV_{\hbar,k}(\fd)$.  
\end{Cor}

\begin{proof}
  For any $a\in Y_\hbar (\fd)$ and $v\in \wh{DY}_{\hbar,\ell}(\fd)$, we have
    \be
        \begin{split}
            &\sum_{(a)}\wt \CY_\hbar (a^{(2)}v, z)\tau_z (a^{(1)}) \\&= \sum_{(a),(v)}\tau_z(a^{(3)}v^{(2)})\langle S^{-1}(a^{(2)}v^{(1)}) \otimes 1,R(z)\rangle \tau_z(a^{(1)}) 
            \\&=\sum_{(a),(v)}\tau_z(a^{(3)}v^{(2)})\langle S(v^{(1)}) \otimes S^{-1}(a^{(2)}) \otimes 1,(\wt\Delta_\hbar^\gamma \otimes 1)R(z)\rangle \tau_z(a^{(1)}) 
            \\&=\sum_{(a),(v)}\tau_z(a^{(3)}v^{(2)})\langle S(v^{(1)}) \otimes S^{-1}(a^{(2)}) \otimes 1,R^{13}(z)R^{23}(z)\rangle \tau_z(a^{(1)}) 
            \\&=\sum_{(a),(v)}\tau_z(a^{(3)})\tau_z(v^{(2)}) R(S(v^{(1)}), z)R(S^{-1}(a^{(2)}), z)\tau_z(a^{(1)})\\&=\tau_z(a) \wt \CY_\hbar (v, z).
        \end{split}
    \ee
    Acting on $\CV_{\hbar,k}(\fd)$ gives the desired result. 
\end{proof}

The next issue we need to deal with is the action of $\CS(z)$. To be able to understand this action correctly, we need to understand the composition $L_2(z)^{-1}L_1(z)$ and how to dualize this action. Let us consider the pairing between $\CV_{\hbar,k}(\fd)$ and $Y_\hbar (\fd)$ via the identification $Y_\hbar(\fd)^*\cong \CV_{\hbar,k}(\fd)$. 

\begin{Lem}
    We have the following commutation relation in $\wh{DY}_\hbar (\fd)/(c_\ell-k)$:
    \be
av=\sum_{(a),(v)}\langle e^{\hbar k\pd_\ell/2}v^{(1)}, a^{(1)}\rangle v^{(2)}a^{(2)}\langle e^{-\hbar k\pd_\ell/2}v^{(3)}, S^{-1}(a^{(3)})\rangle, \qquad \forall a\in Y_\hbar (\fd), v\in Y_\hbar^*(\fd)^{\textnormal{co-op}}.
    \ee
    In particular, under the pairing of $\CV_{\hbar,k}(\fd)$ with $Y_\hbar (\fd)$, we find
    \be\label{eq:conjact}
\langle av, b\rangle =\sum_{(a)}\langle v, S^{-1}(\tau_{- \hbar k e_\ell/2} (a^{(2)}))b \tau_{ \hbar k e_\ell/2}(a^{(1)})\rangle\,,\qquad \forall a,b \in Y_\hbar(\fd), v\in Y_\hbar^*(\fd),
    \ee
    where we wrote \(\tau_{\pm \hbar k e_\ell/2} = e^{\pm\hbar k\textnormal{ad}(\partial_\ell)/2}\).
\end{Lem}

\begin{proof}

    We have the following commutation relation in Drinfeld double
    \be\label{eq:comm_rel_extdouble}
av=\langle \wt v^{(1)}, a^{(1)}\rangle\wt v^{(2)}a^{(2)}\langle \wt v^{(3)}, S^{-1}(a^{(3)})\rangle,
    \ee
    where $v \mapsto \wt v^{(1)} \otimes \wt v^{(2)} \otimes \wt v^{(3)}$ is the double coproduct of $v$ in the \(\widehat{DY}_{\hbar,\ell}(\fd)\). The only contribution from $e^{\hbar (\pd_\ell\otimes c_\ell-c_\ell\otimes \pd_\ell)/2}$ is if $c_\ell$ appears on the second factor, which comes from $e^{\hbar(\pd_\ell\otimes c_\ell\otimes 1- 1\otimes c_\ell\otimes \pd_\ell)/2}$. This gives the desired commutation relation. 

    Now, \eqref{eq:conjact} holds, since
    \be
        \begin{split}
            \langle av,b\rangle&=\sum_{(a),(v)}\langle e^{\hbar k\pd_\ell/2}v^{(1)}, a^{(1)}\rangle \langle v^{(2)}a^{(2)},b\rangle\langle e^{-\hbar k\pd_\ell/2}v^{(3)}, S^{-1}(a^{(3)})\rangle
            \\&= \sum_{(a),(v)}\langle e^{\hbar k\pd_\ell/2}v^{(1)}, a^{(1)}\rangle \langle v^{(2)},b\rangle\epsilon(a^{(2)})\langle e^{-\hbar k\pd_\ell/2}v^{(3)}, S^{-1}(a^{(3)})\rangle
            \\&= \sum_{(a),(v)}\langle v^{(1)}, \tau_{k\hbar e_\ell/2}(a^{(1)})\rangle \langle v^{(2)},b\rangle\langle v^{(3)}, S^{-1}(\tau_{-k\hbar e_\ell/2}(a^{(2)}))\rangle
            \\& = \sum_{(a)}\langle v, S^{-1}(\tau_{- \hbar k e_\ell/2} (a^{(2)}))b \tau_{ \hbar k e_\ell/2}(a^{(1)})\rangle
        \end{split}
    \ee
 using the commutation relation \eqref{eq:comm_rel_extdouble}. 
\end{proof}

Let us denote by $\rhd_{k} $ the action of $Y_\hbar (\fd)$ on itself given by
\be\label{eq:rhd_k}
a\rhd_{k} b \coloneqq \tau_{- \hbar k e_\ell/2} (a^{(1)})b \tau_{ \hbar k e_\ell/2}(S(a^{(2)}))\,,\qquad a,b \in Y_\hbar(\fd).
\ee
Then the equation \eqref{eq:conjact} can be written as
\be\label{eq:rhd_k_duality}
\langle av, b\rangle=\langle v, S^{-1}(a)\rhd_k b\rangle\,,\qquad v \in \CV_{\hbar,k}(\fd). 
\ee
This allows us to understand $\CS(z)$ explicitly.

\begin{Prop}
     The action of \(\CS(z)\) on $\CV_{\hbar,k}(\fd)\otimes_{\C[\![\hbar]\!]}\CV_{\hbar,k}(\fd)$, which was defined as the adjoint of \(S^\bot(z)\) from \eqref{eq:CSperp},  is given by the action of $R(z)$ using the $Y_\hbar(\fd)$-module structure, i.e.
     \begin{equation}
         \langle u\otimes v,\CS(z)^\perp(a\otimes b)\rangle=\langle R(z)\cdot (u\otimes v), a\otimes b\rangle
     \end{equation}
     holds for all \(a,b \in Y_\hbar(\fd),u,v \in \CV_{\hbar,k}(\fd)\).
\end{Prop}

\begin{proof}
    We calculate the action of $L_1(z)^{-1}L_2(z)$ on $Y_\hbar (\fd)^{\otimes 2}$. To do so, we first prove that
\be\label{eq:inverse_L1}
L_1(z)^{-1}(a\otimes b)=(a\otimes 1)(S\otimes 1) (R(z+k\hbar e_\ell)^{-1}(1\otimes b) R(z)).
\ee
holds for all \(a,b \in Y_\hbar(\fd)\). To do so, let us recall that 
\begin{equation}\label{eq:R_inverse_identity}
    \begin{split}
        1 \otimes 1 = (\epsilon \otimes 1)R = (\nabla_\hbar^\gamma(S \otimes 1)\Delta_\hbar^\gamma \otimes 1)R = ((S \otimes 1)R)R
    \end{split}
\end{equation}
holds. 
Let us write
\begin{equation}\label{eq:R_tensor_notation}
    R(z) = \sum_{i \in I}R^{(1)}_i(z) \otimes R_i^{(2)} \,\quad \textnormal{ and }\quad R(z)^{-1} = \sum_{i \in I}R_i^{-1,(1)}(z) \otimes R_i^{-1,(2)}.
\end{equation}
Then, we can calculate
\begin{equation}
    \begin{split}
        &L_1(z)L_1(z)^{-1}(a \otimes b) = L_1(z)\left(\sum_{i_1,i_2 \in I}aS(R^{(1)}_{i_1}(z))S(R_{i_2}^{-1,(1)}(z+k\hbar e_\ell)) \otimes R_{i_2}^{-1,(2)}bR_{i_1}^{(2)}\right)
    \\&=\sum_{i_1,i_2,i_3,i_4\in I}aS(R^{(1)}_{i_1}(z))S(R_{i_2}^{-1,(1)}(z+k\hbar e_\ell))R_{i_3}^{-1,(1)}(z+k\hbar e_\ell)R_{i_4}^{(1)}(z) \otimes R_{i_3}^{-1,(2)}R_{i_2}^{-1,(2)}bR_{i_1}^{(2)}R_{i_4}^{(2)}
    \\&= a \otimes b,
    \end{split}
\end{equation}
where we used \eqref{eq:R_inverse_identity} adapted to \(R(z)\) and a similar identity for \(R(z+k\hbar e_\ell)^{-1}\). This proves \eqref{eq:inverse_L1}.

Now, using \eqref{eq:inverse_L1} and the notation \eqref{eq:R_tensor_notation}, \(L_1(z)^{-1}L_2(z)(a \otimes b)\) for \(a,b \in Y_\hbar(\fd)\) is equal to 
\be
\sum_{i_1,i_2,i_3,i_4 \in I}R_{i_1}^{-1,(1)}(z)R_{i_2}^{(1)}(z-k\hbar e_\ell)a S(R_{i_3}^{(1)}(z))S(R_{i_4}^{-1,(1)}(z+k\hbar e_\ell))\otimes R_{i_4}^{-1,(2)}R_{i_1}^{-1,(2)}b R_{i_2}^{(2)}R_{i_3}^{(2)},
%S(R_1^{(1)}(z))R_2^{(1)}(z-k\hbar e_\ell)a S(S(R_3^{(1)})(z+k\hbar e_\ell)R_4^{(1)}(z))\otimes R_3^{(2)}R_1^{(2)}b R_2^{(2)}R_4^{(2)}
\ee
If we use \eqref{eq:R_inverse_identity}, we can rewrite this in the form $\sum O^{(1)}aS(O^{(2)})\otimes O^{(3)}b S(O^{(4)})$, where
\be
\begin{aligned}
    &\sum O^{(1)}\otimes O^{(2)} \otimes O^{(3)} \otimes O^{(4)}\\&=R^{23}(z+k\hbar e_\ell)^{-1} R^{24}(z)^{-1}R^{13}(z)^{-1} R^{14}(z-k\hbar e_\ell)^{-1}\\
    &=(\tau_{-k\hbar e_\ell/2}\otimes \tau_{k\hbar e_\ell/2} \otimes \tau_{-k\hbar e_\ell/2}\otimes \tau_{k\hbar e_\ell/2})\left(R^{23}(z)^{-1} R^{24}(z)^{-1}R^{13}(z)^{-1} R^{14}(z)^{-1}\right)\\
    &=(\tau_{-k\hbar e_\ell/2}\otimes \tau_{k\hbar e_\ell/2} \otimes \tau_{-k\hbar e_\ell/2}\otimes \tau_{k\hbar e_\ell/2})\left((1\otimes 1 \otimes \Delta_{\hbar}^\gamma) (R^{23}(z))^{-1} (1\otimes 1\otimes \Delta_{\hbar}^\gamma) (R^{13}(z))^{-1}\right)\\&=(\tau_{-k\hbar e_\ell/2}\otimes \tau_{k\hbar e_\ell/2} \otimes \tau_{-k\hbar e_\ell/2}\otimes \tau_{k\hbar e_\ell/2})\left((1\otimes 1 \otimes \Delta_{\hbar}^\gamma)(\Delta_{\hbar}^\gamma\otimes 1)(R(z))^{-1}\right)
    \\&=(\tau_{-k\hbar e_\ell/2}\otimes \tau_{k\hbar e_\ell/2} \otimes \tau_{-k\hbar e_\ell/2}\otimes \tau_{k\hbar e_\ell/2})\left((\Delta_{\hbar}^\gamma\otimes \Delta_\hbar^\gamma)(R(z)^{-1})\right).
    %   O^{(1)}\otimes O^{(2)} \otimes O^{(3)} \otimes O^{(4)}&=R^{23}(z+k\hbar e_\ell)^{-1} R^{24}(z)^{-1}R^{13}(z)^{-1} R^{14}(z-k\hbar e_\ell)^{-1}\\
    % &=(\tau_{-k\hbar e_\ell/2}\otimes \tau_{k\hbar e_\ell/2})^{\otimes 2}(R^{23}(z)^{-1} R^{24}(z)^{-1}R^{13}(z)^{-1} R^{14}(z)^{-1})\\
    % &=(\tau_{-k\hbar e_\ell/2}\otimes \tau_{k\hbar e_\ell/2})^{\otimes 2}1\otimes \Delta_{\hbar}^\gamma (R^{23}(z))^{-1} 1\otimes \Delta_{\hbar}^\gamma (R^{13}(z))^{-1}\\&=(\tau_{-k\hbar e_\ell/2}\otimes \tau_{k\hbar e_\ell/2})^{\otimes 2}1\otimes \Delta_{\hbar}^\gamma(\Delta_{\hbar}^\gamma\otimes 1(R(z))^{-1})
    % \\&=(\tau_{-k\hbar e_\ell/2}\otimes \tau_{k\hbar e_\ell/2})^{\otimes 2}\Delta_{\hbar}^\gamma\otimes \Delta_\hbar^\gamma (R(z)^{-1})
    % \\&=(\tau_{-k\hbar e_\ell/2}\otimes \tau_{k\hbar e_\ell/2})^{\otimes 2}\Delta_{\hbar}^\gamma\otimes \Delta_\hbar^\gamma(R(z)^{-1}).
\end{aligned}
\ee
Comparing with the definition of $\rhd_k$ in \eqref{eq:rhd_k}, we see that $L_1(z)^{-1}L_2(z)=R(z)^{-1}\rhd_k$, and so $\CS(z)^\perp=R(z)\rhd_k$ by definition in \eqref{eq:CSperp}. Consequently,
\be
\langle u\otimes v,\CS(z)^\perp(a\otimes b)\rangle=\langle u\otimes v, (S^{-1}\otimes S^{-1})(R(z))\rhd_k (a\otimes b)\rangle=\langle R(z)\cdot (u\otimes v), a\otimes b\rangle,
\ee
holds due to \eqref{eq:rhd_k_duality}, as desired. 
\end{proof}

We can now prove Theorem \ref{Thm:QVAd}. To begin, let us note that (QV1) and (QV2) are clear. We now prove (QV3), namely $\CS$-locality.

\begin{Lem}
    The endomorphism $\CS (z)$ satisfies
    \be
\langle \CY_\hbar^{13}(z_1) \CY_\hbar^{23}(z_2), \CS^{12}(z_1-z_2)(u\otimes v)\rangle =\langle \CY_\hbar^{23}(z_2) \CY_\hbar^{13}(z_1), u\otimes v\rangle.
    \ee
\end{Lem}

\begin{proof}
    By definition, we have
    \be
\begin{aligned}
    \langle \CY_\hbar^{13}(z_1) \CY_\hbar^{23}(z_2), \CS^{12}(z_1-z_2)(u\otimes v)\rangle &\stackrel{\phantom{\eqref{eq:LYY}}}= \langle L_2 (z_1-z_2)^{-1}L_1 (z_1-z_2)\CY_\hbar^{13}(z_1) \CY_\hbar^{23}(z_2), u\otimes v\rangle\\ &\stackrel{\eqref{eq:LYY}}{=} \langle L_2 (z_1-z_2)^{-1} L_2 (z_1-z_2)\CY_\hbar^{23}(z_2)\CY_\hbar^{13}(z_1), u\otimes v\rangle\\ &\stackrel{\phantom{\eqref{eq:LYY}}}=\langle \CY_\hbar^{23}(z_2) \CY_\hbar^{13}(z_1), u\otimes v\rangle,
\end{aligned}
    \ee
    for all \(u,v \in Y_\hbar(\fd)\), as desired. 
\end{proof}

Axioms (S1)-(S3) simply follows from the identification of $\CS(z)$ with the action of $R(z)$, and identities satisfied by it. The hexagon identity (QV4) now immediately follows from the cocycle condition of $R(z)$ and the fact that $\CY_\hbar$ is a \(Y_\hbar(\fd)\)-module homomorphism in the sense of Corollary \ref{Cor:Ymodulemap}. This completes the proof of Theorem \ref{Thm:QVAd}.

\begin{Rem}

    The above proof is completely general. In fact, one can show that for any Hopf algebra $A$ together with a nilpotent derivation $T$, a spectral $R$-matrix $R(z)\in (A\otimes A)\lpp z^{-1}\rpp$ and an appropriate dual Hopf algebra $A^*$, the vacuum module $\mathrm{Ind}_A^{DA}(\C)$ has the structure of a quantum vertex algebra. 
    
\end{Rem}

\subsection{Quantum KZ equations}\label{subsec:QKZ}

In this section, we study the quantum KZ connections associated to the spectral \(R\)-matrix $R(z)$ of \(Y_\hbar(\fd)\); see \eqref{eq:spectral_Rmatrix_of_Yangian}. This system of connections takes the form
\be
\nabla_i(\mathbf{z})=R^{i-1, i}(z_{i-1}-z_i+k\hbar) \cdots R^{1,i}(z_1-z_i+k\hbar)\cdot R^{n,i} (z_n-z_i)\cdots R^{i+1, i}(z_{i+1}-z_i),
\ee
for any \(\mathbf{z}=(z_i)_{i \in I} \in \textnormal{conf}_I(X) \coloneqq X \setminus \left\{z_i = z_j \mid i,j \in I,i\neq j\right\}\), where as always \(X = \mathbb{P}^1\).
Following \cite{EK4}, we want to show that this system of equations naturally arises from quantum conformal blocks associated to the central extension of the cotangent double Yangian, for a specific choice of $k$ and $\ell$. 

\subsubsection{Modules of the centrally-extended double}\label{subsec:dimodules}

In \cite[Section 3]{EK4}, the authors considered various notions of dimodules for the centrally extended dual Yangian for $\fg$. This is partly because the double of ordinary Yangian is not conveniently known as an algebra. In the case of $\fd=T^*\fg$, however, we have complete control over the double. We therefore simply consider usual modules over the double Yangian.

Let $W$ be a finite-dimensional module of $Y_\hbar (\fd)$, and let $\mathbf{z} = (z_i)_{i \in I} \subset X$ be a finite set of points. For any \(i \in I\), we consider the embedding \(Y_\hbar(\fd) \cong Y_{\hbar}(\fd)_{z_i} \to Y_{\hbar}(\fd)_{\mathbf{z}}:= Y_\hbar(\fd)^{\otimes |I|}\) into the \(i\)-th factor and view $W$ as a module of $ Y_{\hbar}(\fd)_{\mathbf{z}}$ whose action is non-trivial only for \(Y_{\hbar}(\fd)_{z_i}\).

On the other hand, the meromorphic R-matrix $R(z)$ supplies an algebra morphism
\be
\wh Y_{\hbar,\ell}^*(\fd)_{\hat{\mathbf{z}}}^{\textnormal{co-op}} \to Y_\hbar(\fd)_{z_i}, \qquad \hat{\mathbf{z}}_i=\mathbf{z}\setminus\{z_i\} = \{z_j\}_{j \in I\setminus\{i\}}.
\ee
Using this map, we obtain a module $W(z_i)$ of $\wh Y_{\hbar,\ell}^*(\fd)_{\hat{\mathbf{z}}}^{\textnormal{co-op}}$ for every finite-dimensional module $W$ of $Y_\hbar (\fd)$ and any $\mathbf{z} = \{z_i\}_{i \in I} \subset X$. 

\begin{Lem}
  Let $Y_{\hbar,\ell} (\fd)_{\mathbf{z}, i}$ be the Hopf subalgebra of $\widehat{DY}_{\hbar,\ell} (\fd)_{\mathbf{z}}$ generated by $ Y_{\hbar,\ell}^*(\fd)_{\hat{\mathbf{z}}_i}$ and $Y_\hbar(\fd)_{z_i}$, then the above structures make $W(z_i)$ into a module of  $Y_{\hbar,\ell} (\fd)_{\mathbf{z}, i}$. 
\end{Lem}

We can now define $\wh W_k (z_i)$ to be the induction of $W(z_i)$ from $Y_{\hbar,\ell} (\fd)_{\mathbf{z}, i}$ to $\widehat{DY}_{\hbar,\ell} (\fd)_{\mathbf{z}}$. As a vector space, we can identify $\wh W_k (z_i)$ with $Y_\hbar^*(\fd)_{z_i}\otimes W_k$, where \(W_k\) is the \(Y_\hbar^*(\fd)^{\textnormal{co-op}}\otimes \C[c_\ell]\)-module obtained from \(W\) on which \(c_\ell\) acts by multiplication with \(k\).

\subsubsection{Quantum conformal blocks}\label{subsec:Qconformal}

Let $\{z_i\}_{i \in I}\subset X\) be a finite subset and $\{W_i\}_{i \in I}$ be a finite family of finite-dimensional modules of $Y_\hbar (\fd)$. Furthermore, let $\wh W_{i, k}(z_i)$ be the corresponding module of $\wh{DY}_{\hbar,\ell} (\fd)_{\mathbf{z}}$ for all \(i \in I\). The tensor product
\be
M_k(\{W_i\}_{i \in I}, \mathbf z)\coloneqq\bigotimes_{i\in I} \wh W_{i, k}(z_i)
\ee
is a module of $\wh{DY}_{\hbar,\ell} (\fd)_{\mathbf{z}}$ where $c_{\ell,i}=c_{\ell,j}$ acts by \(k\) for all $i\ne j$. Therefore, we can let $\wh{DY}_{\hbar,\ell} (\fd)_{\mathbf{z}, k}$ be the quotient Hopf algebra by the ideal $(c_{\ell,i}-c_{\ell,j})$ and $(c_{\ell,i}-k)$, then $M_k(\{W_i\}, \mathbf z)$ is a module of $\wh{DY}_{\hbar,\ell} (\fd)_{\mathbf{z}, k}$ with central charge $k$. We have seen that $Y_{\hbar,\ell}^*(\fd)$ is a subalgebra of $\wh{DY}_{\hbar,\ell} (\fd)_{\mathbf{z}, k}$. 

\begin{Def}
    Define the quantum conformal block $B_k(\{W_i\}, \mathbf z)$ to be the space of invariant functionals
    \be
        B_k(\{W_i\}_{i \in I}, \mathbf z)\coloneqq\Hom_{Y_{\hbar,\ell}^*(\fd)_{\mathbf{z}}} \lp M_k(\{W_i\}_{i \in I}, \mathbf z), \C\lbb\hbar\rbb\rp.
    \ee
\end{Def}

Since we can identify $\bigotimes_{i\in I} \wh W_{i, k}(z_i)$ with the induction of $\bigotimes_{i \in I} W (z_i)$ from $Y_{\hbar,\ell}(\fd)_{\mathbf z}$ to $\wh{DY}_{\hbar,\ell} (\fd)_{\mathbf{z}, k}$, using the induction-restriction adjunction, we have an isomorphism
\be
    \begin{split}
        B_k(\{W_i\}_{i\in I}, \mathbf z)&=\Hom_{Y_{\hbar,\ell}^*(\fd)_{\mathbf{z}}} \lp M_k(\{W_i\}_{i\in I}, \mathbf z), \C\lbb\hbar\rbb\rp
        \\&\cong \Hom_{\C\lbb\hbar\rbb}\left(\bigotimes_{i\in I} W_i, \C\lbb\hbar\rbb\right)=\bigotimes_{i\in I} W_i^*.         
    \end{split}
\ee

\subsubsection{Segal-Sugawara construction}

Let $R\in Y_\hbar(\fd)\otimes Y_\hbar^* (\fd)$ be the $R$-matrix of the double. We define $Q\in \wh{DY}_{\hbar,\ell} (\fd)$ by the formula
\be\label{eq:quantum_segal}
Q=\widehat{\nabla}_\hbar^\gamma \lp\lp S\otimes  e^{k\hbar\textnormal{ad}(\pd_\ell)/2}\rp R^{21}\rp,
\ee
where \(\widehat{\nabla}_\hbar^\gamma\) is the multiplication map of \(\widehat{DY}_{\hbar,\ell}(\fd)\).
This is well-defined when acting on smooth modules of $\wh{DY}_{\hbar,\ell} (\fd)$, and in particular it acts on $\wh W_k$ for all finite-dimensional \(Y_\hbar(\fd)\)-module \(W\). 

Let $\alpha$ be the automorphism of $\wh{DY}_{\hbar,\ell} (\fd)$ defined by conjugation with $\exp \lp \hbar ( k\pd_{\ell}+\pd_\epsilon)\rp$, and let $\wh W_k^\alpha$ be the twist of the module $\wh W_k$ by $\alpha$.

\begin{Lem}\label{lem:Q_and_alpha}
    The map $Q$ gives an isomorphism \(Q\colon \wh W_k\stackrel{\cong}\to \wh W_k^\alpha\).
\end{Lem}
\begin{proof}
    Since \(c_\ell\) acts on \(\widehat{W}_k\) as multiplication with \(k\), we have 
    \begin{equation}
        \widehat{\nabla}_{\hbar}^\gamma((S\otimes 1)\widetilde{R}^{21}) = \widehat{\nabla}_{\hbar}^\gamma(e^{-\hbar c_\ell \otimes \partial_\ell/2}((S\otimes 1)R^{21})e^{-\hbar c_\ell \otimes \partial_\ell/2}) = e^{-k\hbar \partial_\ell}Q
    \end{equation}
    when acting on \(\widehat{W}_k\).
    Furthermore, using the description 
    \begin{equation}
        S^2 = \textnormal{Ad}_{\widetilde{DY}_{\hbar,\ell}(\fd)}\left(\widetilde{\nabla}_\hbar^\gamma \left((S \otimes 1)\widetilde{R}^{21}\right)\right)    
    \end{equation}
    by Drinfeld \cite{Drinfeld_almost_cocommutative}, where \(\wh\nabla_\hbar^\gamma\) is the multiplication map of \(\widetilde{DY}_{\hbar,\ell}(\fd)\), and \(S^2 = \exp({\hbar \textnormal{ad}(\partial_\epsilon)})\),  we deduce \(\textnormal{Ad}_{\widetilde{DY}_{\hbar,\ell}(\fd)}(Q) = \exp \lp \hbar \textnormal{ad}( k\pd_{\ell}+\pd_\epsilon)\rp = \alpha\).
\end{proof}

\subsubsection{qKZ equation from Segal-Sugawara}\label{subsec:QKZ}

We can make $M_k(\{W_i\}_{i \in I}, \mathbf z)$ into an infinite-dimensional vector bundle over $\mathbf{z} \in \textnormal{conf}_I(X)$. Moreover, using the identification $\wh W_i(z_i)=W_i\otimes Y_\hbar^*(\fd)$, we can trivialize this vector bundle. Let $I_{\mathbf z, \mathbf{z}'}\colon M_k(\{W_i\}_{i \in I}, \mathbf z) \to M_k(\{W_i\}_{i \in I}, \mathbf z')$ be the identification for another \(\mathbf{z}' \in \textnormal{conf}_I(X)\). The same data induces a trivialization of $B_k(\{W_i\}_{i \in I}, \mathbf z)$, whose fibers are $\bigotimes_{i \in I} W_i^*$. 

Define difference connections on $M_k(\{W_i\}_{i \in I}, \mathbf z)$ by
\be
A_i(\mathbf{z})=I_{\mathbf z, \mathbf z+k\hbar \mathbf{e}_i} Q_i,
\ee
where $Q_i$ is the quantum Segal-Sugawara operator \eqref{eq:quantum_segal} acting on the $i$-th tensor factor. These clearly form a compatible system of difference equations, since $Q_i$ commute with each other when acting on tensor factors. 

Assume now that \(k\ell + 1 = 0\). Then \(\alpha = \exp(k\hbar \partial_t)\) and we can see that
\begin{equation}
    A_i(\mathbf{z}) \colon M_k(\{W_i\}_{i \in I},\mathbf{z}) \to M_k(\{W_i\}_{i\in I},\mathbf{z} + k\hbar \mathbf{e}_i).    
\end{equation}
Furthermore, \(\partial_t\) quantizes to a Hopf derivation of \(\wh{DY}_\hbar(\fd)\) and therefore defines a Hopf derivation on \(\widehat{Y}^*_{\hbar,\ell}(\fd)^{\textnormal{co-op}}\). According to Lemma \ref{lem:Q_and_alpha}, we then have
\begin{equation}\label{eq:quamtum_Segal_Sugawara_and_derivation}
    A_i(\mathbf{z})a_j(t) = a_j(t_j + \delta_{ij}k\hbar) A_i(\mathbf{z})\,,\qquad a \in \widehat{Y}^*_{\hbar,\ell}(\fd),
\end{equation}
for \(a_j \in  \widehat{Y}^*_{\hbar,\ell}(\fd)_{z_j}^{\textnormal{co-op}} \subseteq  \widehat{Y}^*_{\hbar,\ell}(\fd)_{\mathbf{z}}^{\textnormal{co-op}}\).
This can be viewed as
the quantization of the classical identity
\begin{equation}\label{eq:classical_Segal_Sugawara_and_derivation}
    [\partial_{t_i} + L_i, a_j] =\delta_{ij} (\partial_t a)_j \,,\qquad a \in \widehat{\fd}_{<0,\ell},
\end{equation}
for elements \(a_j \in \widehat{\fd}_{<0,\ell,z_j} \subseteq \widehat{\fd}_{<0,\ell,\mathbf{z}}\) and where the classical Segal-Sugawara operator \(L_i\) acts on the \(i\)-th tensor factor. The latter formula \eqref{eq:classical_Segal_Sugawara_and_derivation} can be derived similarly to \cite[Lemma 13.2.3.]{frenkel2004vertex}.

\begin{Rem}
    Let us discuss why \(A_i(\mathbf{z})\) does not preserve \(\widehat{Y}_{\hbar,\ell}^*(\fd)\) for arbitrary pairs \((k,\ell)\). By virtue of Lemma \ref{lem:Q_and_alpha}, we have \(\textnormal{Ad}(Q_i) = \exp(\hbar\textnormal{ad}(k\partial_\ell + \partial_\epsilon))^{(i)}\). We can already see the issue in the semi-classical limit, i.e.\ the \(\hbar\)-coefficient. The \(\hbar\)-coefficient of \(A_i(\mathbf{z})\) is equal to \(k\partial_{z_i} + L_{-1}^{(i)}\), where \([L_{-1},a] = -(k\partial_\ell + \partial_\epsilon)a\) for \(a \in \widehat{\fd}\). 
    
    For every \(\frac{a}{(t-z_i)^n} \in \widehat{\fd}_{<0,z_i}\) the associated element in \(\widehat{\fd}_{<0,\mathbf{z}}\) has the form
    \begin{equation}
        \frac{a}{(t-z_i)^n} = a_{-n}^{(i)} - \frac{\partial_{z_i}^{n-1}}{(n-1)!}\sum_{j \neq i}\sum_{m \ge 0}\frac{a_m^{(j)}}{(z_i-z_j)^{m+1}}.
    \end{equation}
    Now applying \(\partial_{z_j} + L_{-1}^{(j)}\) for \(j \neq i\) yields
    \begin{equation}
        \begin{split}
            - \frac{\partial_{z_i}^{n-1}}{(n-1)!}\sum_{j \neq i}\sum_{m \ge 0}\frac{(-(k\partial_\ell + \partial_\epsilon) + k\partial_t) a_m)^{(j)}}{(z_i-z_j)^{m+1}}.
        \end{split}
    \end{equation}
    This cannot come from an element of \(\widehat{\fd}_{<0,\ell,\mathbf{z}}\) unless it vanishes, which happens precisely for \(k\ell + 1 = 0\), so \(\widehat{\fd}_{<0,\ell,\mathbf{z}}\) is only normalized by \(A_j(\mathbf{z})\) in this case.
\end{Rem}

\begin{Rem}\label{Rem:anomaly}

    Note that $k\ell+1=0$ implies that the bilinear form that specifies the vacuum VOA is of the form $k(\kappa_0+\ell\kappa_\fg)=k\kappa_0-\kappa_\fg$. The term $-\kappa_\fg$ is precisely the boundary anomaly found in \cite{dimofte2018dual} or \cite[Appendix A]{BCDN23}. Classically, this level ensures that the conformal element is the Casimir element corresponding to $\kappa_0$. This phenomenon persists at the quantum level. 
    
\end{Rem}

Crucially, equation \eqref{eq:quamtum_Segal_Sugawara_and_derivation} implies that $Q_i$ descends to the quotient by $Y_{\hbar,\ell}^*(\fd)_{\mathbf{z}}^{\textnormal{co-op}}$. Dualizing its action on the quotient, we obtain a well-defined operator 
\begin{equation}
    A_i^*(\mathbf{z}) \colon B_k(\{W_i\}_{i \in I},\mathbf{z} - \hbar k \mathbf{e}_i) \to B_k(\{W_i\}_{i \in I},\mathbf{z}).    
\end{equation}
Similar to \cite{EK4}, we can now deduce the following formula.

\begin{Thm}\label{Thm:qKZ}
    The quantum KZ connection is given by
    \be
        \begin{split}
            &\nabla_i(\mathbf{z})\coloneqq (A_i^*(\mathbf{z}+k\hbar\mathbf{e}_i))^{-1} 
            \\&= R^{i-1, i}(z_{i-1}-z_i+k\hbar) \cdots R^{1,i}(z_1-z_i+k\hbar)\cdot R^{n,i} (z_n-z_i)\cdots R^{i+1, i}(z_{i+1}-z_i).
        \end{split}
      \ee
\end{Thm}

\subsection{Examples from 3d gauge theories}\label{subsec:examples}

So far, we have focused on ordinary Lie algebras $\fg$ and $\fd=T^*\fg$. However, by following the Koszul sign rule carefully, one can easily extend all results to a DG Lie algebra $\fh$. Let us consider now $\fh=\fg\ltimes V[-1]$, where $\fg$ is an ordinary Lie algebra and $V$ is a finite-dimensional $\fg$-module, placed in cohomological degree $1$. The cotangent $\fd=T^*\fh$ is closely related to symplectic reduction $T^*V/\!/\!/\!/ G$, and, in fact, is identified with its tangent Lie algebra \cite{NiuQGSR}. This Lie algebra appears in \cite{costello2019vertex}, and its affine Lie algebra is the perturbative boundary VOA of the B-twist of the 3d $\CN=4$ gauge theory defined by $(G, T^*V)$. We therefore construct, for every $\fh$, a Yangian superalgebra $Y_\hbar (\fd)$ and a quantum vertex algebra $\CV_{\hbar,k}(\fd)$ that quantizes the perturbative boundary VOA $V_k(\fd)$. We now consider two specific examples of $\fd$. 

\subsubsection{$\fd=T^*\fg$ for a reductive Lie algebra $\fg$}

Let $\kappa=\kappa_0+k\kappa_\fg$, where as always \(\kappa_\fg\) is the Killing form of \(\fg\) and \(\kappa_0\) is the pairing between $\fg$ and $\fg^*$. In this case, the Lie algebra $V_1(\fd)$ contains a Lie subalgebra $V_{k}(\fg)$ and a commutative subalgebra $\CO(t\fg\lbb t\rbb)$, the algebra of functions on $t\fg\lbb t\rbb$. The OPE of these two subalgebras is given by
\be
x(z) f(w)\sim \frac{\kappa_0(x, f)}{(z-w)^2} +\frac{[x, f](w)}{z-w}\,,\qquad x\in \fg, f\in \fg^*.
\ee
Let $D_{G, k}$ be the algebra of chiral differential operators of $G$ at level $k$ \cite{arkhipov2002differential, gorbounov2001chiral}. It is generated by $V_{k} (\fg)$ and $\CO (G\lbb t\rbb)$ and satisfy the OPE
\be\label{eq:DGOPE}
x(z) f(w)\sim \frac{x(f)(w)}{z-w}\,, \qquad x\in \fg, f\in \CO (G). 
\ee
Here, $x$ acts on \(f\) via the left-invariant vector field $x_L$ on \(G\) defined by \(x\). It is known \cite{gorbounov2001chiral} that the degree $0$ part of $D_{G, k}$ is $\CO_G$, and the degree $1$ part of $D_{G, k}$ is isomorphic to $\mathrm{Der}(G)\oplus \Omega^1(G)$ as a $\CO_G$-module. It is of course generated by these two parts. It is known to experts, and at least implicitly contained in \cite{gorbounov2001chiral}, that $V_1(\fd)$ embeds into $D_{G, k}$. However, since we can't find an explicit proof, we include a proof here for completeness. 

\begin{Lem}

    Let $x\in \fg$ and $\omega\in \Omega^1(G)$, then
    \be
x(z)\omega (w)\simeq \frac{(x_L, \omega)(w)}{(z-w)^2}+\frac{x_L\rhd \omega (w)}{z-w},
    \ee
    where $x_L\rhd \omega$ is the action of the left-invariant vector field $x$ on $\Omega^1(G)$.
\end{Lem}

\begin{proof}
    Let $\omega=f\pd (g)$ for some $f, g\in \CO(G)$, then
    \be
    \begin{aligned}
        x(z) f(w)\pd g(w)&\sim\frac{x_L(f)(w)\pd g(w)}{z-w}+f(w) \pd \frac{x_L(g)(w)}{z-w}\\&=\frac{f(w)x_L(g)(w)}{(z-w)^2}+\frac{x_L(f)(w)\pd g(w)+f(w)\pd (x_L(g)(w))}{z-w}\\ &=\frac{(x, \omega)(w)}{(z-w)^2}+\frac{x\rhd \omega(w)}{z-w}.
    \end{aligned}
    \ee
    This completes the proof. 
    
\end{proof}
    
This leads to the following corollary. 

\begin{Cor}

    Let $\fg^*\subseteq \Omega^1(G)$ be the subset of right-invariant 1-forms. Then the subalgebra of $D_{G, k}$ generated by $\fg$ and $\fg^*$ is isomorphic to $V_1(\fd)$. 
    
\end{Cor}

Consequently, we construct a quantization of a large vertex subalgebra of $D_{G, k}$. In fact, $D_{G, k}$ can be written as an extension of $V_1(\fd)$.\footnote{We thank S. Nakatsuka for explaining to us about the relation between $V(\fd)$ and $D_{G, k}$.}

\begin{Prop}
    There is an isomorphism of VOAs:
    \be\label{eq:indOG}
\Ind_{\fd(\CO)\oplus \C c}^{\wh{\fd(\CK)}}\lp \CO_G\rp\cong D_{G, k},
    \ee
    where $\CO_G$ is the ring of functions on $G$, viewed as a module of $U(\fd(\CO))$ via the algebra map $U(\fd(\CO))\to U(\fg)$.
\end{Prop}

\begin{proof}
    Equation \eqref{eq:DGOPE} implies that there is a VOA morphism from the LHS of equation \eqref{eq:indOG} to the RHS. On the other hand, $D_{G, k}$ is, as a vector space, isomorphic to
    \be
\mathrm{Sym}(t^{-1}\fg[t^{-1}])\otimes \CO_{G(\CO)}=\mathrm{Sym}(t^{-1}\fg[t^{-1}])\otimes \mathrm{Sym}(t^{-1}\fg^*[t^{-1}])\otimes \CO_G,
    \ee
    since $G(\CO)=G\ltimes t\fg(\CO)$. Therefore, the associated graded of the above morphism is an isomorphism. This completes the proof. 
    
\end{proof}

Consequently, $D_{G, k}$ can be written as an extension of $V_1(\fd)$ of the form
\be
D_{G, k}=\bigoplus_{V\in \mathrm{Irrd}(G)} M_{V}\otimes V^*,
\ee
where $M_{V}=\Ind_{\fd(\CO)\oplus \C c}^{\wh{\fd(\CK)}}V$ is the Weyl module corresponding to the $\fd$-module $V$. We can now define a similar module of $\CV_1(\fd)$:
\be
\CD_{G, k}\coloneqq\bigoplus_{V\in\mathrm{Irrd}(G) }\CM_V\otimes V^*=\Ind_{Y_\hbar (\fd)\oplus \C c}^{\wh{DY}_{\hbar, \ell}}(\CO_G),
\ee
where $\CM_V$ is the induced module $\CM_V\coloneqq\Ind_{Y_\hbar (\fd)\oplus \C c}^{\wh{DY}_{\hbar, \ell}(\fd)}V$. We present the following conjecture.

\begin{Conj}

    The $\wh{DY}_{\hbar, \ell}(\fd)$-module $\CD_{G, k}$ has the structure of a quantum vertex algebra, extending the quantum vertex algebra structure of $\CV_1(\fd)$. 
    
\end{Conj}

\subsubsection{Affine $\fgl (1|1)$ and the $\beta\gamma$ VOA}

As another example, let us consider $\fh=\fg\ltimes V[-1]$, where $\fg$ is the Lie algebra of $(\C^\times)^r$ and $V=\C^n$ is a representation of $(\C^\times)^r$ defined by a charge matrix $\rho$. The associated vertex algebra $V_k(\fd)$ was studied in \cite{BCDN23}, and was shown to be related to BRST cohomology of $\beta\gamma$ VOAs. The simplest example is when $\fg=\C$ and $V=\C$ with weight $1$. In this case, $\fd$ is the smallest type-A Lie superalgebra $\fgl (1|1)$. Our construction therefore gives rise to a quantization $\CV_{\hbar,k}(\fgl (1|1))$ of $V_k (\fgl (1|1))$. For this example, we set $\ell=-1$ and $k=1$. 

The vertex algebra $V_1(\fgl (1|1))$ has simple current modules labeled by integers $(n,e)\in \Z^2$, which are defined as spectral-flows of the vacuum module
\be
V_{n, e}\coloneqq\sigma^{n,e}V_1(\fgl (1|1)), \qquad \sigma^{n,e}=\exp (nN+eE),
\ee
where \(E\) and \(N\) are the standard bosonic generators of \(\fgl (1|1)\) and $ \sigma^{n,e}$ is the spectral flow automorhism generated by $nN+eE$.  We have an identification of VOAs \cite{creutzig2013w}:
\be
\bg\otimes \FF\cong \bigoplus_n V_{n, 0}.
\ee
Here $\bg$ is the symplectic boson VOA and $\FF$ is the free fermion VOA. For the quantum algebra $\wh{DY}_{\hbar, \ell}(\fgl (1|1))$, the Kac-Moody symmetry corresponding to $\sigma^{n, m}$ remains, and we can consider the module
\be
\CV\coloneqq\bigoplus_n \sigma^{n,0} \CV_1(\fgl (1|1)). 
\ee
We present the following conjecture.

\begin{Conj}

    The $\wh{DY}_{\hbar, \ell}(\fgl (1|1))$-module $\CV$ has the structure of a quantum vertex algebra, extending $\CV_1(\fgl (1|1))$. It quantizes $\bg\otimes \FF$. 
    
\end{Conj}

Let us return to the general case of a representation $V=\C^n$ of $G=(\C^\times)^r$ defined by a charge matrix $\rho$. Assume that we have a short exact sequence
\be
\btik
0\rar & \Z^r\rar{\rho}& \Z^n\rar{(\rho^!)^T}\rar &\Z^{n-r}\rar & 0.
\etik
\ee
Then it is known that the 3d gauge theory corresponding to $\rho$ and $\rho^!$ are mirror dual to each other. Let $\fd_\rho = T^*(\fg \ltimes V[-1])$ be the cotangent Lie algebra corresponding to $\rho$. We again choose $\ell=-1$ and $k=1$. Then the VOA $V_1(\fd_\rho)$ admits spectral flows $\sigma^{\mu, \lambda}$ labeled by $(\mu,\lambda)\in \Z^{2r}$ and the non-perturbative VOA $V_{B, \rho}$ is defined by
\be
V_{B,\rho}\coloneqq \bigoplus_{\mu\in \Z^r} \sigma^{\mu, 0}V_1(\fd_\rho).
\ee
On the other hand, one can define $V_{A,\rho^!}$ as the relative BRST cohomology of $\bg^n\otimes \FF^n$ by the action of $\fgl (1)^{n-r}$ specified by $\rho^!$:
\be
V_{A,\rho^!}\coloneqq H_{\textnormal{BRST}, \textnormal{rel}}(\fgl (1)^{n-r}, \bg^n\otimes \FF^n).
\ee
We can think of $V_{A,\rho^!}$ as chiral differential operators on the symplectic reduction $T^*V/\!/\!/\!/ (\C^\times)^{n-r}$. One of the main results of \cite{BCDN23} is the following theorem.

\begin{Thm}[\cite{BCDN23} Theorem 4.4]

 $V_{A,\rho^!}\cong V_{B, \rho}$ as VOAs.
    
\end{Thm}

Again, we have constructed the quantum VOA $\CV_1(\fd_\rho)$. The spectral flow automorphisms $\sigma^{\mu, \nu}$ still exist for the quantum double. Let us defined the quantum extension
\be
\CV_{B,\rho}\coloneqq\bigoplus_{\mu\in \Z^r} \sigma^{\mu, 0}\CV_1(\fd_\rho).
\ee
We present the following conjecture.

\begin{Conj}

    The module $\CV_{B,\rho}$ admits the structure of a quantum VOA, extending $\CV_{\hbar,1}(\fd_\rho)$. It provides a quantization of chiral differetial operators on $T^*V/\!/\!/\!/ (\C^\times)^{n-r}$. 
    
\end{Conj}

\appendix

\section{Notations}

\begin{itemize}
    \item \(\fg\) is a finite-dimensional complex Lie algebra and \(\fd \coloneqq \fg \ltimes \fg^*\) is the cotangent Lie algebra, which the semi-direct product uses the coadjoint action of \(\fg\) on its dual space \(\fg^*\). We fix a basis \(\{I_\alpha\}_{\alpha = 1}^d \subseteq \fg\) and denote its dual basis by \(\{I^\alpha\}_{\alpha = 1}^d \subseteq \fg^*\).
    
    \item \(\CO \coloneqq \C[\![t]\!]\) and \(\CK \coloneqq \C(\!(t)\!) = \CO[t^{-1}]\) are the rings of Taylor power series and (lower-bounded) Laurent power series in a formal variable \(t\). We write \(\fg(\CO) = \fg \otimes \CO\) and \(\fg(\CK) = \fg \otimes \CK\) as well as \(\fd(\CO) = \fd \otimes \CO\) and \(\fd(\CK) = \fd \otimes \CK\). For \(a \in \fd\) and \(m \in \mathbb{Z}\), we write \(a_m = a\otimes t^m \in \fd(\CK)\). Tensor products involving objects over \(\CO\) and \(\CK\) are always understood in the completed way with respect to the \((t)\)-adic topologies of these spaces.

    \item \(\partial_t \colon \fd(\CK) \to \fd(\CK)\) is the derivation with respect to the formal variable \(t\), \(\partial_\epsilon = \psi\partial_t\) for the isomorphism \(\psi \colon \fg \to \fg^*\) determined by the Killing form \(\kappa_0 \colon \fg \times \fg^*\to \C\), and for every \(s \in \C\) we wrote \(\partial_s = \partial_t + s\partial_\epsilon\). 
    
    \item \(\kappa_0 \colon \fg \times \fg^*\to \C\) is the pairing of \(\fg\) with its dual space while \(\kappa_\fg \colon \fg \times \fg \to \C\) is the Killing form of \(\fg\). Both \(\kappa_0\) and \(\kappa_\fg\) are extended by \(0\) to bilinear forms \(\kappa_0,\kappa_\fg\colon \fd \times \fd \to \C\). For any \(\ell \in \C\), we write \(\kappa_\ell = \kappa_0 + \ell\kappa_\fg\). 

    \item For any \(\ell,s \in \C\), \(\widetilde{\fd}_{\ell,s} \coloneqq \C\partial_s\ltimes ( \fd(\CK) \oplus \C c_\ell)\) is the affine Kac-Moody algebra associated to \(\fd\) with respect to the bilinar form \(\kappa_\ell\) and derivation \(\partial_s\) and \(\widehat{\fd}_{\ell,s,<0} \coloneqq t^{-1}\fd[t^{-1}]\oplus \C c_\ell, \widetilde{\fd}_{\ell,s,\ge 0} = \C \partial_s \ltimes \fd(\CO) \subset \widetilde{\fd}_{\ell,s}\); see Section \ref{subsec:centralext}. Furthermore, we wrote \(\widehat{\fd}_{\ell,<0} \coloneqq \widehat{\fd}_{\ell,\ell,<0}, \widetilde{\fd}_{\ell,\ge 0}\coloneqq \widetilde{\fd}_{\ell,\ell,\ge 0} \subset \widetilde{\fd}_{\ell} \coloneqq \widetilde{\fd}_{\ell,\ell}\).

    \item For a Hopf algebra \(H\) over a ring \(R\) with coproduct \(\Delta\), we use Sweedler's notation \(\Delta(h) = \sum_{(h)}h^{(1)}\otimes_R h^{(2)}\). Furthermore, if \(H\) has a right Hopf action \(\lhd\) on an associative \(R\)-algebra \(A\), the semi-direct product \(H \ltimes_R A\), also called smash product and often denoted by \(H\#_R A\) (e.g.\ in our previous work \cite{ANyangian}), is the \(R\)-algebra with base space \(H \otimes_R A\) and multiplication 
    \begin{equation}
         (h \otimes_R 1) (1 \otimes_R a) = h\otimes_R a\quad \textnormal{ and }\quad (1 \otimes_R a)(h\otimes_R 1) = \sum_{(h)}h^{(1)} \otimes_R (a\lhd h^{(2)}) 
    \end{equation}
    for all \(a \in A,h \in H\). Similarly, if \(H\) has a left Hopf action \(\rhd\) on \(A\), we write the semi-direct product as \(A \rtimes_R H\).
    
    \item \(Y_\hbar(\fd)\) is the Yangian of \(\fd\), \(\Delta_\hbar^\gamma\) is its coproduct, and \(R(z)\) is its spectral \(R\)-matrix, all of which were introduced in \cite{ANyangian}; see also Section \ref{subsec:doublequotient} for a short summary. Furthermore, \(Y_\hbar^\circ(\fd)\subset Y_\hbar(\fd)\) is the Hopf subalgebra generated by \(\fd[t] \subset \fd(\CO)\).
    
    \item \(Y_\hbar^*(\fd)\) and \(DY_\hbar(\fd) = Y_\hbar(\fd) \otimes_{\C[\![\hbar]\!]} Y_\hbar^*(\fd)^{\textnormal{co-op}}\) are the dual and the double of the Yangian, introduced in Section \ref{subsec:dual+double}, and we write \(\Delta_\hbar^\gamma\) for the coproduct on the entire double \(DY_\hbar(\fd)\). The \(R\)-matrix of \(DY_\hbar(\fd)\) is denoted by \(R\). In particular, we distinguish it from the spectral \(R\)-matrix \(R(z)\) of \(Y_\hbar(\fd)\) by dropping the \(z\)-dependence, which is consistent with the fact that \(R(z)\) can be viewed as a power series expansion at \(z = \infty\) of the shift \(t \mapsto t + z\) of \(R\).
    
    \item For any \(\ell \in \C\), \(\widetilde{Y}_{\hbar,\ell}(\fd) = \C[\partial_\ell]\ltimes Y_\hbar(\fd)\) is the Yangian extended by the derivation \(\partial_\ell\), \(\widehat{Y}_{\hbar,\ell}^*(\fd) = Y_\hbar(\fd) \otimes \C[c_\ell]\) is the dual Yangian extended by the central element \(c_\ell\), \(\widehat{DY}_{\hbar,\ell}(\fd) = Y_\hbar(\fd) \otimes_{\C[\![\hbar]\!]} \widehat{Y}_{\hbar,\ell}^*(\fd)^{\textnormal{co-op}}\) is the centrally extended double Yangian, and \(\widetilde{DY}_{\hbar,\ell}(\fd) = \widetilde{Y}_{\hbar,\ell}(\fd) \otimes_{\C[\![\hbar]\!]} \widehat{Y}_{\hbar,\ell}^*(\fd)^{\textnormal{co-op}}\) is the full Kac-Moody double Yangian; see Section \ref{subsec:centralext}.

    \item We write \(X = \mathbb{P}^1\) for the Riemann sphere. For any ordered finite set \(\mathbf{z} = (z_i)_{i \in I} \subset X \coloneqq \mathbb{P}^1\), where \(I\) is a finite indexing set with \(|I|\in \N\) elements, we denote by \(Y_\hbar(\fd)_{\mathbf{z}},Y_\hbar^*(\fd)_{\mathbf{z}}\), and \(DY_\hbar(\fd)_{\mathbf{z}}\) the factorization versions of the Yangian, its dual, and its double, from Section \ref{subsec:EKfac}. 
    Similarly, we write \(\widetilde{Y}_{\hbar,\ell}(\fd)_{\mathbf{z}},\widehat{Y}_{\hbar,\ell}(\fd)_{\mathbf{z}}\), and \(\widetilde{DY}_{\hbar,\ell}(\fd)_{\mathbf{z}}\) for the factorized versions of the extended versions. 
\end{itemize}

\bibliographystyle{ytamsalpha} 

\bibliography{Doubleyangian}

\providecommand{\bysame}{\leavevmode\hbox to3em{\hrulefill}\thinspace}
\providecommand{\MR}{\relax\ifhmode\unskip\space\fi MR }
\providecommand{\MRhref}[2]{%
  \href{http://www.ams.org/mathscinet-getitem?mr=#1}{#2}
}
\providecommand{\href}[2]{#2}
\providecommand{\doihref}[2]{\href{#1}{#2}}
\providecommand{\arxivfont}{\tt}
\begin{thebibliography}{EGNO16}

\bibitem[Abe25]{abedin2024r}
R.~Abedin, \emph{{The $ r $-matrix structure of Hitchin systems via loop group
  uniformization}}, Ann. Henri Poincar\'e (2025) .

\bibitem[AG02]{arkhipov2002differential}
S.~Arkhipov and D.~Gaitsgory, \emph{{Differential operators on the loop group
  via chiral algebras}}, International Mathematics Research Notices
  \textbf{2002} (2002) 165--210.

\bibitem[AN24a]{abedin2024quantum}
R.~Abedin and W.~Niu, \emph{{Quantum groupoids from moduli spaces of $ G
  $-bundles}}, arXiv preprint arXiv:2411.05068 (2024) .

\bibitem[AN24b]{ANyangian}
R.~Abedin and W.~Niu, \emph{{Yangian for cotangent Lie algebras and spectral
  $R$-matrices}}, 2024. \href{http://arxiv.org/abs/2405.19906}{{\arxivfont
  arXiv:2405.19906 [math.QA]}}.

\bibitem[BCDN23]{BCDN23}
A.~Ballin, T.~Creutzig, T.~Dimofte, and W.~Niu, \emph{3d mirror symmetry of
  braided tensor categories}, 2023.
  \href{http://arxiv.org/abs/2304.11001}{{\arxivfont arXiv:2304.11001
  [hep-th]}}. \url{https://arxiv.org/abs/2304.11001}.

\bibitem[BD04]{BDchiral}
A.~Beilinson and V.~G. Drinfeld, \emph{{Chiral algebras}}, vol.~51, American
  Mathematical Soc., 2004.

\bibitem[BD25]{beilinson2025chiral}
A.~Beilinson and V.~Drinfeld, \emph{Chiral algebras}, vol.~51, American
  Mathematical Society, 2025.

\bibitem[BZFN10]{BFNloop}
D.~Ben-Zvi, J.~Francis, and D.~Nadler, \emph{Integral transforms and drinfeld
  centers in derived algebraic geometry}, Journal of the American Mathematical
  Society \textbf{23} (2010) 909--966.

\bibitem[CG19]{costello2019vertex}
K.~Costello and D.~Gaiotto, \emph{{Vertex Operator Algebras and 3d $\CN= 4$
  gauge theories}}, Journal of High Energy Physics \textbf{2019} (2019) 1--39.

\bibitem[CR13]{creutzig2013w}
T.~Creutzig and D.~Ridout, \emph{{W-algebras extending $\wh{\fgl (1|1)}$}}, Lie
  Theory and Its Applications in Physics: IX International Workshop, Springer,
  2013, pp.~349--367.

\bibitem[DGP18]{dimofte2018dual}
T.~Dimofte, D.~Gaiotto, and N.~M. Paquette, \emph{{Dual boundary conditions in
  3d SCFT’s}}, Journal of High Energy Physics \textbf{2018} (2018) 1--101.

\bibitem[Dri86]{drinfeld1986quantum}
V.~Drinfeld, \emph{{Quantum groups}}, Zapiski Nauchnykh Seminarov POMI
  \textbf{155} (1986) 18--49.

\bibitem[Dri90]{Drinfeld_almost_cocommutative}
\bysame, \emph{Almost cocommutative hopf algebras}, Leningrad Math. J.
  \textbf{1} (1990) 321--342.

\bibitem[EGNO16]{etingof2016tensor}
P.~Etingof, S.~Gelaki, D.~Nikshych, and V.~Ostrik, \emph{{Tensor categories}},
  vol. 205, American Mathematical Soc., 2016.

\bibitem[EK96]{EK3}
P.~Etingof and D.~Kazhdan, \emph{{Quantization of Lie bialgebras, III}}, 1998.
  \href{http://arxiv.org/abs/q-alg/9610030}{{\arxivfont arXiv:q-alg/9610030}}.
  \url{https://arxiv.org/abs/q-alg/9610030}.

\bibitem[EK98a]{EK4}
\bysame, \emph{{Quantization of Lie bialgebras, IV}}, 1998.
  \href{http://arxiv.org/abs/math/9801043}{{\arxivfont arXiv:math/9801043
  [math.QA]}}. \url{https://arxiv.org/abs/math/9801043}.

\bibitem[EK98b]{EK5}
\bysame, \emph{{Quantization of Lie bialgebras, V}}, 1998.
  \href{http://arxiv.org/abs/math/9808121}{{\arxivfont arXiv:math/9808121
  [math.QA]}}. \url{https://arxiv.org/abs/math/9808121}.

\bibitem[ESW22]{elliott2022taxonomy}
C.~Elliott, P.~Safronov, and B.~R. Williams, \emph{{A taxonomy of twists of
  supersymmetric Yang--Mills theory}}, Selecta Mathematica \textbf{28} (2022)
  73.

\bibitem[FBZ04]{frenkel2004vertex}
E.~Frenkel and D.~Ben-Zvi, \emph{{Vertex algebras and algebraic curves}},
  no.~88, American Mathematical Soc., 2004.

\bibitem[Fel98]{felder_kzb}
G.~Felder, \emph{{The KZB equations on Riemann surfaces}}, Sym\'etries
  quantiques (Les Houches, 1995), North-Holland, Amsterdam (1998) 687--725.

\bibitem[Gar23]{garner2023vertex}
N.~Garner, \emph{{Vertex operator algebras and topologically twisted
  Chern-Simons-matter theories}}, Journal of High Energy Physics \textbf{2023}
  (2023) 1--47.

\bibitem[GMS01]{gorbounov2001chiral}
V.~Gorbounov, F.~Malikov, and V.~Schechtman, \emph{{On chiral differential
  operators over homogeneous spaces}}, International Journal of Mathematics and
  Mathematical Sciences \textbf{26} (2001) 83--106.

\bibitem[Kas95]{Kassel1995quantum}
C.~Kassel, \emph{{Quantum Groups}}, Springer, 1995.

\bibitem[Niu24]{NiuQGSR}
W.~Niu, \emph{{Quantum groups and symplectic reductions}}, 2024.
  \href{http://arxiv.org/abs/2411.04195}{{\arxivfont arXiv:2411.04195
  [math.RT]}}.

\bibitem[RSTS90]{reshetikhin_tianshansky}
N.~Reshetikhin and M.~Semenov-Tian-Shansky, \emph{{Central extensions of
  quantum current groups}}, Lett. Math. Phys. \textbf{19} (1990) 133--142.

\end{thebibliography}

\end{document}